\date{\today}
\documentclass[10pt,reqno]{amsart}
\usepackage{amsthm,amsmath,verbatim,lscape,color,tensor,url,todonotes}
\usepackage{stmaryrd}
\usepackage{MnSymbol}

\usepackage[margin=1.5in]{geometry}
\usepackage[enableskew]{youngtab}
\usepackage[all]{xy}
\usepackage{enumerate}
\usepackage{paralist}
\usepackage{mathrsfs}
\usepackage{caption}

\usepackage{tikz}
\usetikzlibrary{matrix,arrows,decorations.pathmorphing,
  cd,patterns,calc}

\def\biblio{\bibliography{bibliography}\bibliographystyle{alpha}}

\usepackage{xr-hyper}
\usepackage[pagebackref]{hyperref}

\usepackage{subfiles}

\usepackage{graphicx}

\hypersetup{
  bookmarksnumbered=true,%
  colorlinks=true,%
  linkcolor=blue,%
  citecolor=blue,%
  filecolor=blue,%
  menucolor=blue,%
  urlcolor=blue,%
  pdfnewwindow=true,%
  pdfstartview=FitBH}
\pdfcatalog{/PageLayout /SinglePage}

\usepackage[nameinlink,capitalise,noabbrev]{cleveref}

\newtheoremstyle{named}%
    {}{}{\itshape}{}{\bfseries}{.}{.5em}{\thmnote{#3}}
\theoremstyle{named}

\crefname{equation}{}{}
\crefname{introcor}{Corollary}{Corollary}

\Crefname{figure}{Figure}{Figures}

\theoremstyle{theorem}
\newtheorem{thm}{Theorem}[section]
\newtheorem{theorem}[thm]{Theorem}
\newtheorem*{thm*}{Theorem}
\newtheorem{cor}[thm]{Corollary}
\newtheorem{corollary}[thm]{Corollary}
\newtheorem{prop}[thm]{Proposition}
\newtheorem{proposition}[thm]{Proposition}
\newtheorem{lemma}[thm]{Lemma}

\newtheorem*{introthm}{Theorem}
\newtheorem*{introprop}{Proposition}

\theoremstyle{definition}
\newtheorem{defn}[thm]{Definition}
\newtheorem{definition}[thm]{Definition}

\newtheorem{ex}[thm]{Example}
\newtheorem{example}[thm]{Example}

\newtheorem{rem}[thm]{Remark}
\newtheorem{remark}[thm]{Remark}
\newtheorem{notation}[thm]{Notation}
\newtheorem{notn}[thm]{Notation}

\DeclareMathOperator{\Aut}{Aut}

\DeclareMathOperator{\im}{im}

\DeclareMathOperator{\Res}{Res}

\DeclareMathOperator{\Tr}{Tr}

\DeclareMathOperator{\Stab}{Stab}
\DeclareMathOperator{\Po}{\mathbb{P}}

\newcommand{\C}{\mathbb{C}}

\newcommand{\Z}{\mathbb{Z}}

\newcommand{\bI}{\mathbb{I}}

\newcommand{\bP}{\mathbb{P}}
\newcommand{\bS}{\mathbb{S}}

\newcommand{\rH}{\mathrm{H}}

\newcommand{\arrowspace}{\hspace{-0.4ex}}
\newcommand{\bigarrowspace}{\hspace{-0.6ex}}
\newcommand{\biggerarrowspace}{\hspace{-0.8ex}}
\newcommand{\downwardarrow}[4]{\hbox{${#1}\arrowspace \downarrow_{#2}^{#3} \bigarrowspace ({#4})$}}
\newcommand{\upwardarrow}[4]{\hbox{${#1} \arrowspace \uparrow_{#2}^{#3} \biggerarrowspace ({#4})$}}
\newcommand{\justdownwardarrow}[3]{\hbox{${#1} \arrowspace \downarrow_{#2}^{#3}$}}
\newcommand{\justupwardarrow}[3]{\hbox{${#1} \arrowspace \uparrow_{#2}^{#3}$}}
\newcommand{\up}[1]{\hbox{$#1\arrowspace\uparrow$}}
\newcommand{\Rup}{\up{\ul{R}}}
\newcommand{\MackUp}[1]{\MackUpGroupPower{#1}{G}{m}}
\newcommand{\MackUpOrbit}[2]{\MackUpGroupPowerOrbit{#1}{G}{m}{#2}}
\newcommand{\MackUpPower}[2]{\MackUpGroupPower{#1}{G}{#2}}
\newcommand{\MackUpGroupPower}[3]{\hbox{$\ul{#1} \arrowspace \uparrow_{#2\times\Sigma_#3}^{#2}$}}
\newcommand{\MackUpGroupPowerOrbit}[4]{\hbox{$\ul{#1} \arrowspace \uparrow_{#2\times\Sigma_#3}^{#2} \hspace{-1ex} ({#4})$}}

\makeatletter
\let\c@equation\c@thm
\makeatother

\numberwithin{equation}{section}

\DeclareMathOperator{\Set}{Set}

\newcommand{\Frob}{F^m}

\newcommand{\ul}{\underline}
\newcommand{\rtarr}{\longrightarrow}
\newcommand{\xrtarr}{\xrightarrow}

\newcommand{\id}{\mathrm{id}}
\newcommand{\iso}{\cong}
\newcommand{\mf}[1]{\ul{#1}}
\newcommand{\mpi}{\mf{\pi}}
\newcommand{\inductioncolorname}{orange}
\newcommand{\inductioncolor}{\inductioncolorname}
\newcommand{\powercolor}{blue}
\newcommand{\choicecolor}[1]{\textcolor{blue}{#1}}
\newcommand{\restrict}{\arrow[bend right=35, swap, color=black]}

\newcommand{\induct}{\arrow[bend right=35, swap, color=\inductioncolor]}

\newcommand{\powerarr}{\mathrel{\color{\powercolor}\rtarr}}
\newcommand{\powerxarr}[1]{\mathrel{\color{\powercolor}\xrtarr{#1}}}
\newcommand{\indarr}{\mathrel{\color{\inductioncolor}\rtarr}}
\newcommand{\indxarr}[1]{\mathrel{\color{\inductioncolor}\xrtarr{#1}}}

\newcommand{\mylittlematrix}[1]
{\scalebox{0.4}{$\begin{pmatrix} #1 \end{pmatrix}$}}

\newcommand{\pdiv}{p\text{-}\mathrm{div}}
\newcommand{\pnodiv}{p\text{-}\mathrm{prime}}

\newcommand{\set}[1]{\left\{#1\right\}}%
\newcommand{\sets}[2]{\left\{ #1 \;|\; #2\right\}}%
\newcommand{\longto}{\longrightarrow}%
\newcommand{\into}{\hookrightarrow}%
\usepackage{harpoon}
\newcommand{\vect}[1]{\text{\overrightharp{\ensuremath{#1}}}}

\newcommand{\bigmod}[2]{{\raisebox{.2em}{$#1$}\left/ \raisebox{-.2em}{$#2$}\right.}}
\usepackage[final]{microtype}

\makeatletter
\newcommand*{\da@rightarrow}{\mathchar"0\hexnumber@\symAMSa 4B }
\newcommand*{\da@leftarrow}{\mathchar"0\hexnumber@\symAMSa 4C }
\newcommand*{\xdashrightarrow}[2][]{%
  \mathrel{%
    \mathpalette{\da@xarrow{#1}{#2}{}\da@rightarrow{\,}{}}{}%
  }%
}
\newcommand*{\da@xarrow}[7]{%
  \sbox0{$\ifx#7\scriptstyle\scriptscriptstyle\else\scriptstyle\fi#5#1#6\m@th$}%
  \sbox2{$\ifx#7\scriptstyle\scriptscriptstyle\else\scriptstyle\fi#5#2#6\m@th$}%
  \sbox4{$#7\dabar@\m@th$}%
  \dimen@=\wd0 %
  \ifdim\wd2 >\dimen@
    \dimen@=\wd2 %
  \fi
  \count@=2 %
  \def\da@bars{\dabar@\dabar@}%
  \@whiledim\count@\wd4<\dimen@\do{%
    \advance\count@\@ne
    \expandafter\def\expandafter\da@bars\expandafter{%
      \da@bars
      \dabar@ 
    }%
  }%
  \mathrel{#3}%
  \mathrel{%
    \mathop{\da@bars}\limits
    \ifx\\#1\\%
    \else
      _{\copy0}%
    \fi
    \ifx\\#2\\%
    \else
      ^{\copy2}%
    \fi
  }%
  \mathrel{#4}%
}
\makeatother

\date{\today}

\begin{document}
\title
{Additive power operations in equivariant cohomology}
\author[P. J. Bonventre]{Peter \ul{J}. Bonventre}
\address{Department of Mathematics\\ University of Kentucky\\
Lexington, KY 40506, USA}
\email{peterbonventre@uky.edu}
\author[B. J. Guillou]{Bertrand \ul{J}. Guillou}
\email{bertguillou@uky.edu}
\author[N. J. Stapleton]{Nathaniel \ul{J}. Stapleton}
\email{nat.j.stapleton@uky.edu}
\thanks{
Guillou was supported by NSF grant DMS-1710379.
Stapleton was supported by NSF grant DMS-1906236.
}



\begin{abstract} 
Let $G$ be a finite group and  $E$ be an $H_\infty$-ring $G$-spectrum. For any $G$-space $X$ and positive integer $m$, we give an explicit description of the smallest Mackey ideal $\underline{J}$ in $\ul{E}^0(X\times B\Sigma_m)$ for which the reduced $m$th power operation $\ul{E}^0(X) \rtarr \ul{E}^0(X \times B\Sigma_m )/\underline{J}$ is a map of Green functors. 
We obtain this result as a special case of a  general theorem that we establish in the context of $G\times\Sigma_m$-Green functors. 
This theorem also specializes to 
characterize the appropriate ideal $\ul{J}$ when $E$ is a $G_\infty$-ring in global spectra.
We give  example computations for the sphere spectrum, complex $K$-theory, and Morava $E$-theory.
\end{abstract}

\date{\today}


\maketitle

{\hypersetup{linkcolor=black}\tableofcontents}

\def\biblio{}


\section{Introduction}
\label{INTRO_SEC}

\newcommand{\bedit}[1]{{\color{orange}#1}}

Cohomology theories are one of the most effective tools in homotopy theory for the study of topological spaces.
This is especially true of cohomology theories with additional structure, such as a multiplication. 
Another important example of  additional structure is that of a {\it power operation}.
For $E$ a multiplicative cohomology theory and $X$ a space,  an $m$th power operation is a suitable factorization of the multiplicative endomorphism $x\mapsto x^m$ on $E^*(X)$ through $E^*(X\times B \Sigma_m)$,
where $\Sigma_m$ is the symmetric group. 
In many such cohomology theories $E$, this allows one to make use of the cohomology of symmetric groups to produce exotic operations on $E$.
Power operations in cohomology theories
have played an important role in both our theoretical and our computational understanding of essentially all naturally-occurring cohomology theories. 

Power operations on cohomology theories arise from $H_\infty$-ring structures 
on the representing spectra. Recall \cite[Definition~I.3.1]{BMMS} that an $H_\infty$-ring structure on $E$ consists of maps $E^{\wedge m}_{h\Sigma_m} \to E$ in the homotopy category, for all $m\geq 0$, 
satisfying certain compatibility relations. 
This is weaker than the notion of an $E_\infty$-ring, which is the homotopy-invariant version of commutative ring spectrum.

Given an $H_\infty$-ring structure on $E$ and a space $X$, the power operations on $E^0(X)$ arise as follows.
A cohomology class $\alpha\in E^0(X)$ corresponds to a map $\Sigma^\infty X_+ \to E$. Then the $m$th power operation on $\alpha$ is the composition
\begin{equation}
\label{eq:PowerOpComposition}
    \Sigma^\infty (X \times B\Sigma_m)_+ =
    \Sigma^\infty (X_{h\Sigma_m})_+ \xrightarrow{\Delta} \Sigma^\infty (X^{\times m}_{h\Sigma_m})_+ \xrightarrow{\alpha} E^{\wedge m}_{h\Sigma_m} \to E,
\end{equation}
where $\Delta$ is induced by the diagonal map of $X$ and the map on the right comes from the $H_\infty$ structure. In fact, an $H_\infty$-ring structure on $E$ is determined by a collection of power operations on $E$-cohomology \cite[Proposition~VIII.1.2]{BMMS}.

The most useful power operations, Steenrod operations {in ordinary cohomology} and Adams operations {in $K$-theory}, are both additive power operations,
though in general power operations are not additive.
The simplest example of this failure is exhibited in the binomial formula
\[
	(x+y)^m = x^m + y^m + \sum_{i=1}^{m-1} \binom{m}{i} x^i y^{m-i}.
\]
The operation $x\mapsto x^m$ only becomes additive once the right-most sum is set to 0.
More generally, additive
power operations in a cohomology theory $E$
are all built from the universal additive power operation
\begin{equation}
      \label{ADDITIVEPOWEROP_EQ}
P_m \colon E^0(X) \to E^0(X \times B\Sigma_m)/I_{\Tr},
\end{equation}
where 
$I_{\Tr}$ is a specific transfer ideal that can be defined naturally 
in ring spectra $E$ and spaces $X$.
Given the effectiveness of these additive power operations, it is desirable to understand their analogues in other contexts. In this paper we focus on the case of equivariant cohomology theories.
In particular we consider power operations in Borel equivariant, in Bredon or ``genuine'' equivariant, and in global equivariant cohomology theories.

For a finite group $G$
and a $G$-spectrum $E$, the coefficients of $E$ form a $G$-Mackey functor, 
 given by the formula
\[
G/H \mapsto \ul{E}^0(G/H) = [G/H_{+},E]^{G},
\] 
when $H \leq G$ is a subgroup.
More generally, given a $G$-space $X$, the $E$-cohomology of $X$ is a $G$-Mackey functor 
with values
\[
G/H \mapsto \ul{E}^0(X)(G/H) = E^0(G/H \times X) = [(G/H \times X)_{+},E]^{G}.
\]

Power operations for $G$-spectra will arise from $H_\infty$-ring structures on $G$-spectra, which once again will mean a collection of maps $E^{\wedge m}_{h\Sigma_m}\to E$ in the equivariant stable homotopy category (see \cref{GPOWER_OPS_SEC}).
If we  assume that $E$ is equipped with the structure of an $H_{\infty}$-ring in the category of genuine $G$-spectra, then the associated power operation is a map
\[
P_m \colon \ul{E}^0(G/H) \to \ul{E}^0(B\Sigma_m)(G/H),
\]
where $B\Sigma_m$ is a $G$-space with trivial action. In this case, the Mackey functors $\ul{E}^0$ and $\ul{E}^0(B\Sigma_m)$ are both $G$-Green functors, so both $\ul{E}^0(G/H)$ and $\ul{E}^0(B\Sigma_m)(G/H)$ are commutative rings. The map $P_m$ is multiplicative, but not additive, and it does not respect the induction maps in the $G$-Mackey functors. The additivity of $P_m$ reduces to a classical problem, which was solved for spectra in complete generality.  \cite[Proposition VIII.1.4(iv)]{BMMS} identifies an ideal $I_{\Tr} \subseteq \ul{E}^0(B\Sigma_m)(G/H)$, generated by the image of the transfer maps $\ul{E}^0(B\Sigma_i \times \Sigma_j)(G/H) \to \ul{E}^0(B\Sigma_m)(G/H)$ for $i,j>0$ with $i+j = m$,  with the property that the composite
\[
P_m/I_{\Tr} \colon \ul{E}^0(G/H) \to \ul{E}^0(B\Sigma_m)(G/H) \to \ul{E}^0(B\Sigma_m)(G/H)/I_{\Tr}
\]
is a map of commutative rings that respects the restriction maps in the $G$-Mackey functor structure (see \cref{ITR_MACKEY}). However, these maps do not necessarily respect the induction maps in the $G$-Mackey functor structure. The goal of this paper is to identify and study the minimal Mackey ideal $\ul{J} \subseteq \ul{E}^0(B\Sigma_m)$ so that the composite
\[
P_m/\ul{J} \colon \ul{E}^0 \to \ul{E}^0(B\Sigma_m)  \to \ul{E}^0(B\Sigma_m)/\ul{J}
\]
is a map of Green functors. 

The ideal $\ul{J}(G/H) \subseteq \ul{E}^0(B\Sigma_p)(G/H)$ is built out of transfer maps. Given a surjective map of finite $G \times \Sigma_m$-sets $X \to Y$, applying homotopy orbits for the $\Sigma_m$-action gives a cover of $G$-spaces $X_{h \Sigma_m} \to Y_{h \Sigma_m}$. This gives rise to a transfer map in $E$-cohomology
\[
\Tr \colon E^0(X_{h \Sigma_m}) \to E^0(Y_{h \Sigma_m}).
\]
If $Y = G/H$, with trivial $\Sigma_m$-action, then the target of this transfer map is $E^0(G/H \times B\Sigma_m) = \ul{E}^0(B\Sigma_m)(G/H)$. 

In the case that $m = p$ is a prime, the ideal $\ul{J}(G/H) \subseteq \ul{E}^0(B\Sigma_p)(G/H)$ is defined to be the ideal generated by the transfer maps induced by the maps of $G \times \Sigma_m$-sets 
\begin{enumerate}[(i)]
\item \label{item1} $(G \times \Sigma_p)/(H \times \Sigma_i \times \Sigma_j) \to G/H$ for $i+j = p$ and $i,j>0$ and
\item \label{item2} $(G \times \Sigma_p)/\Gamma(a) \to G/H$, for all subgroups $S \leq H$ and homomorphisms $a \colon S \to \Sigma_p$ with image containing a $p$-cycle, where $\Gamma(a) \subseteq G \times \Sigma_p$ is the graph subgroup of $a$.
\end{enumerate} 

By construction,  $I_{\Tr}$ is contained in $\ul{J}(G/H)$, and $\ul{J}$ is natural in the cohomology theory $E$.
The following result is the special case of the main theorems of this paper when $E$ is an $H_\infty$-ring $G$-spectrum and when $m=p$.
It makes use of
 \cref{JG_ACTION_PRIME_NONBOREL_THM} and is a special case of \cref{JSUBMACKEY_LEM} and \cref{GEquiv_Main_Thm}.

\begin{introprop}
Assume that $E$ is an $H_{\infty}$-ring in genuine $G$ spectra. The ideals $\ul{J}(G/H)$, defined above, 
assemble to a Mackey ideal $\ul{J} \subseteq \ul{E}^0(B\Sigma_p)$,  minimal
with the property that the composite
\[
P_p/\ul{J} \colon \ul{E}^0 \to \ul{E}^0(B\Sigma_p)/\ul{J}
\]
is a map of $G$-Green functors.
\end{introprop}

In fact, this result holds much more generally. Let $\ul{R}$ be a $G \times \Sigma_m$-Green functor. From $\ul{R}$ we may form the induced $G$-Green functor $\justupwardarrow{\ul{R}}{G \times \Sigma_m}{G}$ given by the formula  
\[
\upwardarrow{\ul{R}}{G \times \Sigma_m}{G}{G/H} = \ul{R} \, \big( (G \times \Sigma_m)/(H \times \Sigma_m) \big).
\]
As $\ul{R}$ is a $G \times \Sigma_m$-Green functor, it may be viewed as a functor from the category of finite $G\times \Sigma_m$-sets to commutative rings that admits transfers along surjections. In this situation, we define $\ul{J}(G/H) \subseteq \upwardarrow{\ul{R}}{G \times \Sigma_m}{G}{G/H}$ to be the ideal generated by the images of certain transfer maps generalizing the maps in \eqref{item1} and \eqref{item2}. These maps are described explicitly in \cref{sec:Ideals} and make use of a group extension of $\Gamma(a)$ by a product of symmetric groups described in \cref{STAB_SEC}. 
The following result is \cref{JSUBMACKEY_LEM}.

\begin{introthm}
\hypertarget{IntroThm_2}
Let $\ul{R}$ be a $G \times \Sigma_m$-Green functor. The ideals $\ul{J}(G/H) \subseteq \upwardarrow{\ul{R}}{G \times \Sigma_m}{G}{G/H}$ assemble into a $G$-Mackey ideal $\ul{J} \subseteq \justupwardarrow{\ul{R}}{G \times \Sigma_m}{G}$.  
\end{introthm}

When $E$ is a homotopy commutative ring $G$-spectrum, we get a $G \times \Sigma_m$-Green functor $\ul{R}$ via the formula
\[
\ul{R}\big((G \times \Sigma_m)/\Lambda\big) = E^0\Big(\big((G \times \Sigma_m)/\Lambda\big)_{h\Sigma_m}\Big).
\]
We can restrict this to a $G$-Green functor via the formula
\[
    \justdownwardarrow{\ul{R}}{G}{G \times \Sigma_m}(G/H) = \ul{R}\big((G\times \Sigma_m)/(H\times e)\big).
\]
In this case, the restriction of $\ul{R}$ to $G$ satisfies 
$    \justdownwardarrow{\ul{R}}{G}{G \times \Sigma_m} = \ul{E}^0, $
the induced $G$-Green functor satisfies $\justupwardarrow{\ul{R}}{G \times \Sigma_m}{G} = \ul{E}^0(B\Sigma_m)$ and, when $m=p$, the Mackey ideals called $\ul{J}$ in the proposition and theorem above agree.

Global equivariant homotopy theory furnishes us with further examples of $G \times \Sigma_m$-Green functors. Recall that a global spectrum in the sense of \cite{global} is a spectrum with simultaneous and compatible actions of all finite groups. If $E$ is a homotopy commutative ring global spectrum, then for a finite group $G$, there is an underlying homotopy commutative genuine $G$-spectrum $E_G$. Thus, fixing $m \geq 0$, there is an associated $G \times \Sigma_m$-Green functor $\ul{R}$ given by
\[
\ul{R}((G \times \Sigma_m)/\Lambda) = \ul{E}_{G \times \Sigma_m}((G \times \Sigma_m)/\Lambda) = [((G \times \Sigma_m)/\Lambda)_{+},E_{G \times \Sigma_m}]^{G\times \Sigma_m}.
\]
The above 
theorem furnishes us with a Mackey ideal $\ul{J} \subseteq \justupwardarrow{\ul{R}}{G \times \Sigma_m}{G}$.

In case $E$ is either an $H_{\infty}$-ring in genuine $G$-spectra or a $G_\infty$-ring global spectrum (the global analogue of $H_{\infty}$), then the $m$th power operation is a map
\begin{equation} \label{eq:introeq}
P_m \colon \downwardarrow{\ul{R}}{G}{G \times \Sigma_m}{G/H} \to \upwardarrow{\ul{R}}{G \times \Sigma_m}{G}{G/H}.
\end{equation}
In \cref{GREEN_POWER_OPS_SEC}, we introduce the notion of an $m$th total power operation on a $G \times \Sigma_m$-Green functor that captures the $m$th power operation in each of these examples. Although both the source and target of $P_m$ are Green functors, the  operation $P_m$ is not a map of Green functors before passing to a quotient.
The
proposition above is then a special case of the following theorem (see \cref{GEquiv_Main_Thm,Global_Main_Thm}).

\begin{introthm} 
Let $E$ be an $H_{\infty}$-ring in genuine $G$-spectra or a $G_\infty$-ring global spectrum and let $\ul{R}$ be the associated $G\times\Sigma_m$-Green functor defined above.
The composite
\[
P_m/\ul{J} \colon  \justdownwardarrow{\ul{R}}{G}{G \times \Sigma_m} \rtarr \justupwardarrow{\ul{R}}{G \times \Sigma_m}{G} \rtarr \justupwardarrow{\ul{R}}{G \times \Sigma_m}{G} \! /\ul{J}
\]
is a map of $G$-Green functors.
\end{introthm}

The general case of a $G\times\Sigma_m$-Green functor with $m$th power operation is treated in \cref{Green_Main_Thm}.

Since Borel equivariant cohomology theories are examples of genuine equivariant cohomology theories, if $E$ is a nonequivariant $H_{\infty}$-ring spectrum, then the 
proposition above
may be applied to the Borel equivariant cohomology theory associated to $E$. However, in this setting, it is also natural to ask for the smallest ideal with the property that the transfer from a specific subgroup commutes with the power operation after taking the quotient by the ideal. If $H \leq G$, then $BH \to BG$ is equivalent to a finite cover of spaces. We would like the smallest ideal $J_{H}^{G} \subseteq E^0(BG \times B\Sigma_m)$ such that the following diagram commutes
\[
\xymatrix{E^0(BH) \ar[r] \ar[d]_{\Tr} & E^0(BH \times B\Sigma_m)/I_{\Tr} \ar[d]^{\Tr} \\ E^0(BG) \ar[r] & E^0(BG \times B\Sigma_m)/J_{H}^{G}.}
\]
Specializing the case of genuine $G \times \Sigma_m$-spectra to Borel equivariant $G \times \Sigma_m$-spectra, the $G\times \Sigma_m$-Green functor associated to $E$ is given by the formula
\[
\ul{R}((G \times \Sigma_m)/\Lambda) = E^0\Big(\big((G \times \Sigma_m)/\Lambda\big)_{hG \times \Sigma_m}\Big) \cong E^0(B\Lambda) 
\]
and 
$\upwardarrow{\ul{R}}{G \times \Sigma_m}{G}{G/H}  \iso
E^0(BH \times B\Sigma_m)$.
For the case where $H$ is a normal subgroup of $G$, in \cref{NORMAL_SEC} we explicitly describe a subset of the transfer maps that go into the construction of $\ul{J}(G/G) \subseteq \upwardarrow{\ul{R}}{G \times \Sigma_m}{G}{G/G}$ and show that $J_{H}^{G}$ is the sub-ideal generated by the image of this subset. As a consequence of this description of $J_{H}^{G}$, we learn that, when $H$ is normal, if $m$ and $|G/H|$ are relatively prime, then $J_{H}^{G} = I_{\Tr} \subseteq E^0(BG \times B\Sigma_m)$.  

Making use of the explicit descriptions of $J_{H}^{G}$ and $\ul{J}$ described above, we consider a number of examples coming from equivariant homotopy theory. In particular, we look at group cohomology, Burnside rings, representation rings, class functions, and the generalized class functions of Hopkins, Kuhn, and Ravenel. In each case, we analyze quirks of the particular case and calculate $\ul{J}$ in some small examples such as $m = 2,3$.

\begin{remark}
We note that there are many possible notions of equivariant commutative ring spectra (e.g. \cite{BH15}) (for a fixed group $G$). In the results above, we have chosen to work with a version for which the power operations discussed in this paper exist; namely, we consider $H_\infty$-rings in genuine $G$-spectra (\cref{HRING_DEF}). There are various stronger notions of commutative rings in $G$-spectra, which endow the associated cohomology theories with additional structure. We do not consider such additional structure here.
\end{remark}

\begin{remark}
It is natural to wonder if the presence of an $N_{\infty}$-ring structure could impact the results in this paper. Although an $N_{\infty}$-ring structure does give rise to norms (i.e. a Tambara functor structure on homotopy groups), the authors do not see how these norms relate to the power operations arising from the underlying $H_{\infty}$-ring structure considered here. On the other hand, a $G_{\infty}$-ring global spectrum gives rise to a global Tambara functor structure as well as a global power functor structure on the level of homotopy groups. These are expected to encode equivalent algebraic data, but this is not in the literature at this time \cite[Remark~5.1.7]{global}. A global power functor includes power operations as in \eqref{eq:introeq}.
\end{remark}

\subsection{Conventions}

\begin{itemize}
\item
By $G$, we will always mean a finite group.
\item By a {\it graph subgroup} $\Gamma \leq G \times \Sigma_m$, we will mean a subgroup such that $\Gamma \cap \Sigma_m = \{e\}$. Such a subgroup is the graph of a homomorphism $K \rtarr \Sigma_m$, where $K=\pi_G(\Gamma)$ and $\pi_G\colon G\times \Sigma_m \rtarr G$ is the projection. 
\item We will use the notation $\underline{n}=\{1,\dots,n\}$.
\item A $G$-spectrum will always mean in the ``genuine'' sense. In other words, our $G$-spectra are indexed over a complete $G$-universe.
\item We will use transfer maps in cohomology throughout, and in order to help orient the reader, we always display such transfer maps in \textcolor{\inductioncolor}{\inductioncolorname}.
\item In \cref{sec:Ideals}, we abbreviate an induced Mackey functor $\justupwardarrow{\ul{R}}{G \times \Sigma_m}{G} $ to $\Rup$.
\item In \cref{ExamplesSection}, we abbreviate an induced Mackey functor $\justupwardarrow{(\ul{A}_{G\times \Sigma_m})\hspace{-0.1ex}}{G\times \Sigma_m}{G}$ to
$\MackUp{A}$.
\end{itemize}

\subsection{Organization}

We begin \cref{GROUPTHEORY_SEC} by considering the Borel equivariant case.
Key results about the ideal $\ul{J}$ are given in \cref{ABS_SEC} and \cref{NORMAL_SEC}; these results are specialized to the case $m=p$ is prime in \cref{PRIME_SEC}. 
Our main results about power operations appear in \cref{sec:GreenPowerOps}.
We introduce the notion of an $m$th total power operation for a $G\times\Sigma_m$-Green functor in \cref{GREEN_POWER_OPS_SEC}.
One of the central results of the article, \cref{JSUBMACKEY_LEM}, is that $\ul{J}$ is a Mackey ideal.
 \cref{ExamplesSection} gives a number of examples. We consider the sphere spectrum, $KU$-theory, Eilenberg-Mac~Lane spectra, and height 2 Morava $E$-theory.

\subsection{Acknowledgments} It is a pleasure to thank Sune Precht Reeh, Tomer Schlank, and Michael Stahlhauer for helpful comments.


\tikzset{
  symbol/.style={
    draw=none,
    every to/.append style={
      edge node={node [sloped, allow upside down, auto=false]{$#1$}}}
  }
}

\section{The Borel equivariant case} 
\label{GROUPTHEORY_SEC}

The purpose of this section is to understand the relationship between
the additive power operations 
for an $H_\infty$-ring spectrum $E$, 
or rather its associated Borel equivariant theory,
and transfers along finite covers of the form $BH \to BG$ where $H \leq G$ is a subgroup.

Recall that an $H_{\infty}$-ring structure includes the data of maps of spectra
\[
E^{\wedge m}_{h\Sigma_m} \to E,
\]
for each $m \geq 0$, satisfying a collection of relations in the stable homotopy category. These maps give rise to power operations as in \cref{eq:PowerOpComposition}: given a homotopy class of maps $\Sigma^{\infty}_+ BG \to E$, we may form the composite
\[
\Sigma^{\infty}_+ (BG \times B\Sigma_m) \to 
(\Sigma^{\infty}_+ BG)^{\wedge m}_{h\Sigma_m} \to E^{\wedge m}_{h\Sigma_m} \to E,
\]
in which 
 the first map is induced by the diagonal $G \to G^{\times m}$, the second map is the $m$th symmetric power of the given map $\Sigma^{\infty}_+ BG \to E$, and the final map uses the $H_{\infty}$-ring structure. 
Alternatively, we can and will rewrite the homotopy orbit space $BG^{\times m}_{h\Sigma_m}$ as the equivalent classifying space $BG \wr \Sigma_m$, where $G\wr \Sigma_m$ is the wreath product. Then the first map is induced by the subgroup inclusion $G\times \Sigma_m \leq G \wr \Sigma_m$.
 The composite is an element in $E^0(BG \times B\Sigma_m)$, and this procedure gives the power operation
\[
P_m \colon E^0(BG) \to E^0(BG \times B\Sigma_m),
\]
which is a multiplicative map.

Further, given a subgroup $H \leq G$, there is an additive transfer map $E^0(BH) \indxarr{} E^0(BG)$. Let $I_{\Tr} \subseteq E^0(BG \times B\Sigma_m)$ be the image of the direct sum of the transfer maps along the inclusions $G \times \Sigma_i \times \Sigma_j \to G \times \Sigma_m$, where $i+j = m$ and $i,j > 0$. Frobenius reciprocity 
(see \cref{defn:GGreen})
implies that the subgroup $I_{\Tr}$ is an ideal.  \cite[Proposition VIII.1.4(iv)]{BMMS} implies that this ideal controls the failure of $P_m$ to be additive. Therefore the map 
\[
P_m/I_{\Tr} \colon E^0(BG) \to E^0(BG \times B\Sigma_m)/I_{\Tr}
\]
is a ring map.

It is natural to wonder if the transfers along $H 
\leq G$ and $H \times \Sigma_m \leq G \times \Sigma_m$ are compatible with the power operations $P_m/I_{\Tr}$ for $H$ and $G$. This is false. The failure to be compatible is a consequence of the appearance of the map $G \times \Sigma_m \to G \wr \Sigma_m$ induced by the diagonal $G \to G^{\times m}$ in the definition of $P_m$.

The goal of this section is to describe, as explicitly as possible,
the smallest ideal
\[
J_H^G \subset E^0(BG \times B\Sigma_m )
\]
containing $I_{\Tr}$
such that the diagram
\begin{equation}
      \label{BORELSQ_EQ}
      \begin{tikzcd}
            E^0(BH) \arrow[d, color=\inductioncolor] \arrow[r, "P_m/I_{\Tr}"]
            &
            E^0(BH \times B \Sigma_m )/I_{\Tr} \arrow[d, color=\inductioncolor]
            \\
            E^0(BG) \arrow[r, "P_m/J_H^G"]
            &
            E^0(BG \times B\Sigma_m ) / J^G_H
      \end{tikzcd}
\end{equation}
commutes
and the horizontal maps are additive, so that in particular this is a commuting square of abelian groups. 

We will also describe the absolute ideal $J^G \subseteq E^0(BG \times B\Sigma_m)$, which is the sum of the ideals $J_H^G$ as $H$ varies. 
In terms of the notation from \cref{INTRO_SEC}, $J^G$ is what was denoted there as $\underline{J}(G/G)$.
This provides the smallest ideal such that the reduced power operation $P_m/J^G$ commutes with transfers $BH \to BG$ for {\it all} subgroups $H \leq G$.

In \cref{OVERVIEW_SEC}, we outline this story.
In the remaining subsections, we describe explicitly the ideals $J_H^G$ and $J^G$ in various cases.
First, in \cref{STAB_SEC}, we study the stabilizers for the $G\times \Sigma_m$-action on $(G/H)^{\times m}$ for $H \leq G$ and $m \geq 0$.
In \cref{ABS_SEC}, we provide an explicit description of $J^G$.
In \cref{NORMAL_SEC}, we provide an explicit description of $J_H^G$ in the case where $H \trianglelefteq G$ is a normal subgroup.
In \cref{PRIME_SEC}, we provide a more concrete description in the case that $m$ is prime.
Finally, in \cref{RELPRIME_SEC}, we consider the simpler case where $m$ and $|G|$ are relatively prime.

\subsection{Overview}
\label{OVERVIEW_SEC}

We have two tasks: first, to ensure that the power operation is additive, and second, to ensure it commutes with transfer maps.
As we saw in \cref{INTRO_SEC}, in order for the power operation to be additive and thus a map of commutative rings,
one must quotient by the ideal generated by transfers along the proper partition subgroups
$G\times \Sigma_i \times \Sigma_j < G\times \Sigma_m $;
that is, we must have $I_{\Tr} \subseteq J_H^G$.
However, this ideal is not necessarily sufficient to make diagram \eqref{BORELSQ_EQ} commute,
as we demonstrate  in \cref{ExamplesSection}.
The problem boils down to the relationship between the transfer along the inclusion 
$H \wr \Sigma_m \leq G \wr \Sigma_m$ and the diagonal map $G\times \Sigma_m \to G \wr \Sigma_m$.
We are particularly interested in studying the power operation that lands in the product because
of the relationship to power operations for genuine equivariant cohomology theories.

In \cite{BMMS}, the power operation $P_m \colon E^0(BH) \to E^0(BH \times B \Sigma_m)$ is defined as a composite
\[
      P_m \colon E^0(BH) \xrtarr{\bP_m} E^0(BH \wr \Sigma_m) \xrightarrow{\Delta^{\**}} E^0(BH \times B \Sigma_m)
\]
where $\bP_m$ is the \textit{total power operation}.
The operation $\bP_m$
is functorial on all stable maps, and thus every subgroup $H \subseteq G$ gives rise to a commutative diagram
\begin{equation}
      \label{TPO_EQ}
      \begin{tikzcd}
      E^0(BH) \ar[r,"\Po_m"] \ar[d,color=\inductioncolor,"\Tr"] & E^0(BH \wr \Sigma_m) \ar[d,color=\inductioncolor,"\Tr"] \\ E^0(BG) \ar[r,"\Po_m"] & E^0(BG \wr \Sigma_m).
      \end{tikzcd}
\end{equation}
After composing the bottom arrow with the map in $E$-cohomology induced by the diagonal
$BG \times B\Sigma_m  \to BG \wr \Sigma_m$,
we may extend the diagram above by considering a homotopy pullback. We may do this by making use of the fact (see \cite[Chapter 4]{Adams}, for instance) that, given a homotopy pullback of spaces 
\[
\xymatrix{Y  \ar[d] & H \ar[l] \ar[d] \\ X  & A \ar[l] }
\]
in which $Y \to X$ is a finite cover, there is a commutative diagram
\begin{equation}
      \label{RCohom_HPB}
\begin{tikzcd}
E^0(Y) \arrow[r] \arrow[d,color=\inductioncolor,"\Tr"] & E^0(H) \ar[d,color=\inductioncolor,"\Tr"] \\ E^0(X) \ar[r] & E^0(A),
\end{tikzcd}
\end{equation}
where the horizontal maps are restriction maps and the vertical maps are transfer maps. For any subgroups $H,K \leq G$, the homotopy pullback of the span $BH \to BG \leftarrow BK$ is equivalent to
$(G/H)_{hK}$.

Making use of the isomorphism of $\Sigma_m \times G$-sets
\[
      \left(G \wr \Sigma_m\right) / \left(H \wr \Sigma_m\right) \cong (G/H)^{\times m},
\]    
we get the following proposition:
\begin{prop}
      We have the following homotopy pullback of spaces
\[
      \begin{tikzcd}
      BH \wr \Sigma_m \arrow[d] & 
      (G/H)^{\times m}_{h(G\times \Sigma_m )} \arrow[d] \arrow[l] \\ 
      BG \wr \Sigma_m & 
      BG \times B\Sigma_m. \arrow[l]
      \end{tikzcd}
\]
\end{prop}
Applying $E$-cohomology and composing the resulting diagram \cref{RCohom_HPB} with the total power operation diagram \eqref{TPO_EQ} gives the commutative diagram
\begin{equation}
      \label{BeforeJDiagram}
      \begin{tikzcd}
E^0(BH) \arrow[r] \arrow[d, color=\inductioncolor, "\Tr"'] & 
E^0((G/H)^{\times m}_{h(G\times \Sigma_m )}) \arrow[d, color=\inductioncolor, "\Tr"'] 
\\ E^0(BG) \arrow[r,"P_m"] & 
E^0(BG \times B\Sigma_m).
\end{tikzcd}
\end{equation}

The subset $\Delta(G/H) := \sets{(gH, \ldots, gH)}{g \in G} \subset (G/H)^{\times m}$
is closed under the action of $G\times \Sigma_m $, and there is an equivalence
\[
      \Delta(G/H)_{h(G\times \Sigma_m)}
      \simeq  BH \times B\Sigma_m.
\]
Thus we have a decomposition of spaces
\[
      (G/H)^{\times m}_{h(G\times \Sigma_m)}
      \simeq
      \big(  BH \times B\Sigma_m \big)  \ \coprod \ Z^{G,H}_{h(G\times \Sigma_m)},
\]
where 
\begin{equation} \label{Z}
      Z^{G,H} = (G/H)^{\times m} \setminus \Delta(G/H).
\end{equation}
Applying $E$-cohomology, there is an isomorphism
\begin{equation}
      \label{ECornerDecomp}
      E^0((G/H)^{\times m}_{h(G\times \Sigma_m)})
      \cong
      E^0(BH \times B\Sigma_m) \times E^0(Z^{G,H}_{h(G\times \Sigma_m)}).
\end{equation}
We can obtain $E^0(BH \times B\Sigma_m)/I_{\Tr}$ from this product by taking the quotient by the ideal generated by transfers along
$H \times \Sigma_i \times \Sigma_j  \subseteq H \times \Sigma_m$
for $i,j>0$ and $i+j = m$ and also the transfer along the component
$Z^{G,H}_{h(G\times \Sigma_m)} \subseteq (G/H)^{\times m}_{h(G\times \Sigma_m)}$ (ie. the entire right factor).
We thus make the following definition.

\begin{definition}
      \label{JHG_DEF}
      Define
      $J_{H}^G \subseteq E^0(BG \times B\Sigma_m)$
      to be the ideal generated by the image of the transfers along
      \begin{enumerate}[(i)]
      \item\label{TransferG}  $G\times \Sigma_i \times \Sigma_j \subseteq G\times \Sigma_m$ for $i,j>0$ and $i+j = m$, and
      \item\label{CollapseZ} the composite
            \begin{equation}
                  \label{CollapseZ_EQ}
                  Z^{G,H}_{h(G\times \Sigma_m)} \subseteq
                  (G/H)^{\times m}_{h(G\times \Sigma_m)} \to
                  BG \times B\Sigma_m.
            \end{equation}
      \end{enumerate}
\end{definition}

The following result is then immediate from the above discussion.

\begin{prop} \label{mainthm}
Let $J_{H}^G \subseteq E^0(BG \times B\Sigma_m)$ be the ideal defined above. After taking the quotient by $J_{H}^G$, the transfer and additive power operation are compatible in the sense that the following diagram commutes:
\[
      \begin{tikzcd}[column sep = large]
            E^0(BH) \arrow[d, color=\inductioncolor, "\Tr"'] \arrow[r, "P_m/I_{\Tr}"]
            &
            E^0(BH \times B \Sigma_m)/I_{\Tr} \arrow[d, color=\inductioncolor, "\Tr"]
            \\
            E^0(BG) \arrow[r, "P_m/J_H^G"]
            &
            E^0(BG \times B\Sigma_m)/J_H^G.
      \end{tikzcd}
\]
\end{prop}

\begin{proof}
Consider the commutative diagram \eqref{BeforeJDiagram}. According to \eqref{ECornerDecomp}, the top right vertex decomposes as a product, one factor of which is the desired $E^0(  BH \times B\Sigma_m)$. Thus in order for the right vertical transfer in \eqref{BeforeJDiagram} to factor through the projection onto $E^0(  BH \times B\Sigma_m )/I_{\Tr}$, we must collapse the image in $E^0( BG \times B\Sigma_m)$ of the complementary factor $E^0(Z^{G,H}_{h(G\times \Sigma_m)})$ and $I_{\Tr}$; 
these desiderata motivated the definition of $J_H^G$.
\end{proof}

We give a complete description of $J_H^G \subseteq E^0(BG \times B \Sigma_m)$ in the case where $H \trianglelefteq G$ is a normal subgroup: Assume that $S$ is a subgroup of $G$ with $H < S \leq G$ and $[S : H] = n$, where $n$ divides $m$. Let $a_{S/H} \colon S \to \Aut_{\Set}(S/H) \cong \Sigma_n$ be the action map by left multiplication and let $\Gamma(a_{S/H}) \subseteq G \times \Sigma_n$ be the graph of $a_{S/H}$. 
The projection from $G\times \Sigma_n$ onto $\Sigma_n$ equips the graph $\Gamma(a_{S/H})$ with a map to $\Sigma_n$. Using this, we can form the semi-direct product
\[
(\Sigma_{q})^{n} \rtimes \Gamma(a_{S/H}),
\]
where $qn = m$.
Now we have
\[
(\Sigma_{q})^{n} \rtimes \Gamma(a_{S/H}) \leq (\Sigma_{q})^{n} \rtimes (G \times \Sigma_n) \cong G \times (\Sigma_{q})^{n} \rtimes \Sigma_n \leq G \times \Sigma_m.
\]
See \cref{WREATH_NOTN} for another description of this subgroup of $G \times \Sigma_m$.

The following is a direct consequence of \cref{NORMSURJ_COR}.

\begin{theorem}\label{JHG_ACTION_THM}
      Fix a normal subgroup $H \triangleleft G$.
      Then $J_H^G \subseteq E^0(BG \times B \Sigma_m)$ is the ideal generated by $I_{\Tr}$ and
      the images of the  transfers along the inclusions
      \begin{equation}
            \label{JHG_ACTION_EQ}
            (\Sigma_{q})^{n} \rtimes \Gamma(a_{S/H}) \longto G \times \Sigma_m
      \end{equation}
      for all $m = nq$ and $H < S \leq G$ with $[S : H] = n \neq 1$.
\end{theorem}

Note that although the definition of $a_{S/H}$ depends on a choice of ordering of $S/H$, the choice will not affect the image of the transfer. 

We also consider the related absolute ideal, to ensure compatibility with transfers from all subgroups of $G$.

\begin{definition}
      \label{JG_DEF}
      Define $J^G \subseteq E^0( BG \times B \Sigma_m )$
      to be the ideal generated by $ J_H^G$ for all $H\leq G$.
      More explicitly,
      $J^G$ is the ideal generated by the image of the transfers along
      \begin{enumerate}[(i)]
      \item\label{TransferGG}  $G\times  \Sigma_i \times \Sigma_j \subseteq G\times \Sigma_m$ for $i,j>0$ and $i+j = m$, and
      \item\label{CollapseZH} the composites
            \begin{equation}
                  \label{CollapseZH_EQ}
                  Z^{G,H}_{h(G\times \Sigma_m)} \subseteq
                  (G/H)^{\times m}_{h(G\times \Sigma_m)} \to
                  BG \times B\Sigma_m 
            \end{equation}
            for all $H < G$.
      \end{enumerate}
\end{definition}

\cref{mainthm} implies the following.
\begin{corollary}
      \label{maincor}
      Let $J^G \subseteq E^0(BG \times B\Sigma_m)$ be the ideal defined above.
      Taking  the quotient by $J^G$, the additive power operation is compatible with \textit{all} transfers
      in the sense that the following diagram commutes for all $H < G$:
      \begin{equation}
            \begin{tikzcd}
                  E^0(BH) \arrow[d, color=\inductioncolor,"\Tr"'] \arrow[r, "P_m/I_{\Tr}"]
                  &
                  E^0(BH \times B\Sigma_m)/I_{\Tr} \arrow[d, color=\inductioncolor,"\Tr"]
                  \\
                  E^0(BG) \arrow[r, "P_m/J^G"]
                  &
                  E^0(BG\times B\Sigma_m)/J^G.    
            \end{tikzcd}
      \end{equation}
\end{corollary}

The following description of $J^G$
 is a  consequence of \cref{STABSURJ_PROP}.

\begin{theorem}\label{JG_ACTION_THM}
      The ideal $J^G \subseteq E^0(BG \times B \Sigma_m)$ is generated by $I_{\Tr}$ and
      the images of the  transfers along the inclusions
      \begin{equation}
            \label{JG_ACTION_EQ}
            (\Sigma_{q})^{n} \rtimes \Gamma(a_{S/K}) \longto G \times \Sigma_m
      \end{equation}
      for all $m = nq$ and $K < S \leq G$ with $[S : K] = n \neq 1$,
      where $a_{S/K} \colon S \to \Aut_{\Set}(S/K) \cong \Sigma_n$ is the action map by left multiplication.
\end{theorem}

\subsection{Stabilizers of elements in $(G/H)^{\times m}$}
\label{STAB_SEC}

Our goal is to understand the ideals $J_H^G$ and $J^G$ appearing in \cref{mainthm,maincor} and defined in \cref{JHG_DEF,JG_DEF}
using two collections of transfers.
In general, there can be overlap between these transfers in the following sense:
sometimes the map from a component of $Z^{G,H}_{h(G \times \Sigma_m)}$ 
factors through
$BG \times B\Sigma_i \times B\Sigma_j$ for some choice of $i$ and $j$.
It then suffices to describe the components of $Z^{G,H}_{h (G \times \Sigma_m)}$ that 
do not factor through $BG \times B\Sigma_i \times B\Sigma_j$ for any $i,j > 0$ such that $i+j=m$.

To start, we note that $Z^{G,H}_{h(G \times \Sigma_m)}$ of \eqref{CollapseZ_EQ} 
is equivalent to the disjoint union of classifying spaces of the form
$B\Lambda$ 
for $\Lambda \leq G \times \Sigma_m$ the stabilizer of some element in $Z^{G,H}$ \eqref{Z}.
Moreover, the associated component of $Z^{G,H}_{h (G \times \Sigma_m)}$
does not factor through some $B G \times B \Sigma_i \times B \Sigma_j$
if and only if
the image $\pi_{\Sigma_m}(\Lambda) \leq \Sigma_m$ is a \textit{transitive} subgroup,
where $\pi_{\Sigma_m} \colon G \times \Sigma_m \to \Sigma_m$ is the projection.
Thus, it suffices to analyze the stabilizers of the $G \times \Sigma_m$-action on $(G/H)^{\times m}$ 
that have transitive image in $\Sigma_m$.

Elements of the diagonal $\Delta(G/H) = \sets{(gH, \dots, gH)}{g \in G}$ have the simplest stabilizers:
$\Stab_{G \times \Sigma_m}(gH,\dots, gH) = g H g^{-1} \times \Sigma_m$.
However, the stabilizers of the elements of $Z^{G,H}$ can be quite complicated.

In this section, we establish some group-theoretic results regarding these stabilizers
and set up notation for describing these groups in the sections ahead.

\begin{notation}\label{WREATH_NOTN}
      Given $\Lambda \leq G \times \Sigma_n$ and a group $H$, let
      $H \wr_{\underline{n}} \Lambda \leq G \times (H \wr \Sigma_n)$ denote the preimage
      \begin{equation}
            \label{WRG_EQ}
            H \wr_{\underline{n}} \Lambda = \pi^{-1}\Lambda,
      \end{equation}
      where $G \times (H \wr \Sigma_n) \xrightarrow{\pi} G \times \Sigma_n$ is the canonical map.
\end{notation}

Note that the group $H\wr_{\underline{n}} \Lambda$ is isomorphic to $H^n \rtimes \Lambda$, where $\Lambda$ acts on $H^n$ via its projection to $\Sigma_n$. This follows from the canonical isomorphism $G\times (H\wr \Sigma_n) \iso H^n\rtimes  (G\times \Sigma_n)$.

\begin{notation}
      Let $X$ be a $G$-set, and $Y \subseteq X$ a finite subset.
      We write $S_Y \leq G$ to denote the {set-wise stabilizer} of $Y$,
      \[ S_Y = \{ s\in G \mid  \text{$s\cdot y\in Y$ for all $y \in Y$} \}.\]
      A choice of total ordering $Y = \set{y_1, y_2, \dots, y_n}$ induces an associated {action map}
      \begin{equation}
            \label{AY_EQ}
            a_Y \colon S_Y \longto \Aut_{\Set}(Y) \cong \Sigma_n.
      \end{equation}
Different choices of ordering on $Y$ give conjugate action homomorphisms.
\end{notation}

In general, $S_Y$, $a_Y$, and $\ker(a_Y)$ can be difficult to compute. We give one primary example.
\begin{example}
      \label{KMODH_EX}
      Let $H \leq G$,
      $X = G/H$, and $Y = \set{g_1H, \dots, g_n H} \subseteq G/H$.
      Then we have
      \[
            S_Y = \bigcup_{\sigma \in \im(a_Y)} \bigcap_{i=1}^{n}g_{\sigma(i)} H g_i^{-1}
            \qquad \text{and}
            \qquad
            \ker(a_Y) = \bigcap_{i=1}^n g_i H g_i^{-1}.
      \]
      If $Y = K/H$ for some $H \leq K \leq G$,
      then  $S_{K/H} = K$.
      Indeed, if $g \in G$ satisfies $ g \cdot eH = kH$, then $g = g\cdot e$ lies in  $k H \subseteq K$.
\end{example}

\begin{lemma}
      \label{GRAPHSTAB_LEM}
      Let $Y \subseteq X$ be a finite subset of a $G$-set, equipped with a total ordering $Y = \set{y_1, \dots, y_n}$.
      \begin{enumerate}
      \item[(i)] Let $\vect y = (y_1, y_2, \dots, y_n) \in X^{\times n}$. Then
            \[
                  \Stab_{G \times \Sigma_n}(\vect y) = \Gamma(a_Y), 
            \]
            where $\Gamma(a_Y)$ is the graph subgroup associated to  $a_Y\colon S_Y \rtarr \Sigma_n$.
      \item[(ii)] Let 
      \[ \vect y^{\** q} = ( \underbrace{y_1,\dots, y_1}_q, \underbrace{y_2,\dots, y_2}_q, \dots, \underbrace{y_n,\dots, y_n}_q)\]
      be the $q$-fold shuffle of $\vect y$.
      Then the stabilizer of $\vect y^{\** q}$ is 
            \[
                  \Stab_{G \times \Sigma_{qn}}(\vect y^{\** q})
                  = \Sigma_q \wr_{\underline{n}} \Gamma(a_Y) \leq G \times (\Sigma_q \wr \Sigma_n) \leq G \times \Sigma_{qn}.
            \]
      \end{enumerate}
\end{lemma}

\begin{proof}
      For (i), $(g, \sigma)$ is in $\Stab(\vect y)$ if and only if $g y_i = y_{\sigma i}$ for all $i$.
      This defines an action of $g$ on $Y$, and thus $g$ is in $S_Y$ and $\sigma = a_Y(g)$.

      For (ii), it is clear that $(\Sigma_q)^{\times n}$ stabilizes the $q$-fold shuffle $\vect y^{\** q}$. 
      The group $\Sigma_n$ acts by permuting the blocks of size $q$.
      Given $(g,(\vect \tau,\sigma))\in G\times (\Sigma_q\wr \Sigma_n)$, we have
      $\left((g, (\vect \tau,\sigma)) \cdot \vect y^{\** q}\right)_{i + kq} = g y_{\sigma^{-1} k}$ for $1\leq i \leq q$.
      Thus if $(g, (\vect \tau,\sigma))$ is in the stabilizer, we must have $(g, \sigma) \in \Gamma(a_Y)$,
      while the $\tau_i \in \Sigma_q$ have no restrictions.
\end{proof}

\subsection{The absolute transfer ideal $J^G$}
\label{ABS_SEC}
In \cref{STABSURJ_PROP}, we give a complete description  of the subgroups of $G\times \Sigma_m$ that we must transfer along  to form $J^G$.
This description will be used in the proof of \cref{JSUBMACKEY_LEM}.

First, we establish a restricted case. Let $\Delta^{\mathrm{fat}}(G/H) \subseteq (G/H)^{\times n}$ be the fat diagonal, which consists of tuples of cosets in which two or more of the cosets are identical.

\begin{proposition}
      \label{FATSURJ_PROP}
      Fix $n \geq 1$. Then the assignment of the stabilizer $\Stab_{G \times \Sigma_n}(-)$
      admits a section $\zeta$
      \begin{equation}
            \label{FATSURJ_EQ}
            \begin{tikzcd}[ampersand replacement=\&]
                  \left\{
                        \begin{array}{c|c}
                          \vect g H \in (G/H)^{\times n} \setminus \Delta^{\mathrm{fat}}(G/H)
                          \Big. 
                          &
                          H \leq G,\ \pi_{\Sigma_n}(\Stab( \vect g H)) \leq \Sigma_n \mbox{ transitive}
                        \end{array}
                  \right\}
                  \arrow[d, "\Stab"]
                  \\
                  \left\{
                        \begin{array}{c|c}
                        \Big. 
                         \Gamma(\phi) \leq G \times \Sigma_n 
                          &
                        \Big.
                        \phi \colon S \to \Sigma_n,\
                        S \leq G,\
                        \im(\phi) \leq \Sigma_n \mbox{ transitive}
                        \end{array}
                  \right\}
                  \arrow[u,bend left=25,dashed,"\zeta"]
            \end{tikzcd}
      \end{equation}
             and  is therefore surjective.
\end{proposition}

\begin{proof}
      According to \cref{GRAPHSTAB_LEM}(i), the stabilizer is a graph subgroup.
      Since the cosets $g_i H$ are all distinct 
      and $G/H$ is a transitive $G$-set,
      the image of $\phi$ in $\Sigma_n$ must be a transitive subgroup.
      
      Now, any $\phi$ as above encodes a transitive action of $S$ on $\{1,\dots,n\}$.
      In particular, 
      letting $K=\Stab_S(1)$, 
      the action
      provides a bijection $S/K\iso \{1,\dots,n\}$, which specifies an ordering of $S/K$.
      Thus we may consider $S/K$ as an element of $(G/K)^{\times n}$, and
      the assignment $\Gamma(\phi) \mapsto S/K \in (G/K)^{\times n}$
      is a section of \cref{FATSURJ_EQ} by \cref{KMODH_EX}.
\end{proof}

\begin{remark}
      \label{NEWSTAB_REM}
      In light of the description of the section $\zeta$ to \eqref{FATSURJ_EQ} given in the proof above,
      after passing to $\Sigma_n$-conjugacy classes, 
      we may replace the target of $\Stab$ in \eqref{FATSURJ_EQ} with
      \begin{equation}
            \left\{
                  \begin{array}{c|c}
                    [ \Gamma(a_{S/K}) \leq G \times \Sigma_n ] \Big.
                    &
                      K \leq S \leq G,\ [S \colon K] = n
                  \end{array}
            \right\},
      \end{equation}
      where $a_{S/K}\colon S \rtarr \Sigma_n$ is the action homomorphism specified in \cref{AY_EQ}.
     Note that different choices of orderings $S/K \xrightarrow{\iso} \underline{n}$
      induce the same $\Sigma_n$-conjugacy class $[\Gamma(a_{S/K})]$.
\end{remark}

      We note that the source in \cref{FATSURJ_EQ} runs over {\it all} subgroups of $G$. 
      One might hope for a similar result with a fixed $H\leq G$. However, 
      \[
            K = \Stab_S(1) = a_{\vect gH}^{-1}(\Sigma_1 \times \Sigma_{n-1}) = g_1 H g_1^{-1} \cap S_{\vect g H}
      \]
      need not equal $H$, as we show in \cref{S_EXAMPLE}. 
      Therefore, we cannot expect a section if we first fix $H\leq G$.
      We will show in \cref{FATNORMSURJ_COR} that
      such a section does exist if $H$ is normal in $G$. 

\begin{example}
      \label{S_EXAMPLE}
      Let $G = D_{8}$, generated by a rotation $r$ and a reflection $s$. 
      Let $H = \langle rs \rangle$, and consider $(eH, sH) \in (G/H)^{\times 2}$.
            We check by hand that $S = \langle s \rangle$ and $K = e$,
            so the section $\zeta$ sends $\Gamma(a_{(eH,sH)})$ to the $S$-set $S/e = (e, s)$.

            Note in particular that $K \neq H$, $m = 2$ is prime, and $r^2 \in H$.
            Thus in general we cannot  a priori fix $H \leq G$ in \eqref{FATSURJ_EQ}.
\end{example}

We will now give a complete description of the stabilizers which appear in $J^G$.
Recall that $Z^{G,H}$ is $(G/H)^{\times m} \setminus \Delta(G/H)$.

\begin{proposition}
      \label{STABSURJ_PROP}
      Fix $m \geq 1$. Then the assignment of the stabilizer $\Stab_{G \times \Sigma_m}(-)$ gives a surjection
      \begin{equation}
            \label{STABSURJ_EQ}
            \begin{tikzcd}[ampersand replacement=\&, column sep=small]
                  \left\{
                        \begin{array}{c|c}
                          \vect g H \in Z^{G,H}
                          \Big. 
                          &
                            H < G,\
                            \pi_{\Sigma_m}(\Stab( \vect g H ) ) \leq \Sigma_m \mbox{ transitive}
                        \end{array}
                  \right\}
                  \arrow[d, "\Stab"]
                  \\
                  \left\{
                        \begin{array}{c|c}
                          [\Sigma_q \wr_{\underline{n}} \Gamma(a_{S/K}) \leq G \times \Sigma_m]
                          \Big.
                          &
                            m = nq,\
                            K < S \leq G,\
                            [S : K] = n \neq 1
                        \end{array}
                  \right\},
            \end{tikzcd}
      \end{equation}
      where
      $\Sigma_q \wr_{\underline{n}} \Gamma(a_{S/K})$ is defined as in \eqref{WRG_EQ}
      and $[\Sigma_q \wr_{\underline{n}} \Gamma(a_{S/K})]$ denotes the $\Sigma_m$-conjugacy class of the subgroup in $G \times \Sigma_m$.
\end{proposition}

\begin{proof}
      Suppose given $\vect g H$ in $Z^{G,H} \subset (G/H)^{\times m}$.
      For  $(g, \sigma) \in \Stab_{G \times \Sigma_m}(\vect g H)$, we have
      \[
            g_i H = g_j H \Leftrightarrow
            g \cdot g_i H = g \cdot g_j H \Leftrightarrow
            g_{\sigma(i)} H = g_{\sigma(j)} H.
      \]
      Since  $\pi_{\Sigma_m}(\Stab( \vect g H ) )$ is a transitive subgroup of $\Sigma_m$, we conclude that
      after  reordering if necessary, 
      $\vect g H$ is the $q$-fold shuffle of $(g_1H,\dots,g_n H)$, where the cosets
      $g_1H$, \dots, $g_n H$ are distinct.
      Using both parts of \cref{GRAPHSTAB_LEM}, we see that
      \[
            \Stab_{G \times \Sigma_m}( \vect g H ) = \Sigma_q \wr_{\underline{n}} \Stab_{G \times \Sigma_n}((g_1H,\dots,g_n H)),
      \]
      and so the arrow \eqref{STABSURJ_EQ} is well-defined.

      We wish to show that \eqref{STABSURJ_EQ} is surjective.
      A choice of $a_{S/K} \colon S \to \Sigma_n$ specifies an ordering of $S/K$.
      Given $\Sigma_q \wr_{\underline{n}} \Gamma(a_{S/K})$, 
      the assignment
      \[
           \zeta\colon \Sigma_q \wr_{\underline{n}} \Gamma(a_{S/K}) \longmapsto (S/K)^{\** q},
      \]
      is a section of \eqref{STABSURJ_EQ} by an argument similar to that used in the proof of \cref{FATSURJ_PROP}.
\end{proof}

\subsection{The relative ideal $J_H^G$ for normal subgroups $H \trianglelefteq G$}
\label{NORMAL_SEC}

The relative transfer ideal $J_H^G$ was defined in \cref{JHG_DEF} and appears in \cref{mainthm}.
\cref{FATSURJ_PROP,STABSURJ_PROP} can be used to obtain a description 
of the subgroups that we must transfer along  to form $J_H^G$
in the case that $H$ is normal in $G$.

\begin{prop}
      \label{FATNORMSURJ_COR}
      Fix $n \geq 1$ and $H \triangleleft G$ normal. The map \eqref{FATSURJ_EQ} restricts to give a surjection
      \begin{equation}
            \label{FATNORMSURJ_EQ}
            \begin{tikzcd}[ampersand replacement = \&]
                  \left\{
                        \begin{array}{c|c}
                          \vect g H \in (G/H)^{\times n} \setminus \Delta^{\mathrm{fat}}(G/H)
                          \Big.
                          &
                            \pi_{\Sigma_n}(\Stab( \vect g H)) \leq \Sigma_n \mbox{ transitive}
                        \end{array}
                  \right\}
                  \arrow[d, "\Stab"]
                  \\
                  \left\{
                        \begin{array}{c|c} 
                       [   \Gamma(a_{S/H}) \leq G \times \Sigma_n ]
                           \Big.
                          &
                            H \leq S \leq G,\ [S : H] = n
                        \end{array}
                  \right\},
            \end{tikzcd}
      \end{equation}
      where $[-]$ denotes the $\Sigma_n$-conjugacy class of the subgroup in $G \times \Sigma_n$.
\end{prop}

\begin{proof}
      Since $H$ is normal, the $G$-stabilizer of each $g_iH \in G/H$ is $H$. 
      Thus, if $S$ denotes the set-wise stabilizer of $\vect{g}H\subset G/H$ and $K\leq S$ denotes 
      the stabilizer in $S$ of $g_1 H\in G/H$,
      then $K=H$.
      Furthermore, since $K=H$, the section $\zeta$ of \eqref{FATSURJ_EQ} restricts to a section for fixed, normal $H$.
      It follows that \cref{FATNORMSURJ_EQ} is surjective.
\end{proof}

The next result is an analogue of \cref{STABSURJ_PROP}.

\begin{proposition}
      \label{NORMSURJ_COR}
      Fix a normal subgroup $H \triangleleft G$.
      Then the assignment of the stabilizer \eqref{STABSURJ_EQ} restricts to a surjection
      \begin{equation}
            \label{NORMSURJ_EQ}
            \begin{tikzcd}[ampersand replacement = \&]
                  \left\{
                        \begin{array}{c|c}
                          \vect gH \in Z^{G,H}
                          \Big.
                          &
                            \pi_{\Sigma_m}(\Stab(\vect g H)) \leq \Sigma_m \mbox{ transitive}
                        \end{array}
                  \right\}
                  \arrow[d, "\Stab"]
                  \\
                  \left\{
                        \begin{array}{c|c}
                          [\Sigma_q \wr_{\underline{n}} \Gamma(a_{S/H}) \leq G \times \Sigma_m]
                          \Big.
                          &
                            m = nq,\
                            H < S \leq G,\
                            [S:H]= n \neq 1
                        \end{array}
                  \right\},
            \end{tikzcd}
      \end{equation}
      where
      $\Sigma_q \wr_{\underline{n}} \Gamma(\phi)$ is defined as in \cref{WRG_EQ}
      and $[-]$ denotes the $\Sigma_m$-conjugacy class of the subgroup in $G \times \Sigma_m$.
\end{proposition}

\subsection{Specializing to a prime}
\label{PRIME_SEC}

In this section, we give more explicit identifications of the subgroups of $G \times \Sigma_p$ which appear in $J^G$ and $J_H^G$, where $p$ is a prime.
We identify the relevant tuples in $Z^{G,H}$ (\cref{LongCycleStab})
and give closed-form descriptions of their stabilizers.

Working at a prime has the advantage that transitive subgroups of $\Sigma_p$ are exactly those which contain 
a $p$-cycle $\sigma_p$.
This is an immediate consequence of Cauchy's theorem.

We first classify, for general $m$, those tuples $\vect gH \in (G/H)^{\times m}$ 
for which $\pi_{\Sigma_m}(\mathrm{Stab}(\vect g H))$ contains a long cycle.

\begin{prop}
      \label{LongCycleStab}
      Assume that $\vect {g}H = (eH, g_1H, \ldots, g_{m-1}H)$, and let $\sigma_m = (1\ 2 \dots \ m)$ be the long cycle.
      Then $(g, \sigma_m)$ 
      lies in $\Stab(\vect{g}H)$ if and only if $g^m \in H$ and
      \[
            \vect{g}H = (eH, gH, g^2H, \ldots, g^{m-1}H).
      \]
\end{prop}

\begin{proof}
      First we will prove the forward direction. By direct observation, we see
      \[
            g \in (Hg_{m-1}^{-1}) \cap (g_{m-1}Hg_{m-2}^{-1}) \cap \ldots \cap (g_1H).
      \]
      Thus, there exists $h_i \in H$ such that $g = g_ih_ig_{i-1}^{-1}$ when $0 < i \leq m$
      (where we set $g_0 = g_m = e$).
      With this convention we see that
      \[
            g^iH = g_ih_ig_{i-1}^{-1}\cdot g_{i-1}h_{i-1}g_{i-2}^{-1} \cdots g_1H = g_iH.
      \]
      Thus we have that $(g,\sigma_m)$ stabilizes $\vect g H = (eH, gH, g^2H, \ldots, g^{m-1}H)$,
      so  $g^m \in H$.
      
      For the reverse direction, it suffices to note that under the condition that $g^m$ is  in $H$, 
      \[\begin{split}
        (g, \sigma_m) \cdot (eH, gH, g^2H, \ldots, g^{m-1}H) &= (gg^{m-1}H, geH, ggH, \ldots, gg^{m-2}H) \\ &= (eH, gH, g^2H, \ldots, g^{m-1}H).
     \qedhere\end{split} \]
\end{proof}

\begin{corollary}
      Using the notation of \cref{FATSURJ_PROP}, the subgroup $\pi_{\Sigma_p} (\Stab(\vect g H))\leq \Sigma_p$ 
      is transitive if and only if $\vect gH \in (G/H)^{\times p}$ lies in the same 
      $G\times \Sigma_p$-orbit as the $p$-tuple $(g^iH) = (eH, gH, g^2 H, \dots g^{p-1}H)$ for some $g \in G$ such that $g^p\in H$.
\end{corollary}

\cref{FATSURJ_PROP} then specializes to the following.
\begin{corollary}
      \label{PRIMESURJ_COR}
      Fix a prime $p$. Then the assignment of the stabilizer $\Stab_{G \times \Sigma_p}(-)$ gives a surjection 
      \begin{equation}
            \label{PRIMESURJ_EQ}
            \begin{tikzcd}[ampersand replacement=\&, column sep=small]
                  \left\{
                        \begin{array}{c|c}
                          (g^iH) \in Z^{G,H}\subset  (G/H)^{\times p}
                          &
                            \Big.
                            {H < G,\  g^p \in H,\ g \notin H}
                        \end{array}
                  \right\} 
                  \arrow[d, "\Stab"] 
                  \\
                  \left\{
                        \begin{array}{c|c}
                          [    \Gamma(a_{S/K}) 
                          \leq G \times \Sigma_p  ]
                          &
                            \Big.
                            {K < S \leq G,\ [S:K] = p}
                        \end{array}
                  \right\},
            \end{tikzcd}
      \end{equation}
where $[-]$ denotes the $\Sigma_p$-conjugacy class of the subgroup in $G \times \Sigma_p$.
\end{corollary}

\begin{remark}
The codomain of \cref{PRIMESURJ_EQ} can be described more simply as 
\[
      \left\{
            \begin{array}{c|c}
            [  \Gamma( S \xrtarr{a} \Sigma_p) \leq G \times \Sigma_p. ]
              \Big.
              &
                S\leq G, \text{$\im(a)$ contains a $p$-cycle}
            \end{array}
      \right\}
\]
\end{remark}

\cref{S_EXAMPLE} shows that we still cannot restrict to a fixed $H < G$ on either side for general subgroups $H$.
However, a cleaner description of $\Gamma(a_{S/K})$ does occur for $H$ normal in  $G$.

\begin{notation}
      For $\vect g H = (g^i H) \in (G/H)^{\times p}$ with $g^p \in H$ and $H \triangleleft G$ normal,
      we write $a_g \colon S_g \to \Sigma_m$ for the action map
      (so $a_g = a_{(g^i H)}$, $S_g = S_{(g^i H)}$).
\end{notation}

The following result is a specialization of \cref{FATNORMSURJ_COR}.

\begin{corollary}
\label{PRIME_NORMSURJ_COR}
      Fix a prime $p$ and a normal subgroup $H \trianglelefteq G$.
      The assignment $(g^i H) \mapsto \Stab_{G\times \Sigma_p}\big((g^i H)\big) = \Gamma(a_g)$ induces a surjection
            \begin{equation}
            \label{NORMPRIMESURJ_EQ}
            \begin{tikzcd}[ampersand replacement=\&]
                  \left\{
                        \begin{array}{c|c}
                          (g^iH) \in Z^{G,H} \subset (G/H)^{\times p}
                          \Big.
                          &
                           {  g^p \in H,\ g \notin H}
                                                   \end{array}
                  \right\}
                  \arrow[d, "\Stab"]
                  \\
                  \left\{
                        \begin{array}{c|c}
[                       \Gamma(a_{S/H}) 
                          \leq G \times \Sigma_p ]
                          \Big.
                          &
                            {H < S \leq G,\ [S:H] = p}
                        \end{array}
                  \right\}.
            \end{tikzcd}
      \end{equation}
      Moreover, the action map $a_g$ is precisely
      \[
            a_g \colon \langle H,g \rangle \longto C_p \leq \Sigma_p,
            \qquad
            g^i H \longmapsto \sigma_p^i
      \]
      with $\sigma_p = (1\ 2\ \dots p)$ the long cycle.
\end{corollary}

\begin{proof}
It remains to describe the action map $a_g$.
First, note that, given $(g^i H)$, the set-wise stabilizer of $(g^i H) \subset G/H$ is the subgroup $\langle H, g\rangle \leq G$.
Thus $S=S_g = \langle H, g\rangle$.
The formula then follows from the fact that $S/H = \langle H,g\rangle/H$ is isomorphic to $C_p$.
\end{proof}

Generally speaking, the action maps $a_{S/K}$ are difficult to understand. However, in \cref{PRIME_NORMSURJ_COR}, we explicitly describe the action map when $H$ is normal in $G$ and $m=p$ is prime. This will be useful later when we compute power operations.

\subsection{The relatively prime case}
\label{RELPRIME_SEC}

In this section, we record conditions on the integers $m$, $|G|$, and $|G/H|$ that force
every component of $Z^{G,H}_{h(G \times \Sigma_m)}$ to factor through  $BG \times B \Sigma_i \times B\Sigma_j$ for some $i,j>0$ with $i+j = m$.

\begin{corollary}
      \label{JGRELPRIME_COR}
      Suppose that $m$ and $|G|$ are relatively prime.
      Then $J^G = I_{\Tr}$.
\end{corollary}
\begin{proof}
      By \cref{JG_ACTION_THM},
      it suffices to show that the codomain of \eqref{STABSURJ_EQ} is empty.
      Suppose not; then
      we would have subgroups $K < S \leq G$ with $[S \colon K] \neq 1$ dividing $m$.
      But $[S : K]$ divides $|S|$ and hence $|G|$, a contradiction.
\end{proof}

When $H \triangleleft G$ is normal, we have the following specification of \cref{JGRELPRIME_COR}.
\begin{corollary}
      Let $H \triangleleft G$ be normal, and suppose $m$ and $|G/H|$ are relatively prime.
      Then $J_H^G = I_{\Tr}$.
\end{corollary}
\begin{proof}
      Similarly, by \cref{JHG_ACTION_THM} it suffices to show that 
      the codomain of \eqref{NORMSURJ_EQ} is empty.
      Suppose not; then there exists $H < S \leq G$ such that $[S : H]$ is larger than 1 and  divides $m$.
      But $[S : H]$ also divides $[G : H] = |G/H|$, a contradiction.
\end{proof}

\begin{remark}
      We record that this result fails if $H$ is not a normal subgroup.
      Consider $G = \Sigma_3$ with $H = \{e, (12)\}$ so that $|G/H| = 3$, and let $m=2$.
      We note that $((13), (12)) \in G \times \Sigma_2$ is in the stabilizer of $(eH, (13)H) \in (G/H)^{\times 2}$.
      Therefore, letting $\Gamma\leq G\times \Sigma_2$ be the order two subgroup generated by the element $((13), (12))$,
      the ideal $J_H^G$ contains the image of the transfer along $\Gamma \rtarr G\times \Sigma_2$, which is
      not contained in $I_{\Tr}$.
\end{remark}


\section{Additive power operations and Green functors}
\label{sec:GreenPowerOps}

Making use of the group-theoretic results in \cref{GROUPTHEORY_SEC}, we provide in this section a general framework for additive power operations in the equivariant setting. In \cref{subsec:Mackey} we recall the notion of a $G$-Green functor and describe two sources of examples from equivariant homotopy theory. Motivated by the discussion in \cref{OVERVIEW_SEC}, 
in \cref{sec:Ideals} we prove that the induced $G$-Green functor associated to a $G \times \Sigma_m$-Green functor contains a canonical Mackey ideal $\ul{J}$. In \cref{GREEN_POWER_OPS_SEC}, we introduce the notion of a $G \times \Sigma_m$-Green functor with $m$th total power operation and show in \cref{Green_Main_Thm} that taking the quotient by the Mackey ideal $\ul{J}$ leads to a reduced power operation that is a map of Green functors. In the final two subsections, we show that $H_{\infty}$-rings in $G$-spectra and $G_\infty$-rings in global spectra provide two classes of examples of $G \times \Sigma_m$-Green functors with $m$th total power operation.

\subsection{Reminder on Mackey functors and Green functors} \label{subsec:Mackey}

Let $G$ be a finite group. Recall that a $G$-Mackey functor $\ul{M}$ consists of abelian groups $\ul{M}(G/H)$ for each subgroup $H\leq G$, together with restriction and induction maps
\[
      \begin{tikzcd}
            \mathrm{Res}\colon \ul{M}(G/K) \arrow[r]
            &
            \ul{M}(G/H)
            & 
            \text{and}
            &
            \mathrm{Tr}\colon \ul{M}(G/H) \arrow[r, color=\inductioncolor]
            &
            \ul{M}(G/K)
      \end{tikzcd}
\]
for each map $G/H \to G/K$,
satisfying a number of axioms.
The most notable axiom is the double-coset formula, which 
we describe in \cref{DOUBLECOSET_REM} below.

A $G$-Mackey functor can be extended 
to all finite $G$-sets via
\[
      \underline{M}(A_1 \amalg \dots A_n) = \bigoplus_i \ul{M}(A_i) \iso \bigoplus_i \ul{M}(G/H_i)
\]
for $G$-orbits $A_1$, \dots, $A_n$.

\begin{rem}
      \label{DOUBLECOSET_REM}
      The double-coset formula  says that for any pullback of $G$-sets
\[ 
            \begin{tikzcd}
                  A  \arrow[d, twoheadrightarrow]
                  &
                  P \arrow[d, twoheadrightarrow] \ar[l]
                  \\
                  C 
                  &
                  B \ar[l],
            \end{tikzcd}
\]
      in which the vertical maps are surjective,
      the diagram of abelian groups
      \begin{equation}
            \label{DC_EQ}
            \begin{tikzcd}
            \ul{M}(A) \arrow[r, "\Res"] \arrow[d, "\Tr", color=\inductioncolor]
                  &
                  \ul{M}(P) \arrow[d, "\Tr", color=\inductioncolor]
                  \\
                  \ul{M}(C) \arrow[r, "\Res"] 
                  &
                  \ul{M}(B) 
            \end{tikzcd}
      \end{equation}
      commutes.
\end{rem}

\begin{defn} \label{defn:GGreen}
      A \textit{$G$-Green functor} is a $G$-Mackey functor $\ul{R}$ such that
      each $\ul{R}(G/H)$ is a commutative ring,
      each restriction map is a ring homomorphism,
      and each induction map $\mathrm{Tr}\colon \ul{R}(G/H) \indarr \ul{R}(G/K)$ is an $\ul{R}(G/K)$-module map.
      The condition that induction is a module map is also referred to as ``Frobenius reciprocity.''
      A \textit{Mackey ideal} in a Green functor is a sub-Mackey functor which is levelwise an ideal.
\end{defn}

Equivalently, a $G$-Green functor is a commutative monoid in the category of $G$-Mackey functors under the box product;
see e.g. \cite[Prop. 1.4]{LewisGreen}, \cite[Lemma 2.17]{Shu10}.

\begin{ex}
\label{ringGGreen}
      If $E$ is a homotopy-commutative ring (genuine) $G$-spectrum,
      then the $G$-Green functor of coefficients is given by
      \[
            \underline{E}^0(G/H) = E^0(G/H).
      \]
      The restriction and transfer maps are defined via naturality on the maps of $G$-sets
      $G/K \to G/H$ and the $G$-transfers $\Sigma^\infty_G G/H_+ \indxarr{\Tr} \Sigma^\infty_G G/K_+$.
      Moreover, the double coset formula \eqref{DC_EQ} is a special case of Nishida's push-pull property for equivariant cohomology \cite[Prop. 4.4]{Nis78} (see also \cite[IV.1]{LMS}).
\end{ex}

We will pay particular attention to $G\times \Sigma_m$-Green functors
and to the following associated $G$-Green functors.

\begin{defn} \label{def:induced}
      If $\underline{R}$ is any $G\times \Sigma_m$-Green functor, we define $\justupwardarrow{\ul{R}}{G\times \Sigma_m}{G}$
       to be the $G$-Green functor given by the induction 
      of $\ul{R}$ along the projection $G\times \Sigma_m \to G$ (see \cite[Lemma 5.4(ii)]{TWebb}). The value of $\justupwardarrow{\ul{R}}{G\times \Sigma_m}{G}$ on $G/H$ is given by
      \[
            \upwardarrow{\ul{R}}{G\times \Sigma_m}{G}{G/H} = \underline{R}\big((G\times \Sigma_m)/(H\times \Sigma_m)\big).
      \]
      In other words, given a $G$-set $A$, we define $\upwardarrow{\ul{R}}{G\times \Sigma_m}{G}{A}$ to be $\ul{R}(A)$, where in $\ul{R}(A)$ we equip $A$ with a trivial $\Sigma_m$-action. For example, there is an isomorphism of $G \times \Sigma_m$-sets $(G\times \Sigma_m)/(H\times \Sigma_m) \cong G/H$, where $G\times  \Sigma_m$  acts on $G/H$ through the projection $G\times \Sigma_m \to G$. 
\end{defn}

\begin{definition}
      \label{def:restricted}
      If $\underline{R}$ is any $G \times \Sigma_m$-Green functor, we define
      $ \justdownwardarrow{\underline{R}}{G}{G \times \Sigma_m}$
      to be the restriction of $\underline{R}$ along the inclusion $G \cong  G \times \set{e} \leq G \times \Sigma_m$. 
      Explicitly,
      \[
            \downwardarrow{\ul{R}}{G}{G \times \Sigma_m}{G/H} = \underline{R}\big((G\times \Sigma_m)/(H\times e)\big).
      \]
\end{definition}

\begin{ex} \label{BorelGSigma}
      If $E$ is a homotopy-commutative ring $G$-spectrum, then the assignment
      \begin{equation}
            \label{BorelGSigma_EQ}
            (G \times \Sigma_m)/\Lambda \mapsto E^0\Big(\big((G\times \Sigma_m)/\Lambda\big)_{h\Sigma_m}\Big),
      \end{equation}
      for $\Lambda$ a subgroup of $G\times \Sigma_m$, is a $G\times \Sigma_m$-Green functor. In this case, the induced $G$-Green functor as in \cref{def:induced} is given by $\underline{E}^0(B \Sigma_m)$, or explicitly the assignment
      \begin{equation}
            \label{BorelGInd_EQ}
            G/H \mapsto E^0(G/H \times B\Sigma_m),
      \end{equation}
      where $B\Sigma_m$ has a trivial $G$-action.
      The restricted $G$-Green functor as in \cref{def:restricted} is naturally isomorphic to $\underline{E}^0$, as
      \begin{equation}
            \label{BorelGRes_EQ}
            G/H \mapsto E^0\Big( \big(G/H \times \Sigma_m\big)_{h \Sigma_m} \Big) = E^0(G/H \times E \Sigma_m) \cong E^0(G/H).
      \end{equation}
\end{ex}

\begin{example}
\label{globalring_ex}
      If $E$ is a homotopy commutative global ring spectrum, then, for each finite group $G$, there is an underlying homotopy commutative ring $G$-spectrum $E_G$ (see the discussion leading up to \cite[Theorem 4.5.24]{global}). This gives a
      $G$-Green functor 
      \[
      \underline{E}^0_G(G/H) = [G/H_+,E_G]^G
      \]
      and a
      $G\times \Sigma_m$-Green functor $\underline{E}^0_{G\times \Sigma_m}$ (see \cref{GLOBAL_OPS_SEC}).
      In this case, the induced $G$-Green functor as in Definition \ref{def:induced} is given by the assignment
	  \[
	  G/H \mapsto \ul{E}^0_{G \times \Sigma_m}(G/H),
	  \]
	where $G/H$ is given a trivial $\Sigma_m$-action.
\end{example}

We highlight a particular case of the double-coset formula.

\begin{corollary}
      \label{PUSHPULL_COR_TWO}
      For any $G\times \Sigma_m$-Green functor $\underline{R}$,
      the following square of abelian groups commutes
      \[
            \begin{tikzcd}
                  \ul{R}\left(  (G/H)^{\times m} \right)
                  \arrow[r, "i^{\**}"] \arrow[d, color=\inductioncolor, "\Tr"']
                  &
                  \ul{R}\left( (G/H)^{\times_{G/L} m} \right)
                  \arrow[d, color=\inductioncolor, "\Tr"]
                  \\
                  \ul{R}\left(  (G/L)^{\times m} \right)
                  \arrow[r, "\Delta^{\**}"]
                  &
                  \ul{R}\left( G/L \right)
            \end{tikzcd}
      \]
      for any map of $G$-sets $G/H \rtarr G/L$.
\end{corollary}
\begin{proof}
      This follows from \cref{DOUBLECOSET_REM}, as we have a pullback square of $G\times \Sigma_m$-sets
      \[ 
            \begin{tikzcd}
                  (G/H)^{\times m}
                  \arrow[d]
                  &
                  (G/H)^{\times_{G/L} m}
                  \arrow[l, "i", swap] \arrow[d]
                  \\
                  (G/L)^{\times m}
                  &
                  G/L, \arrow[l,"\Delta", swap]
            \end{tikzcd}
      \]
      in which the vertical maps are surjective. 
\end{proof}

\subsection{\for{toc}{Certain ideals in the induced $G$-Green functor}\except{toc}{Certain ideals in $\justupwardarrow{\underline{R}}{G\times \Sigma_m}{G}$}}
\label{sec:Ideals}

Given a $G\times \Sigma_m$-Green functor $\ul{R}$, 
\cref{def:induced} produces a $G$-Green functor $\justupwardarrow{\underline{R}}{G \times \Sigma_m}{G}$.

\begin{notn}
\label{UpArrowAbr}
Since the induced Mackey functor $\justupwardarrow{\underline{R}}{G \times \Sigma_m}{G}$ will appear many times in this subsection, we will abbreviate it to $\Rup$.
\end{notn}

 In this subsection, we describe two Mackey ideals in the $G$-Green functor $\Rup$ that depend on the fact that $\Rup$ is induced from $\ul{R}$. The 
definitions of these Mackey ideals are motivated by considerations coming from power operations as in \cref{GROUPTHEORY_SEC}; however, they make sense in any $G$-Green functor of the form $\Rup$.

We begin with the transfer Mackey ideal.

\begin{defn}
      \label{ITR_DEF}
      Fix a $G\times \Sigma_m$-Green functor $\ul{R}$. Define $\ul{I}_{\Tr}(G/H) \subseteq \Rup(G/H)$ to be the image of the transfers
      \[
                 \bigoplus_{\substack{i+j=m \\ i,j>0}} 
                  \ul{R}\big( (G\times \Sigma_m)/(H \times \Sigma_i \times \Sigma_j ) \big)
                  \indxarr{\Tr} \ul{R}\big((G\times \Sigma_m)/(H\times  \Sigma_m ) \big).
      \]
      We note that the target is $\Rup(G/H)$.
\end{defn}

\begin{lemma}
\label{ITR_MACKEY}
      The ideals $\ul{I}_{\Tr}(G/H)$ of \cref{ITR_DEF} fit together to define a Mackey ideal of $\Rup$.
\end{lemma}

\begin{proof}
      Frobenius reciprocity implies that $\ul{I}_{\Tr}(G/H)$ is an ideal of $\Rup(G/H)$.

      It remains to show that $\ul{I}_{\Tr}$ is a sub-Mackey functor.
      To see that $\ul{I}_{\Tr}$ is closed under restriction maps, note that 
            \[
            \begin{tikzcd}
                  G/H \times  \Sigma_m/(\Sigma_i \times \Sigma_j) \arrow[r] \arrow[d]
                  &
                  G/K \times  \Sigma_m/(\Sigma_i \times \Sigma_j) \arrow[d]
                  \\
                  G/H \times  \Sigma_m/ \Sigma_m  \arrow[r]
                  &
                  G/K \times \Sigma_m/ \Sigma_m 
            \end{tikzcd}
      \]      
      is a  pullback square of $G\times \Sigma_m$-sets and apply \cref{DOUBLECOSET_REM}.
      Finally, $\ul{I}_{\Tr}$ is closed under inductions since the composition of inductions is again an induction.
\end{proof}

Now we will define a Mackey ideal $\ul{J} \subseteq \Rup$, inspired by the ideal $J^G$ of \cref{GROUPTHEORY_SEC}, with the property that $\ul{I}_{\Tr} \subseteq \ul{J}$.

We have a diagonal inclusion of $G\times \Sigma_m$-sets
\[ G/H \xrtarr{\Delta} (G/H)^{\times_{G/L} m}\]
with complementary $G\times \Sigma_m$-set (cf. \eqref{Z})
\[ Z^{L,H}_G = Z^{L,H} = (G/H)^{\times_{G/L} m} \setminus \Delta(G/H).\]
In other words, there is a decomposition of $G\times \Sigma_m$-sets
\begin{equation} \label{IterFiber_Decomp}
(G/H)^{\times_{G/L} m} \ \iso \ G/H \ \amalg \ Z^{L,H}_G,
\end{equation}
where $G/H$ has a trivial $\Sigma_m$-action.
Note that $Z^{G,H}_G$ is what was previously called $Z^{G,H}$ in \cref{Z}.
We will often suppress the subscript $G$ in the notation when there is no likelihood for confusion.
The following description of $Z^{L,H}_G$ will be useful below. 

\begin{prop}
\label{Z_LEM}
The $G\times \Sigma_m$-set $Z^{L,H}_G$ is induced from the subgroup $L\times \Sigma_m $:
\[ Z^{L,H}_G \iso G\times_L \left[ (L/H)^{\times m} \setminus \Delta(L/H)\right] = G\times_L Z^{L,H}_L.\]
\end{prop}

\begin{proof} Since the $G$-set induction functor $G\times_L(-) \colon L\text{-Set} \rtarr G\text{-Set}$ preserves pullbacks, 
it follows that $(G/H)^{\times_{G/L}m}$ is isomorphic to $G\times_L\big((L/H)^{\times m}\big)$.
Since $G\times_L L/H \iso G/H$, applying induction to the $L=G$ case of the decomposition \cref{IterFiber_Decomp}
produces a decomposition
$ G\times_L Z^{L,H}_L \iso \ G/H \amalg G\times_L Z^{L,H}_H.$
\end{proof}

The decomposition \cref{IterFiber_Decomp} induces an isomorphism of commutative rings 
\begin{equation}
      \label{RGHL_EQ}
      \ul{R}\left( (G/H)^{\times_{G/L} m} \right)
      \cong
     \Rup\left( G/H \right) \times \ul{R}(Z^{L,H}).
\end{equation}

We may obtain $\Rup(G/H)/\ul{I}_{\Tr}(G/H)$
 from $\ul{R}((G/H)^{\times_{G/L} m})$ by taking the quotient by
$\ul{I}_{\Tr}(G/H)$ in the first factor of \eqref{RGHL_EQ} and taking the quotient with respect to the entire second factor.
This inspires the definition of $\ul{J}$:

\begin{definition}
      \label{JGL_DEF}
      Fix a $G\times \Sigma_m$-Green functor $\ul{R}$.
      Define $\ul{J}(G/L) \subseteq \Rup(G/L)$ to be
      the ideal generated by the images of the transfers along:
      \begin{itemize}
      \item the quotients $G/L \times \Sigma_m/(\Sigma_i \times  \Sigma_j) \to G/L \times \Sigma_m / \Sigma_m $ for $i + j = m$ and
      \item the composition $Z^{L,H} \into (G/H)^{\times_{G/L} m} \to  G/L \times \Sigma_m / \Sigma_m $ for $H \leq L$.
      \end{itemize}
\end{definition}

By construction, we have the following compatibility.

\begin{prop}\label{JCOMPAT_LEM}
For $H$ subconjugate to $L$, we have the following commutative diagram
\[
    \begin{tikzcd}
          \ul{R}\left( (G/H)^{\times_{G/L} m} \right) \arrow[d, "\Tr",color=\inductioncolor] \arrow[r, twoheadrightarrow]
          &
          \Rup\left( G/H  \right)/\ul{I}_{\Tr}(G/H) \arrow[d, "\Tr",color=\inductioncolor]
          \\
          \Rup(G/L  ) \arrow[r,twoheadrightarrow]
          &
          \Rup(G/L) / \ul{J}(G/L),
    \end{tikzcd}
\]
in which the top map is given by first projecting to the left factor in \eqref{RGHL_EQ} and then taking the quotient by the ideal $\ul{I}_{\Tr}(G/H)$.
\end{prop}

The analogue of \cref{JG_ACTION_THM} in this context reads as follows. Like \cref{JG_ACTION_THM}, it is an immediate consequence of \cref{STABSURJ_PROP}.

\begin{prop}\label{JG_ACTION_NONBOREL_THM}
      Let $\ul{R}$ be a  $G\times \Sigma_m$-Green functor  and  fix $m \geq 1$.
      Then $\ul{J}(G/L) \subseteq \Rup( G/L)$ is  generated by $\ul{I}_{\Tr}(G/L)$ and
      the images of the  transfers 
\[             
             \ul{R}\Big(\bigmod{ G\times \Sigma_m}{\Sigma_q \wr_{\ul{n}} \Gamma(a_{S/K})} \Big)
              \indxarr{\Tr}  \ul{R}(G/L ) = \Rup(G/L)
 \]
      for all $m = nq$, $K < S \leq L$ with $[S : K] = n \neq 1$,
      and $a_{S/K} \colon S \to \Aut_{\Set}(S/K) \cong \Sigma_n$ the action map by left multiplication.
\end{prop}

In the case that $m$ is prime, \cref{PRIMESURJ_COR} gives the following simplified form.

\begin{prop}\label{JG_ACTION_PRIME_NONBOREL_THM}
      Let $\ul{R}$ be a  $G\times \Sigma_m$-Green functor  and let $m=p$ be prime.
      Then $\ul{J}(G/L) \subseteq \Rup (G/L)$ is  generated by $\ul{I}_{\Tr}(G/L)$ and
      the images of the  transfers 
\[
            \ul{R}\Big( \bigmod{ G\times  \Sigma_p}{ \Gamma(a)} \Big)
              \indxarr{\Tr}  \ul{R}(G/L) = \Rup(G/L)
\]
      for all subgroups $S\leq L$ and homomorphisms $a\colon S \rtarr \Sigma_p$ whose images contain a $p$-cycle.
\end{prop}

The proof of the following corollary is the same as for \cref{JGRELPRIME_COR}.

\begin{cor} If $m$ is relatively prime to the order of $G$, then $\ul{J} = \ul{I}_{\Tr}$.
\end{cor}

Now that we have described the ideals $\ul{J}(G/L)$ for fixed $L$, we turn to the question of how they interact as $L$ varies.

\begin{thm}
      \label{JSUBMACKEY_LEM}
      The ideals $\ul{J}(G/L)$ of \cref{JGL_DEF} fit together to define a Mackey  ideal of $\Rup$.
\end{thm}

\begin{proof}
It suffices to show that $\ul{J}$ is a sub-Mackey functor, i.e. that the image of the transfers in \cref{JGL_DEF} is closed under restriction and induction.  
By \cref{ITR_MACKEY}, it suffices to show that 
if $H\leq K \leq L$ (up to conjugacy), then 
\begin{enumerate}
\item\label{ResItem} the image of 
\[ \begin{tikzcd} 
\ul{R}( Z^{L,H}) \ar[r, color=\inductioncolor, "\Tr"] &  \ul{R}(G/L) = \Rup(G/L) \ar[r, "{\mathrm{Res}}"] & \Rup (G/K)
\end{tikzcd}\]
lands in 
$\ul{J}(G/K) \subset \Rup (G/K)$
and
\item\label{IndItem} the image of
\[ \begin{tikzcd} 
\ul{R}( Z^{K,H}) \ar[r, color=\inductioncolor,"\Tr"] &  \ul{R}(G/K)=\Rup (G/K) \ar[r, color=\inductioncolor, "\Tr"]  & \Rup (G/L) 
\end{tikzcd}\]
lands in 
$\ul{J}(G/L) \subset \Rup (G/L)$.
\end{enumerate}

We begin with \eqref{IndItem}, as it is much simpler to verify.
We have a commutative diagram of $G\times \Sigma_m$-sets
\[\begin{tikzcd}
Z^{K,H} \ar[r, hookrightarrow] \ar[d, dashed, hookrightarrow] & 
G/H^{\times_{G/K} m} \ar[d, hookrightarrow] \ar[r, twoheadrightarrow] & 
G/K \ar[d,twoheadrightarrow] \\
Z^{L,H} \ar[r, hookrightarrow]  & 
G/H^{\times_{G/L} m} \ar[r, twoheadrightarrow] &
G/L,
\end{tikzcd}\]
which yields the commutative diagram
\[ \begin{tikzcd}
\ul{R}(Z^{K,H})  \ar[r,  color=\inductioncolor, "\Tr"] \ar[d, hookrightarrow] & 
\ul{R}(G/K)  \ar[d, color=\inductioncolor, "\Tr"] \ar[r,equal] & \Rup(G/K) \ar[d, color=\inductioncolor, "\Tr"] \\
\ul{R}(Z^{L,H})  \ar[r,  color=\inductioncolor, "\Tr"] &
\ul{R}(G/L) \ar[r,equal] & \Rup(G/L). 
\end{tikzcd}\]
It follows that $\ul{J}$ is closed under Mackey induction.
      
We now turn to \eqref{ResItem}, which is more difficult to handle.
      It suffices to show that we have a factorization
      \begin{equation}
            \label{JSUBMACK_EQ}
            \begin{tikzcd}
                  \displaystyle{
                    \bigoplus_{H < L} \ul{R}\left( Z^{L,H}  \right) 
                }
                  \arrow[d, dashed] \arrow[r, color=\inductioncolor, "\Tr"]
                  &
                  \Rup ( G/L) \arrow[d, "\mathrm{Res}"]
                  \\
                  \displaystyle{
                    \bigoplus_{i+j=m}\ul{R}\Big(G/K \times \Sigma_m/ (\Sigma_i \times  \Sigma_j) \Big) \oplus
                    \bigoplus_{H' < K} \ul{R}\left( Z^{K,H'} \right) 
                  }
                  \arrow[r, color=\inductioncolor, "\Tr"]
                  &
                  \Rup ( G/K),
            \end{tikzcd}
      \end{equation}
as the image of the bottom horizontal map lies in $\ul{J}(G/K)$.
      
We have a  pullback diagram of $G\times \Sigma_m$-sets 
      \begin{equation}
            \label{JSUBMACKGPD_EQ}
            \begin{tikzcd}
                  {Z^{L,H}\!}  
                  \arrow[r]
                  &
                  {G/L}
                  \\
                  (Z^{L,H} \times_{G/L} G/K)
                  \arrow[r] \arrow[u,"q"]
                  &
                  G/K \arrow[u,"p",swap]
            \end{tikzcd}
      \end{equation}     
      in which all maps are surjective. 
      Then \cref{DOUBLECOSET_REM} gives a commutative diagram
\[
            \begin{tikzcd}
                  \displaystyle{
                   \ul{R}\left( Z^{L,H}  \right) 
                  }
                  \arrow[d,"q^*",swap] \arrow[r, color=\inductioncolor, "\Tr"]
                  &
                  \Rup ( G/L) \arrow[d, "\mathrm{Res}=p^*"]
                  \\
		\ul{R}\big( Z^{L,H}\times_{G/L}G/K \big)
                  \arrow[r, color=\inductioncolor, "\Tr"]
                  &
                  \Rup ( G/K).
            \end{tikzcd}
\]       
It remains to show that the bottom transfer map factors through the sum
\[  \bigoplus_{i+j=m}\ul{R}(G/K \times \Sigma_m/(\Sigma_i \times  \Sigma_j) ) \oplus
                    \bigoplus_{H' < K} \ul{R}\left( Z^{K,H'}  \right).
\]
We may decompose $Z^{L,H}\times_{G/L} G/K$ into $G\times \Sigma_m$-orbits, 
and it suffices to produce the factorization at the level of $G\times \Sigma_m$-sets on each orbit. 

Thus let $U\subset Z^{L,H}\times_{G/L}G/K$ be such an orbit, and choose $(\vect{x},y)\in U$. Then if $\Lambda\leq G\times \Sigma_m$ is the stabilizer of $(\vect{x},y)$, the presence of the factor $G/K$ forces $\pi_G(\Lambda)$ to be subconjugate to $K$. We now consider two cases.

In the first case, suppose that $\pi_{\Sigma_m}(\Lambda)$ is not a transitive subgroup. Then $\pi_{\Sigma_m}(\Lambda)$ is subconjugate to $\Sigma_i\times \Sigma_j$ for some positive $i$ and $j$ summing to $m$. It follows that there exists a map of $G\times \Sigma_m$-sets of the form
\[
      \begin{tikzcd}[row sep = tiny]
            U \arrow[r, "\iso"]
            &
            (G\times \Sigma_m)/\Lambda \arrow[r]
            &
            (G\times \Sigma_m )/(K \times \Sigma_i \times \Sigma_j )
            \\
            (\vect x, y) \arrow[r, mapsto]
            &
            e\Lambda \arrow[r, mapsto]
            &
            (g,\sigma)(K \times \Sigma_i \times \Sigma_j )
      \end{tikzcd}
\]
where $y=gK$. Then composing this map with the projection onto $G/K$ produces a map of $G\times \Sigma_m$-sets sending $(\vect{x},y)$ to $gK=y$. It follows that the image of the $\ul{R}$-transfer along 
$U \rtarr G/K$ is contained in the image of the $\ul{R}$-transfer along the $G$-cover $G/K \times  \Sigma_m/(\Sigma_i\times \Sigma_j) \rtarr G/K$.

In the second case, we suppose that $\pi_{\Sigma_m}(\Lambda)$ is a transitive subgroup of $\Sigma_m$.
We further assume for simplicity that $H$ and $K$ are subgroups of $L$, rather than merely subconjugate. 
The general case is similar but notationally more complex.

We now reduce to the case $L=G$: 
recall from \cref{Z_LEM} that the $G\times \Sigma_m$-set $Z^{L,H}_G$ is induced up from the subgroup $L \times  \Sigma_m $.
Since the $G$-set induction $G\times_L(-) \colon L\Set \rtarr G\Set$ preserves pullbacks, the pullback square of $L$-sets
\[
\begin{tikzcd}
Z^{L,H}_L \times L/K \ar[r] \ar[d]
& L/K \ar[d] \\
Z^{L,H}_L \ar[r] & L/L
\end{tikzcd} 
\]
gives rise to an isomorphism of $G\times \Sigma_m$-sets
\[ Z^{L,H}_G\times_{G/L} G/K \iso G\times_L \left( Z^{L,H}_L \times L/K\right).\]
Moreover, the projection $Z^{L,H}_G\times_{G/L} G/K \rtarr G/K$ is the induction from $L$ to $G$ of the projection $Z^{L,H}_L\times L/K \rtarr L/K$. 
We may therefore restrict to the case $L=G$.

We now return to the $G\times \Sigma_m$-orbit $U$ with chosen point $(\vect{x},y)$ and $\Lambda=\Stab_{G\times \Sigma_m}(\vect{x},y)$. 
Up to $\Sigma_m$-conjugacy, the tuple $\vect{x}$ is a $q$-fold shuffle.
We have assumed that $\pi_{\Sigma_m}(\Lambda) \leq \Sigma_m$ is transitive. Let $y=gK$. Then
\[ \Lambda = \Stab(\vect{x},y) = \Stab(\vect{x}) \cap \Stab(y) = \Stab(\vect{x}) \cap \big(gKg^{-1} \times \Sigma_m\big).\]
\cref{GRAPHSTAB_LEM} explicitly describes $\Stab(\vect{x})$ as $ \Sigma_q \wr_{\underline{n}} \Gamma(a)$ for some homomorphism $a\colon S \rtarr \Sigma_n$.
The intersection $\Lambda = \Stab(\vect{x})\cap (gKg^{-1}\times \Sigma_m)$ is now 
$\Sigma_q \wr_{\underline{n}} \Gamma(a\mid_{S\cap gKg^{-1}})$, where $a\mid_{S\cap gKg^{-1}}$ is the restriction of $a$ to $S\cap gKg^{-1}$. This is a 
subgroup of $ gKg^{-1} \times \Sigma_m $ that projects onto a transitive subgroup of $\Sigma_m$ by assumption. We may use the surjectivity statement of \cref{STABSURJ_PROP}  to describe this as a stabilizer of some element of 
\[ (gKg^{-1}/H')^{\times m} \setminus \Delta(gKg^{-1}/H')\]
for some $H'\leq gKg^{-1}$. It follows that $U \iso (G\times \Sigma_m)/\Lambda$ appears as an orbit of $Z^{gKg^{-1},H'} \iso Z^{K,g^{-1}H'g}$.
Therefore the image of the $\ul{R}$-transfer from $U$ is contained in the image of the $\ul{R}$-transfer from $Z^{K,g^{-1}H'g}$.
\end{proof}

\begin{rem}
By construction, a map of $G\times  \Sigma_m $-Green functors $\ul{R} \rtarr \ul{S}$ gives rise to a map of $G$-Green functors $\Rup/\ul{J} \rtarr \up{\ul{S}}/\ul{J}$.
\end{rem}

\subsection{Power operations on Green functors}
\label{GREEN_POWER_OPS_SEC}

For any $m\geq 0$, let 
\[\Frob\colon G\text{-Set} \rtarr G\times\Sigma_m\text{-Set}\] 
 be the $m$th power functor $\Frob(A) = A^{\times m}$, where $\Sigma_m$ permutes the factors and $G$ acts diagonally. 
Note that $\Frob$ preserves pullbacks, but not coproducts. 
Given a $G\times\Sigma_m$-Green functor $\ul{R}$, 
we will refer to
the composition $\ul{R}\circ \Frob$ as a ``non-additive $G$-Green functor.''
This means that it is a contravariant functor from $G$-Set to commutative rings 
admitting transfers along surjections, satisfying Frobenius reciprocity as in \cref{defn:GGreen}, 
and satisfying the push-pull property as in \cref{DOUBLECOSET_REM}.
It will not, however, necessarily send coproducts to direct sums.

Following \cref{ITR_DEF}, let $\ul{\bI}_{\Tr} \subseteq \ul{R}\circ \Frob$ be the  ideal 
defined by letting $\ul{\bI}_{\Tr}(A)$ be the image of the transfers
      \[
                 \bigoplus_{\substack{i+j=m \\ i,j>0}} 
                  \ul{R}\big(A^{\times m} \times \Sigma_m/(\Sigma_i \times \Sigma_j) \big)
                  \indxarr{\Tr} \ul{R}\big(A^{\times m} \times \Sigma_m/\Sigma_m \big).
      \]
Note that the target is $\ul{R}\circ \Frob(A)$.

\begin{defn}
      \label{MTP_DEF}
      An $m$th \textit{total power operation} on a $G\times \Sigma_m$-Green functor $\ul{R}$ 
is a natural transformation 
\[
\bP_m\colon \justdownwardarrow{\ul{R}}{G}{G \times \Sigma_m} \rtarr \ul{R}\circ \Frob
\]
as functors to Set
which preserves transfers
as well as
the multiplicative structure and such that $\bP_m/\ul{\bI}_{\Tr}$ is a map of
non-additive $G$-Green functors.
\end{defn}

We will consider two sources of $G \times \Sigma_m$-Green functors with $m$th total power operation in the following two subsections. 
These are $H_{\infty}$-rings in genuine $G$-spectra (\cref{GPOWER_OPS_SEC}) and  $G_\infty$-ring spectra in the sense of \cite[Remark~5.1.16]{global} (\cref{GLOBAL_OPS_SEC}).

For any $G$-set $A$, the diagonal inclusion $\Delta\colon A \into A^{\times m} = \Frob(A)$ is $G\times \Sigma_m$-equivariant. Pulling back along $\Delta$ defines a map of coefficient systems of commutative rings
\[ \ul{R} \circ \Frob \xrtarr{\Delta^*} 
\justupwardarrow{\ul{R}}{G\times \Sigma_m}{G}.
\]
The image of $\ul{\bI}_{\Tr}$ under $\Delta^*$ is $\ul{I}_{\Tr}$ (\cref{ITR_DEF}) by the double coset formula.
We then define the power operation $P_m$ as the composition
\[ P_m \colon \justdownwardarrow{\ul{R}}{G}{G \times \Sigma_m} \xrtarr{\bP_m} \ul{R}\circ \Frob
\xrtarr{\Delta^*} 
\justupwardarrow{\ul{R}}{G\times \Sigma_m}{G}.
\]
In general, the power operation $P_m$ is not additive, and it does not commute with the Mackey induction maps and
is therefore not a map of Mackey functors. Making use of \cite[VIII.1.4]{BMMS} and \cref{ITR_MACKEY}, additivity can be arranged by taking the quotient with respect to $\ul{I}_{\Tr}$.

\begin{prop}
      \label{Green_ITr_Thm}
      Let $\ul{R}$ be a $G\times\Sigma_m$-Green functor with an $m$th total power operation.  Then the composition
      \[
      \justdownwardarrow{\ul{R}}{G}{G \times \Sigma_m} \xrtarr{P_m}
      \justupwardarrow{\ul{R}}{G\times \Sigma_m}{G}
            \twoheadrightarrow
            \justupwardarrow{\ul{R}}{G\times \Sigma_m}{G}/\ul{I}_{\Tr}
      \]
      is a map of 
      coefficient systems of commutative rings. 
\end{prop}

Further passing to the quotient with respect to $\ul{J}$ produces a map of $G$-Green functors.

\begin{theorem}
\label{Green_Main_Thm}
      Let $\ul{R}$ be a $G\times \Sigma_m$-Green functor with an $m$th total power operation, and 
      let $\ul{J}$ be as in \cref{JGL_DEF}.
      Then 
      the reduced $m$th power operation
      \[
      P_m/\ul{J} \colon 
           \justdownwardarrow{\ul{R}}{G}{G \times \Sigma_m}
            \xrightarrow{\bP_m} \ul{R}\circ\Frob
            \xrightarrow{\Delta^*}
            \justupwardarrow{\ul{R}}{G\times \Sigma_m}{G}
            \twoheadrightarrow
            \justupwardarrow{\ul{R}}{G\times \Sigma_m}{G}/\ul{J}
      \]
      is a map of  $G$-Green functors.
\end{theorem}

The argument below follows the strategy outlined in \cref{OVERVIEW_SEC}.

\begin{proof}
      By \cref{JSUBMACKEY_LEM}, $\ul{J}$ is a Mackey functor ideal in $ \justupwardarrow{\ul{R}}{G\times \Sigma_m}{G}$.
      Thus $ \justupwardarrow{\ul{R}}{G\times \Sigma_m}{G}/\ul{J}$ is a  $G$-Green functor.
      Since $\ul{I}_{\Tr}$ is contained in $\ul{J}$, \cref{Green_ITr_Thm} implies that 
      it remains to show that $P_m/\ul{J}$ commutes with induction maps. Thus suppose that $H\leq G$ is subconjugate to $L\leq G$.
      
      Since the total power operation $\bP_m$ is a map of Mackey functors (of sets),
\cref{PUSHPULL_COR_TWO}  implies that we have the following commuting diagram: 
\begin{equation}
      \label{POWERTRANS_EQ_2}
      \begin{tikzcd}
            \downwardarrow{\ul{R}}{G}{G \times \Sigma_m}{G/H} \arrow[d, color=\inductioncolor, "\Tr"'] \arrow[r, "\bP_m"]
            &
             \ul{R}\left(  (G/H)^{\times m} \right) \arrow[r,"i^*"] \arrow[d, color=\inductioncolor, "\Tr"']
            &
             \ul{R}\left( (G/H)^{\times_{G/L} m} \right) \arrow[d, color=\inductioncolor, "\Tr"]
            \\
           \downwardarrow{\ul{R}}{G}{G \times \Sigma_m}{G/L} \arrow[r, "\bP_m"]
            &
             \ul{R}\left(  (G/L)^{\times m} \right) \arrow[r, "\Delta^{\**}"]
            &
            \ul{R}\left( G/L \right).
      \end{tikzcd}
\end{equation}
Note that in the bottom right corner, $G/L$ has a trivial $\Sigma_m$-action, 
so that $\ul{R}(G/L)$ is $ \upwardarrow{\ul{R}}{G\times \Sigma_m}{G}{G/L} $.
\cref{JCOMPAT_LEM} states that we have a commuting diagram
      \[
            \begin{tikzcd}
                  \ul{R}\left( (G/H)^{\times_{G/L} m} \right) \arrow[d, "\Tr",color=\inductioncolor] \arrow[r, twoheadrightarrow]
                  &
                    \upwardarrow{\ul{R}}{G\times \Sigma_m}{G}{G/H}  /\ul{I}_{\Tr}(G/H) \arrow[d, "\Tr",color=\inductioncolor]
                  \\
                 \upwardarrow{\ul{R}}{G\times \Sigma_m}{G}{G/L} \arrow[r,twoheadrightarrow]
                  &
                  \upwardarrow{\ul{R}}{G\times \Sigma_m}{G}{G/L}  / \ul{J}(G/L).
            \end{tikzcd}
      \]
      The result then follows by the factorization
         \[
            \begin{tikzcd}
                 \upwardarrow{\ul{R}}{G\times \Sigma_m}{G}{G/H} / \ul{I}_{\Tr}(G/H) \arrow[d, "\Tr",color=\inductioncolor] \ar[r,twoheadrightarrow]
                 &
                \upwardarrow{\ul{R}}{G\times \Sigma_m}{G}{G/H}  /\ul{J}(G/H) \arrow[d, "\Tr",color=\inductioncolor]
                  \\
                 \upwardarrow{\ul{R}}{G\times \Sigma_m}{G}{G/L}  / \ul{J}(G/L) \ar[r,equal] 
                  & 
                 \upwardarrow{\ul{R}}{G\times \Sigma_m}{G}{G/L}  / \ul{J}(G/L),
            \end{tikzcd}
      \]
      which occurs since $\ul{J}$ is a Mackey ideal containing $\ul{I}_{\Tr}$.   
\end{proof}

\subsection{Power operations on $G$-spectra}
\label{GPOWER_OPS_SEC}

We first consider $H_\infty$-ring spectra, in the sense of \cite{BMMS}, in the equivariant category. 

\begin{definition}
\label{HRING_DEF}
      An \textit{$H_\infty$-ring $G$-spectrum} is a $G$-spectrum $E$ equipped with 
      $G$-equivariant maps $\mu \colon E^{\wedge m}_{h\Sigma_m} \rtarr E$ which make the diagrams of \cite[I.3]{BMMS}
      commute in the $G$-equivariant stable homotopy category.
\end{definition}

We recall that the structure of an $H_\infty$-ring on a (nonequivariant) spectrum $E$ is the vestige in the homotopy category of a point-set $E_\infty$-structure on $E$. 
Said differently, let us write $\mathrm{Sym}$ for the free $E_\infty$-algebra monad on spectra. This descends to a monad $h\mathrm{Sym}$ on the homotopy category, and an $H_\infty$-ring spectrum is an $h\mathrm{Sym}$-algebra.
Similarly so if $E$ is a $G$-spectrum, where now `$E_\infty$' is taken in the non-equivariant sense, meaning that $G$ acts trivially on the universal spaces $E\Sigma_m$  of the $E_\infty$ operad.

As noted in \cref{ringGGreen}, every homotopy-commutative ring $G$-spectrum $E$ induces a Green functor-valued equivariant cohomology theory on $G$-spaces,
defined by
\[
      \ul{E}^0(X)(G/H) = E^0(G/H \times X) = [( G/H \times X)_+,  E]^G,
\]
where $[-,-]^G$ denotes the abelian group of maps in the
equivariant stable homotopy category.

If $E$ is moreover an $H_\infty$-ring $G$-spectrum, more is true:

\begin{proposition}
\label{GRing_PowerOps}
      If $E$ is an $H_\infty$-ring $G$-spectrum and $X$ is a $G$-space, then
      the $G \times \Sigma_m$-Green functor given by 
      $\ul{R}(A) = E^0(X\times A_{h\Sigma_m})$,
      as in \cref{BorelGSigma},
      has an $m$th total power operation.
\end{proposition}

\begin{proof}
In this case, 
\[
 \downwardarrow{\ul{R}}{G}{G \times \Sigma_m}{G/H} = E^0(X\times G/H)
 \]
 and
 \[
 \ul{R}\circ\Frob(G/H) = E^0(X\times (G/H)^{\times m}_{h\Sigma_m}).\
 \]
      We define $\bP_m$ levelwise to be the composite
\begin{align*}
   [(X \times G/H )_+, E]^G
                          \to 
                          [(X \times G/H  )^{\times m}_{h \Sigma_m,_+}, E^{\wedge m}_{h \Sigma_m}]^G
                         & \xrightarrow{\mu}
       [(X \times G/H)^{\times m}_{h \Sigma_m,+}, E]^G
  \\ & 
       \xrightarrow{\Delta_{X}^{\**}}
       [ (X\times (G/H)^{\times m}_{h \Sigma_m})_+,E]^G,
\end{align*}
where $\Delta_X\colon X \times (G/H)^{\times m} \to X^{\times m} \times (G/H)^{\times m} \cong (X\times G/H)^{\times m}$ is the $G$-equivariant map induced by the diagonal on $X$.
The map $\bP_m$ satisfies the requirements of \cref{MTP_DEF}:
      it is natural in all stable maps by construction, and so in particular is a map of $G$-Mackey functors of sets;
      it is multiplicative; and
      it is additive after passing to the quotient by $\ul{\bI}_{\Tr}$ by  \cite[VIII.1.1]{BMMS}.
\end{proof}

\begin{rem}
\label{RX_REM}
      We note that the function spectrum
      $E^X$ is an $H_\infty$-ring $G$-spectrum whenever $E$ is an $H_\infty$-ring $G$-spectrum and $X$ is a $G$-space.\footnote{This is false if we replace $X$ with a $G$-spectrum $Y$, as the diagonal map $X \to X^{\times n}$ plays a key role in this structure.}
      Thus the power operation $\bP_m$ for $E$ defined at $X$ in \cref{GRing_PowerOps} 
      agrees with the power operation $\bP_m$ for $E^X$ at $G/G$.
      Thus, without loss of generality we may assume $X = G/G$ throughout.
\end{rem}

The target of the power operation $P_m$ is $\underline{E}^0(B \Sigma_m)$ \eqref{BorelGInd_EQ}.
Thus, \cref{JGL_DEF} yields an ideal
$\ul{J}(G/L) \subseteq E^0(G/L \times B \Sigma_m)$
generated by the images of the transfers along:
\begin{itemize}
\item the covers $G/L \times B \Sigma_i \times B \Sigma_j \to G/L \times B \Sigma_m $ for $i + j = m$ and
\item the composition $Z^{L,H}_{h\Sigma_m} \into (G/H)^{\times_{G/L} m}_{h\Sigma_m} \to  G/L \times  B \Sigma_m$ for $H \leq L$.
\end{itemize}

The following is then an immediate corollary of \cref{Green_Main_Thm}.

\begin{corollary}
\label{GEquiv_Main_Thm}
      For any $H_\infty$-ring $G$-spectrum $E$,
      the reduced $m$th power operation
      \[
         P_m/\ul{J}\colon   \underline{E}^0 \xrightarrow{P_m} \underline{E}^0(B \Sigma_m) \twoheadrightarrow \underline{E}^0(B \Sigma_m)/\underline{J}
    \]
    is a map of $G$-Green functors.
\end{corollary}

\begin{remark}
By \cref{RX_REM}, for any $G$-space $X$, we get a map of $G$-Green functors
\[ P_m/\ul{J}\colon \ul{E}^0(X) \rtarr \ul{E}^0( X \times B\Sigma_m )/\ul{J},\]
where $\ul{J} \subset \ul{E}^0( X \times B\Sigma_m )$ is defined as above by crossing with $X$.
\end{remark}

\subsection{Power operations on global spectra}
\label{GLOBAL_OPS_SEC}

Many examples of a $G \times \Sigma_m$-Green functor with an $m$th total power operation arise in global stable homotopy theory. For instance, equivariant $K$-theory and equivariant bordism theories are examples of $G_{\infty}$-ring spectra (see \cite[Chapter 6]{global}), which is the global analogue of an $H_{\infty}$-ring spectrum, and these give rise to 
$G\times \Sigma_m$-Green functors
for each choice of $G$ and $m \geq 1$ .

Let $X$ be a global spectrum for the family of finite groups in the sense of \cite{global}. There is an underlying genuine $G$-spectrum $X_G$ associated to $X$ for each finite group $G$. Given a based $G$-space $Y$, the $G$-equivariant zeroeth $X$-cohomology of $Y$ is 
\[
X_{G}^{0}(Y) = [Y,X_{G}]^G.
\]
This gives the $G$-Mackey functor $\ul{X}_{G}^{0}$ with value
\[
\ul{X}_{G}^{0}(A) = [A_+,X_{G}]^G,
\]
when $A$ is a finite $G$-set. A homomorphism $\alpha \colon G' \to G$ induces a restriction map 
\[
\alpha^* \colon [Y,X_{G}]^G \to [Y,X_{G'}]^{G'},
\]
in which $Y$ is viewed as a pointed $G'$-space through $\alpha$ in the target.

If $E$ is a homotopy commutative global spectrum, then $E_G$ is a homotopy commutative genuine $G$-spectrum for each $G$ (see the discussion leading up to \cite[Theorem 4.5.24]{global}). It follows that $\ul{E}_{G}^{0}$ is a $G$-Green functor. 

More refined is the notion of a $G_{\infty}$-ring spectrum (see \cite[Remark~5.1.16]{global} or \cite{G}), which is the global analogue of an $H_{\infty}$-ring spectrum. This structure endows cohomology with multiplicative total power operations (\cite[Remark 5.1.14]{global}). Given a pointed $G$-space $Y$, these are multiplicative operations
\begin{equation} \label{globalpower}
[Y,E_G]^G \to [Y^{\wedge m}, E_{G \wr \Sigma_m}]^{G \wr \Sigma_m},    
\end{equation}
for each $m \geq 1$.

\begin{proposition}
\label{UltraComm_PowerOps}
      If $E$ is a $G_\infty$-ring spectrum, then
      the $G \times \Sigma_m$-Green functor given by 
      \[
      \ul{R}(A) = \ul{E}^{0}_{G\times \Sigma_m}(A) = [A_+, E_{G \times \Sigma_m}]^{G \times \Sigma_m},
      \]
      for a $G \times \Sigma_m$-set $A$,
      has an $m$th total power operation.
\end{proposition}
\begin{proof}
Notice that $\downwardarrow{\ul{R}}{G}{G\times\Sigma_m}{G/H} = [(G/H)_+,E_G]^G$ and that 
\[
(\ul{R} \circ F^m)(G/H) = [((G/H)^{\times m})_+, E_{G \times \Sigma_m}]^{G \times \Sigma_m}.
\]

Taking $Y = (G/H)_+$ (as a $G$-set) in \eqref{globalpower} gives the total power operation
\[
[(G/H)_+,E_G]^G \to [((G/H)^{\times m})_+, E_{G \wr \Sigma_m}]^{G \wr \Sigma_m}.
\]
Restricting along $\alpha \colon G \times \Sigma_m \to G \wr \Sigma_m$ then gives an $m$th total power operation in the sense of \cref{MTP_DEF}:
\[
\bP_m \colon \downwardarrow{\ul{R}}{G}{G\times\Sigma_m}{G/H} = [(G/H)_+,E_G]^G \to [((G/H)^{\times m})_+, E_{G \times \Sigma_m}]^{G \times \Sigma_m} = (\ul{R} \circ F^m)(G/H).
\]
\end{proof}

Continuing to use the notation of \cref{UltraComm_PowerOps}, we have that 
\[
\upwardarrow{\ul{R}}{G \times \Sigma_m}{G}{G/H} = \ul{E}_{G \times \Sigma_m}^{0}((G\times \Sigma_m)/(H\times \Sigma_m)).
\]
Now \cref{JG_ACTION_NONBOREL_THM} gives an explicit description of $\ul{J}(G/H) \subseteq \upwardarrow{\ul{R}}{G \times \Sigma_m}{G}{G/H}$.

The following is then an immediate corollary of \cref{Green_Main_Thm}.

\begin{cor}
\label{Global_Main_Thm}
      Let   $E$ be a $G_\infty$-ring spectrum and let $\ul{R}$ be the $G \times \Sigma_m$-Green functor given by 
      $\ul{R}(A) = \ul{E}_{G \times \Sigma_m}^{0}(A)$, where $A$ is a $G \times \Sigma_m$-set.
      Then the reduced $m$th power operation
      \[
\justdownwardarrow{\ul{R}}{G}{G\times\Sigma_m}{} \to \justupwardarrow{\ul{R}}{G \times \Sigma_m}{G}{} \to \justupwardarrow{\ul{R}}{G \times \Sigma_m}{G}{}/\ul{J}
      \]
      is a map of  $G$-Green functors.
\end{cor}

\begin{rem}
A $G_{\infty}$-ring spectrum gives rise to a global power functor in the sense of both  \cite{Ganter} and \cite[Definition 5.1.6]{global}. A proof of this fact is given in \cite[Theorem 5.1.11]{global}. It is possible to formulate the discussion in this section in these terms.
\end{rem}


\setcounter{MaxMatrixCols}{20}

\setcounter{secnumdepth}{3}

\section{Examples}\label{ExamplesSection}

In this section, we calculate the ideal $\ul{J}$ in a number of examples.
In certain cases, we identify $\ul{J}$ for all finite groups $G$.
Throughout this section, Mackey induction homomorphisms will be displayed in 
\textcolor{\inductioncolor}{\inductioncolorname},
whereas power operations will be displayed in \textcolor{\powercolor}{\powercolor}.
We will deal with many global Green functors in this section, and we will abbreviate an induction such as 
$\justupwardarrow{(\ul{A}_{G\times \Sigma_m})\hspace{-0.1ex}}{G\times \Sigma_m}{G}$ to
$\MackUp{A}$.

\subsection{Ordinary cohomology}
\label{sec:HG}

We begin with ordinary cohomology. This case turns out to be degenerate, in the sense that $\ul{J}$ is equal to $\ul{I}_{\Tr}$.

Let $\ul{R}$ be a $G$-Green functor and $H\ul{R}$ the $G$-equivariant Eilenberg-Mac~Lane spectrum.
We will show
that the composition
\[ \mf{R} = \rH \ul{R}^0 \powerxarr{P_m} \rH \mf{R}^0(B\Sigma_m) \rtarr \rH \mf{R}^0(B\Sigma_m)/\mf{I}_{\Tr}\]
is already a map of Mackey functors in the case of ordinary cohomology, and thus $\ul{J}=\ul{I}_{\Tr}$.
At a $G$-orbit $G/K$, the $m$th power operation $P_m$ is $\ul{R}(G/K) \powerxarr{(-)^m} \ul{R}(G/K)$.

\begin{lemma} If $m=p^r$, where $p$ is prime, then for any $K\subset G$, the ideal $\mf{I}_{\Tr}(G/K)\subset \ul{R}(G/K)$ is the principal ideal $(p)$.
      On the other hand, if $m$ has multiple prime factors, then $\mf{I}_{\Tr}(G/K)=(1)$.
\end{lemma}

\begin{proof}
Let $m=p^r$. We will abbreviate $\mf{I}_{\Tr}(G/K)$ simply to $I_{\Tr}$.
For $i+j=m$, the cover $B\Sigma_i \times B \Sigma_j \rtarr B\Sigma_m$ induces a transfer map 
\[\rH\mf{R}^0(B\Sigma_i\times B\Sigma_j)(G/K) \iso \ul{R}(G/K) \indxarr{\Tr}  \rH\mf{R}^0(B\Sigma_m)(G/K) \iso \ul{R}(G/K), \]
which is multiplication by the index of the cover,  the binomial coefficient $\binom{m}{i}=\frac{m!}{i!\,j!}$. 
Since $0<i<p^r$, all of these coefficients are multiples of $p$, and we conclude that $I_{\Tr} \subset (p)$. 
Taking $i=1$ shows that $p^r\in I_{\Tr}$. In the case $i=p^{r-1}$,  the coefficient $\binom{p^r}{p^{r-1}}$ is congruent to $p$ modulo $p^r$. 
It follows that $p \in I_{\Tr}$.

Suppose, on the other hand, that $m$ is divisible by distinct primes $p_i$. 
By Lucas' Theorem, if $r_i$ is the largest integer such that $p_i^{r_i}$ divides $m$, then $\binom{m}{p_i^{r_i}}$ is prime to $p_i$. It follows that the collection of coefficients $\{ \binom{m}{p_i^{r_i}}\}$ are relatively prime to $\binom{m}1=m$, and therefore generate the ideal $(1)$.
\end{proof}

Thus in the interesting case $m=p^r$, the composition $\ul{R} \powerxarr{\,\,\,(-)^m}\ul{R} \rtarr \ul{R}/\mf{I}_{\Tr}$ 
can be expressed as the composition of two maps of $G$-Green functors: the quotient map $\ul{R} \powerarr \ul{R}/p$ followed by the $p^r$-power map.

\subsection{The  sphere spectrum}
\label{sec:sphere}

We start with the global sphere spectrum $\bS$, which is a
$G_\infty$-ring spectrum \cite[Example~4.2.7]{global}.
Recall that 
$\ul{\pi}^0_G=\ul{\bS}^0_G$
 is isomorphic to the Burnside ring Mackey functor, $\mf{A}_G$ \cite[Corollary to Proposition~1]{Seg}.
Thus 
\[\mpi^0_G(G/H) \iso \ul{A}_G(G/H) \iso A(H)\]
 is the commutative ring of isomorphism classes of finite $H$-sets. The restriction and induction maps, in the case that $H \leq K$, are given by considering a finite $K$-set as a finite $H$-set, on the one hand, and by inducing an $H$-set up to a $K$-set, on the other. 

The $m$th power operation associated to the $G_\infty$-ring spectrum $\bS$ as in \cref{GLOBAL_OPS_SEC}  takes the form
\[P_m \colon \mf{A}_G \powerarr \MackUp{A},\]
 where we recall that $\MackUp{A}$ is the $G$-Green functor given by $G/H \mapsto A(H\times \Sigma_m)$.
 On the other hand, the $m$th power operation associated to the $H_\infty$-ring $G$-spectrum $\bS^0_G$ takes the form
\[ P_m\colon \ul{A}_G \powerarr \ul{\pi}^0_G(B\Sigma_m)\]
as in \cref{GPOWER_OPS_SEC}. 
Restriction along the map of $G\times\Sigma_m$-spaces $E\Sigma_m\rtarr \ast$, in which $G$ acts trivially on $E\Sigma_m$, induces a map of $G$-Green functors $\MackUp{A} \rtarr \ul{\pi}^0_G(B\Sigma_m)$
 which is a completion map in the sense that, at an orbit $G/H$, 
 \[ \ul{\pi}^0_G(B\Sigma_m)(G/H) \iso \pi^0_H(B\Sigma_m) \iso A(\Sigma_m\times H)^{\wedge}_{I_{\Sigma_m}} \iso \MackUpOrbit{A}{G/H}^{\wedge}_{I_{\Sigma_m}},\]
 where $I_{\Sigma_m}$ is the augmentation ideal of $A(\Sigma_m)$.
This observation appears without proof in \cite[Chapter~XX]{alaska};
we provide a proof for completeness.
Recall that for groups $L$ and $W$, the Burnside module
\[A(L,W)\subset A(L\times W)\]
 is free abelian on $(L\times W)$-sets for which the action of $W$ is free.

\begin{prop} For any groups $H$ and $L$, 
there is an isomorphism
\[ A(L \times H) \iso \bigoplus_{[K \subset H]}  A(L,W_{H}(K))\]
of $A(L)$-modules,
where the sum runs over conjugacy classes of subgroups.
This induces an isomorphism of commutative rings
\[ \pi_H^0(BL) \iso A(L\times H)^{\wedge}_{I_L}, \]
where $I_L$ is the augmentation ideal of $A(L)$.
\end{prop}

\begin{proof}
We define an $A(L)$-module map
\[ \Phi \colon A(L \times H) \rtarr \bigoplus_{[K \subset H]}  A(L,W_{H}(K))\]
by 
\[ \Phi\big( (L\times H)/\Gamma\big) = \big(L\times W_H(K_\Gamma)\big)/(\Gamma/K_\Gamma),\]
where $K_\Gamma = \Gamma \cap H$.
Writing $p_H\colon L\times H\rtarr H$ for the projection,
we have $K_\Gamma \trianglelefteq p_H(\Gamma)$, so that $\Gamma/K_\Gamma$ is a subgroup of $L\times W_H(K_\Gamma)$. 
For the reverse direction, we send a $(L\times W_H(K))$-set $Y$ to $Y\times_{W_H(K)} H/K$.
Here, $L$ acts on $Y$, and the quotient (coequalizer) is formed in the category of $H$-sets, where $H$ is acting trivially on $Y$.
This assignment is inverse to $\Phi$, and 
it follows that $\Phi$ is an isomorphism of $A(L)$-modules.
We therefore deduce an isomorphism
\[
 A(L \times H)^{\wedge}_{I_L} \iso \bigoplus_{[K \subset H]}  A(L,W_{H}(K))^{\wedge}_{I_L}
\]
upon completion.

Now consider the ring map $A(L\times H) \rtarr \pi_H^0(BL)$ defined by sending the $(L\times H)$-set $X$ to the composition 
\[ \Sigma^\infty_H(BL)_+\indxarr{\Tr}
\Sigma^\infty_H (X_{hL} )_+\rtarr S^0_H.
\]
This ring map factors 
as in the commutative diagram
\[
            \begin{tikzcd}
                 A(L\times H) \arrow{d}{ \Phi}[swap]{\iso} \ar[r]
                 &
                 \pi_H^0(BL) \arrow{d}{\text{tom Dieck splitting}}[swap]{\iso}
                  \\
                  \displaystyle \bigoplus_{[K \subset H]}  A(L,W_{H}(K)) \ar[r] 
                  & 
                  \displaystyle \bigoplus_{[K \subset H]} [ \Sigma^\infty BL_+, \Sigma^\infty B W_H(K)_+ ].
                               \end{tikzcd}
\]
The lower horizontal arrow is completion at $I_L$ according to the version of the Segal conjecture given in \cite{LMM}.
\end{proof}

The $m$th power operation $P_m \colon \ul{A}_G \powerarr \MackUp{A}$, 
when evaluated at an orbit $G/H$, takes the form
\[ P_m \colon {A}(H) \powerarr A(H
\times \Sigma_m).\]
On an $H$-set $X$, it is given by $X\mapsto X^{\times m}$, where the output is considered as an $(H\times \Sigma_m)$-set. 

In the case $m=2$, we have a complete description of $\ul{J}$ as follows.

\begin{prop}\label{p2Identity} In the case $m=2$, the Mackey ideal $\mf{J} \subset \MackUpPower{A}{2}$ is the kernel of the $\Sigma_2$-fixed point homomorphism $ \MackUpPower{A}{2} \rtarr \mf{A}_G$, and the composition
\[ \mf{A}_G \powerxarr{P_2}\MackUpPower{A}{2} \rtarr \MackUpPower{A}{2}/\mf{J} \iso \mf{A}_G\]
is the identity map.
\end{prop}

\begin{proof}
It suffices to consider the case $H=G$.
The kernel of the $\Sigma_2$-fixed point homomorphism  has generators $(G\times \Sigma_2)/\Gamma$ where $\Gamma$ is a graph subgroup. If $\Gamma$ is a non-transitive graph subgroup, then it is of the form $K\times \Sigma_1\times \Sigma_1$ and therefore in $\ul{J}(G/G)$. On the other hand, if it is transitive, then it is in $\ul{J}(G/G)$ according to \cref{JG_ACTION_PRIME_NONBOREL_THM}.

Now we consider the composition of the power operation $P_2$ and fixed points.
The $\Sigma_2$-fixed points of $P_2(G/H)=(G/H)^{\times 2}$ are simply the diagonal $G/H\iso \Delta(G/H)\subset (G/H)^{\times 2}$.
This shows that the composition is the identity as claimed.
\end{proof}

When $m=p>2$ is prime, we have the following result.

\begin{prop}\label{SecFiveAbelianGraphs} If $p>2$ is prime and $L\leq G$, then the orbit $(L\times \Sigma_p)/\Gamma$ lies in the ideal  $\ul{J}(G/L)\subset A(L\times \Sigma_p)$ if and only if either $\Gamma$ is  a 
graph subgroup of $L\times \Sigma_p$ such that $\pi_{\Sigma_p}(\Gamma)$ contains some $C_p$, 
or $\Gamma$ is
subconjugate to $L\times \Sigma_i \times \Sigma_j$, for $i$ and $j$ positive and summing to $p$.
\end{prop}

\begin{proof}
This  follows from \cref{JG_ACTION_PRIME_NONBOREL_THM}.
\end{proof}

Next, we will calculate some examples of \cref{p2Identity} and \cref{SecFiveAbelianGraphs}. To calculate the power operations in these examples, we make use of the binomial formula: for $G$-sets $X$ and $Y$, there is an isomorphism of $G\times \Sigma_n$-sets
\[
    (X \amalg Y)^n \iso \coprod_{i+j=n} \mathrm{Tr}_{\Sigma_i \times \Sigma_j}^{\Sigma_n} ( X^i \times Y^j ).
\]

\begin{ex}\label{AC2Ex}
Consider the case $G=C_2$ and $m=2$. 
We will describe the power operation $P_2 \colon \mf{A}_{C_2} \powerarr \MackUpGroupPower{A}{C_2}{2}$, which is only a map of coefficient systems (of sets) over $C_2$.
We write $\Gamma=C_2\times \Sigma_2$ and $D < \Gamma$ for the diagonal subgroup.
Writing $1$ for the one-point orbit (of any group), 
we have
\[ A(C_2) = \Z\{C_2,1\},\]
\[ A(\Sigma_2) = \Z\{\Sigma_2,1\},\]
and
\[ A(C_2\times \Sigma_2) = \Z\{\Gamma,\Gamma/C_2,\Gamma/\Sigma_2,\Gamma/D,1\}.\]

\begin{prop} The power operations
\[ P_2^e \colon A(e) \powerarr A(\Sigma_2) \]
and 
\[ P_2^{C_2} \colon A(C_2) \powerarr A(C_2\times \Sigma_2)\]
are given by
\[ P_2^e(k) = \frac{k^2-k}2 \Sigma_2 + k\]
and
\[ P_2^{C_2}(n C_2 + k) = (n^2-n+kn)\Gamma \, + \, \frac{k^2-k}2 \Gamma/C_2 \, + \, n \Gamma/\Sigma_2 \, + \, n \Gamma/D + k  ,\]
respectively.
\end{prop}

\begin{proof}
For $P_2^e$, this is simply a matter of observing that the diagonal of $k\times k$ is fixed by the $\Sigma_2$-action, and the rest is free. The case of $P_2^{C_2}$ is displayed in \cref{C2C2Figure}.
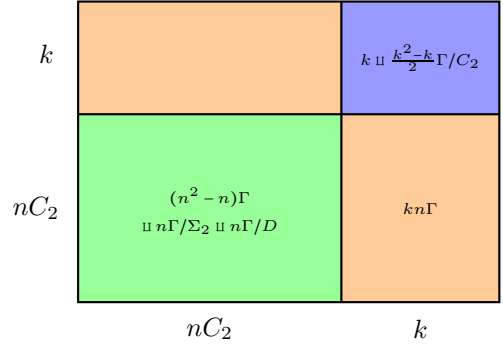
\begin{figure}
\caption{The $(C_2\times \Sigma_2)$-set $(nC_2+k)\times (nC_2+k)$\label{C2C2Figure}}
\begin{tikzpicture}[xscale=1.4]
\filldraw [fill=green!40] (0,0) rectangle (2.5,2.5);
\filldraw [fill=orange!40 ] (2.5,0) rectangle (4,2.5);
\filldraw [fill=orange!40 ] (0,2.5) rectangle (2.5,4);
\filldraw [fill=blue!40 ] (2.5,2.5) rectangle (4,4);
\draw [thick] (2.5,0) -- (2.5,4);
\draw [thick] (0,2.5) -- (4,2.5);
\draw [black,thick] (0,0) rectangle (4,4);
\node at (1.25,-0.35) {$n C_2$};
\node at (3.25,-0.35) {$k$};
\node at (-0.4,1.25) {$n C_2$};
\node at (-0.3,3.3) {$k$};
\node [font=\tiny] at (3.25,3.25) {$k \amalg \frac{k^2-k}2 \Gamma/C_2$};
\node [font=\tiny] at (3.25,1.25) {$k n \Gamma$};
\node [font=\tiny] at (1.25,1.4) {$ (n^2-n) \Gamma $};
\node [font=\tiny] at (1.25,1) {$\amalg \, n \Gamma/\Sigma_2 \amalg n \Gamma/D$};
\end{tikzpicture}
\end{figure}
The key here is that $C_2\times C_2 \iso \Gamma/\Sigma_2 \amalg \Gamma/D$.
\end{proof}

The images of the transfer maps $A(e) \indarr A(\Sigma_2)$ and $A(C_2) \indarr A(C_2\times \Sigma_2)$ are $\Z\{\Sigma_2\}\subset A(\Sigma_2)$ and $\Z\{ \Gamma, \Gamma/C_2\}\subset A(C_2\times \Sigma_2)$. Thus, after modding out by the images of these transfer maps, we have
\begin{center}
\begin{tikzpicture}
\node (C2L) at (0,2) {$\Z\{C_2,1\}$};
\node (eL) at (0,0) {$\Z$};

\draw[bend right=20,->] (C2L) to node[pos=0.4,left] {\mylittlematrix{2 & 1}} (eL);
\draw[bend right=20,color=\inductioncolor,->] (eL) to node[pos=0.6,right] {\mylittlematrix{1 \\ 0}} (C2L);

\node
(C2) at (6,2) {$\Z\{\Gamma/\Sigma_2,\Gamma/D,1\}$};
\node
(e) at (6,0) {$\Z$};

\draw[bend right=20,->] (C2) to node[left] {\mylittlematrix{2 & 0 & 1}} (e);
\draw[bend right=20,color=\inductioncolor,->] (e) to node[right] {\mylittlematrix{1 \\ 0 \\ 0 }} (C2);

\draw[->,color=\powercolor] (eL) to node[above,font=\tiny] {1} (e);
\draw[->,color=\powercolor] (C2L) to node[above] {\mylittlematrix{1 & 0 \\ 1 & 0 \\ 0 & 1 }}  (C2);
\end{tikzpicture}
\end{center}
The power operation does not commute with the Mackey induction homomorphisms. But after collapsing out the image of the additional transfer $A(D) \indarr A(C_2\times \Sigma_2)$, the power operation becomes the identity map of the Burnside Green functor, as it must be according to \cref{p2Identity}.

\end{ex}

\begin{ex} Consider $G=\Sigma_3$ and $m=2$. 
We will write $\rho$ for a 3-cycle and $\tau$  for a transposition in $\Sigma_3$. We will let $\sigma\in \Sigma_2$ be the transposition. 
We begin by restricting attention to $C_3 < \Sigma_3$. 
There is a decomposition of $(C_3\times \Sigma_2)$-sets
\[C_3^{\times 2} \iso (C_3\times \Sigma_2)/\Sigma_2  \amalg (C_3\times \Sigma_2).\]
This decomposition together with the binomial formula 
allows us to understand $(nC_3+k)^2$ as a $C_3 \times \Sigma_2$-set. This gives the following result.

\begin{prop} Writing $\Gamma=C_3\times \Sigma_2$, 
the second power operation \linebreak \mbox{$P_2^{C_3} \colon A(C_3)\rtarr A(C_3 \times \Sigma_2)$} is given by 
\[ P_2^{C_3}(nC_3+k) = \left(nk+\frac{3n^2-n}2\right)\Gamma + \frac{k^2-k}2\Gamma/C_3 + n\Gamma/\Sigma_2 + k.\]\end{prop}

On the other hand, now writing $\Gamma = \Sigma_3 \times \Sigma_2$, we have
\[\Sigma_3^{\times 2} \iso \Gamma /\Sigma_2\, \amalg \Gamma \amalg 3 (\Gamma/D),
\]
where $D < \Gamma$ is  the order two subgroup generated by the element $(\tau,\sigma)$.
We have
\[ A(\Sigma_3) = \Z\{\Sigma_3, \Sigma_3/C_3, \Sigma_3/C_2, 1\}\]
and
\[ A(\Sigma_3 \times \Sigma_2) = \Z\{ \Gamma, \Gamma/C_3, \Gamma/C_2,  \Gamma/\Sigma_3, \Gamma/\Sigma_2, \Gamma/(C_3\times \Sigma_2),  \Gamma/(C_2\times \Sigma_2), \Gamma/D,\Gamma/DC_3,1
\}.\]

\begin{prop} Writing $\Gamma=\Sigma_3\times \Sigma_2$, 
the second power operation \linebreak \mbox{$P_2^{\Sigma_3} \colon A(\Sigma_3)\rtarr A(\Sigma_3 \times \Sigma_2)$} is given by 
\begin{equation*}
\begin{split}
 & P_2^{\Sigma_3}(n \Sigma_3 \amalg i \Sigma_3/C_3 \amalg j \Sigma_3/C_2 \amalg k) = 
 \left(3n^2 - 2n + 2ni + 3nj + nk + ij + \frac{j^2 - j }2\right) \Gamma \\
 & \qquad\qquad  \amalg (i^2 -i + ik) \Gamma/C_3 \amalg \left( \frac{j^2-j}2 + jk\right) \Gamma/C_2 \amalg
 \frac{k^2-k}2 \Gamma/\Sigma_3 \amalg n \Gamma/\Sigma_2 \\
 & \qquad \qquad \amalg i \Gamma/(C_3\times \Sigma_2) \amalg j \Gamma/(C_2\times \Sigma_2) \amalg (3n+j) \Gamma/D \amalg
 i \Gamma/DC_3 \amalg k.
 \end{split}
\end{equation*}
\end{prop}

In addition to the transfers already discussed in \cref{AC2Ex}, the images of the transfers $A(C_3) \indarr A(C_3\times \Sigma_2)$ and $A(\Sigma_3) \indarr A(\Sigma_3\times \Sigma_2)$ are
\[ \Z\{ C_3\times \Sigma_2, (C_3\times \Sigma_2)/C_3\} \subset A(C_3\times \Sigma_2) \]
and 
\[
\Z\{ \Gamma, \Gamma/C_3, \Gamma/C_2, \Gamma/\Sigma_3\} \subset A(\Sigma_3 \times \Sigma_2).
\]
Thus, after modding out by the images of these transfer maps, we have
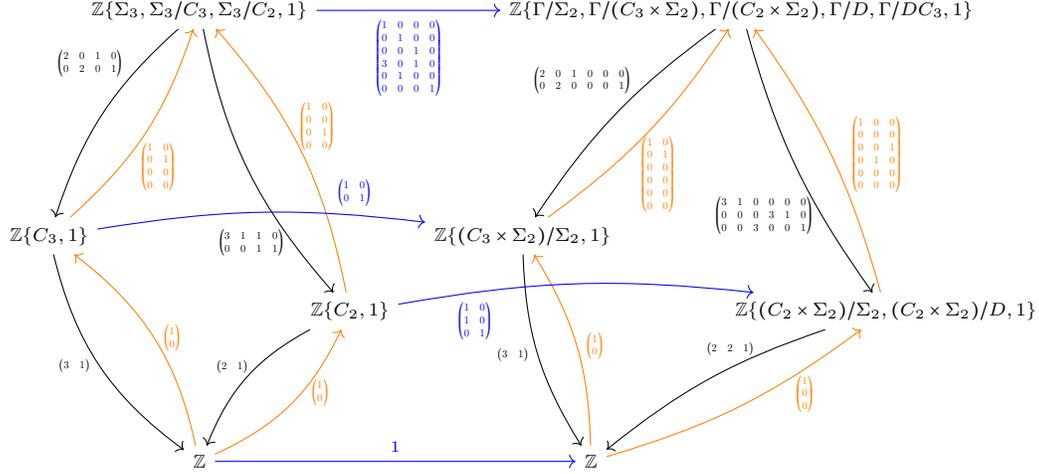
\begin{center}
{ 
\begin{tikzpicture}
\node (GL) at (0,6) {\scriptsize$\Z\{\Sigma_3,\Sigma_3/C_3,\Sigma_3/C_2,1\}$};
\node (C3L) at (-2,3) {\scriptsize$\Z\{C_3,1\}$};
\node (C2L) at (2,2) {\scriptsize$\Z\{C_2,1\}$};
\node (eL) at (0,0) {\scriptsize$\Z$};

\draw[bend right=20,->] (C2L) to node[pos=0.4,left] {\mylittlematrix{2 & 1}} (eL);
\draw[bend right=20,color=\inductioncolor,->] (eL) to node[pos=0.6,right] {\mylittlematrix{1\\ 0}} (C2L);
\draw[bend right=20,->] (C3L) to node[left] {\mylittlematrix{3 & 1  }} (eL);
\draw[bend right=20,color=\inductioncolor,->] (eL) to node[right]{\mylittlematrix{1 \\ 0 }} (C3L);
\draw[bend right=15,->] (GL) to node[left,pos=0.8]{\mylittlematrix{3 & 1 & 1 & 0 \\ 0 & 0 & 1 & 1  }} (C2L);
\draw[bend right=15,color=\inductioncolor,->] (C2L) to node[right,pos=0.6,xshift={1pt}]{\mylittlematrix{1 & 0 \\ 0 & 0 \\ 0 & 1 \\ 0 & 0  }} (GL);
\draw[bend right=15,->] (GL) to node[left,pos=0.2,xshift={-1ex}]{\mylittlematrix{2 & 0 & 1 & 0  \\ 0 & 2 & 0 & 1 }} (C3L);
\draw[bend right=15,color=\inductioncolor,->] (C3L) to node[right,pos=0.3,xshift={1ex}] {\mylittlematrix{1 & 0 \\ 0 & 1 \\ 0 & 0 \\ 0 & 0 }}(GL);

\node[ right = 50mm of GL] (G) at (-1,6) {\scriptsize$\Z\{\Gamma/\Sigma_2, \Gamma/(C_3\times \Sigma_2), \Gamma/(C_2\times \Sigma_2), \Gamma/D, \Gamma/DC_3, 1\}$};
\node[ right = 50mm of C3L] (C3) at (-2,3) {\scriptsize$\Z\{(C_3\times  \Sigma_2)/\Sigma_2, 1  \}$};
\node[ right = 50mm of C2L] (C2) at (2,2) {\scriptsize$\Z\{(C_2\times  \Sigma_2)/\Sigma_2,(C_2 \times \Sigma_2)/D, 1 \}$};
\node[ right = 50mm of eL] (e) at (0,0) {\scriptsize$\Z$};

\draw[bend right=10,->] (C2) to node[pos=0.2,left,xshift={-1ex}] {\mylittlematrix{2 & 2 & 1 }} (e);
\draw[bend right=10,color=\inductioncolor,->] (e) to node[pos=0.6,right,xshift=2ex] {\mylittlematrix{1 \\ 0 \\ 0  }} (C2);
\draw[bend right=15,->] (C3) to node[left] {\mylittlematrix{3 & 1 }} (e);
\draw[bend right=15,color=\inductioncolor,->] (e) to node[right]{\mylittlematrix{1 \\ 0 }} (C3);
\draw[bend right=10,->] (G) to node[left,pos=0.7]{\mylittlematrix{ 
3 & 1 & 0 & 0 & 0 & 0 \\
0 & 0 & 0 & 3 & 1 & 0\\ 
0 & 0 & 3 & 0 & 0 & 1
 }} (C2);
\draw[bend right=10,color=\inductioncolor,->] (C2) to node[right,xshift=1ex]{\mylittlematrix{
1 & 0 & 0 \\
0 & 0  & 0 \\
0 & 0 & 1\\
0 & 1& 0 \\
0 & 0 & 0 \\
0 & 0 & 0 \\  }} (G);
\draw[bend right=10,->] (G) to node[left,pos=0.3,xshift={-1ex}]{\mylittlematrix{
2 & 0 & 1 & 0 & 0 & 0\\
 0 & 2  & 0 & 0 & 0 & 1
}} (C3);
\draw[bend right=10,color=\inductioncolor,->] (C3) to node[right,pos=0.25,xshift={2.5ex}] {\mylittlematrix{
1 & 0\\
0 & 1\\
0 & 0\\
0 & 0\\
0 & 0\\
0 & 0\\
}}(G);

\draw[->,color=\powercolor] (eL) to node[above,font=\tiny] {1} (e);
\draw[->,bend left=8,color=\powercolor] (C2L) to node[below,pos=0.2] {\mylittlematrix{1 & 0 \\ 1 & 0 \\ 0 & 1 }}  (C2);
\draw[->,bend left=8,color=\powercolor] (C3L) to node[above,pos=0.8] {\mylittlematrix{1 & 0 \\ 0 & 1  }} (C3);
\draw[->,color=\powercolor] (GL) to node[below] {\mylittlematrix{
1 & 0 & 0 & 0 \\
0 & 1 & 0 & 0 \\
0 & 0 & 1 & 0 \\
3 & 0 & 1 & 0 \\
0 & 1 & 0 & 0 \\
0 & 0 & 0 & 1
 }} (G);
\end{tikzpicture}
}
\end{center}
In order to make the power operations commute with Mackey induction, we must further collapse $\Z\{(C_2\times \Sigma_2)/D\}\subset A(C_2\times \Sigma_2)$ and $\Z\{\Gamma/D,\Gamma/DC_3\}\subset A(\Sigma_3\times \Sigma_2)$. The resulting power operation of Green functors is the identity on $\ul{A}_{\Sigma_3}$, 
as it must be according to \cref{p2Identity}.
\end{ex}

\begin{ex}
Consider $G=C_3$ and $m=3$. Then 
\[ A(\Sigma_3) = \Z\{\Sigma_3, \Sigma_3/C_3, \Sigma_3/C_2, 1\}. \]
For $\Gamma=C_3\times \Sigma_3$, we write $C_3^R$ for the order 3 subgroup of $\Sigma_3$, and we write 
$C_2$ for a choice of order two subgroup of $\Sigma_3$ and
$\Delta$ for  the order 3 subgroup generated by $(\rho,\sigma_3)$, where $\rho$ generates $C_3$ and $\sigma_3$ is a 3-cycle. Then
\[ A(C_3\times \Sigma_3) \iso 
\Z\{
\Gamma, \Gamma/C_3, \Gamma/C_2, \Gamma/\Sigma_3, \Gamma/C_3^R, \Gamma/(C_3\times C_2), \Gamma/C_3^{\times 2}, \Gamma/\Delta, \Gamma/\Delta C_2, 1
\},
\]
where $\Delta C_2$ is the internal product in $C_3\times \Sigma_3$.

\begin{prop}\label{AC3ThirdPower} 
Writing $\Gamma=C_3\times \Sigma_3$, the power operations
\[ P_3^e \colon  A(e) \powerarr A(\Sigma_3) \]
and 
\[ P_3^{C_3} \colon  A(C_3) \powerarr A(C_3\times \Sigma_3)\]
are given by
\[ P_3^e(k) =
\binom{k}3 \Sigma_3 + k(k-1) \Sigma_3/C_2 + k \Sigma_3/\Sigma_3
\]
and
\begin{equation*}
\begin{split}
 P_3^{C_3}(n C_3 + k) &= 
\Big[ n \binom{k}2 + k \frac{3n^2-n}2  + 6\binom{n}2 + 9\binom{n}3
\Big]
\Gamma \, + \,
\binom{k}3 \Gamma/C_3 \,
 + \,
n \Gamma/\Sigma_3
\\
& \quad 
 + \,
k(k-1) \Gamma/(C_3\times C_2) \,
+ \, \Big[ 2nk +2n+6\binom{n}2\Big] \Gamma/C_2 \,
+ \, n \Gamma/\Delta \,
+ \, k \Gamma/\Gamma \,
\end{split}
\end{equation*}
respectively.
\end{prop}

By \cref{SecFiveAbelianGraphs}, $\ul{J}(C_3/C_3)\subset A(C_3\times \Sigma_3)$ is generated by the orbits
$\Gamma$, $\Gamma/C_3$, $\Gamma/(C_3\times C_2)$, $\Gamma/C_2$, and $\Gamma/\Delta$, which are precisely the terms appearing in \cref{AC3ThirdPower} with a nonlinear coefficient. 
The resulting  power operation of Green functors
is then an inclusion
\[ \ul{A}_{C_3} \powerxarr{P_3/\ul{J}} \MackUpGroupPower{A}{C_3}{3}/\ul{J} \iso \ul{A}_{C_3} \oplus \ul{A}_{C_3} \oplus \hspace{-1ex}
\raisebox{-4ex}{\begin{tikzpicture}
\node[scale=.8] (a) at (0,0){
\begin{tikzcd}
\Z\{\Gamma/\Delta\} \restrict{d} \\
0 \induct{u} 
\end{tikzcd}};
\end{tikzpicture}}
\]

Here the first copy of $\ul{A}_{C_3}$ contains the orbits $\Gamma/\Sigma_3$ and $1=\Gamma/\Gamma$, whereas the second copy contains the orbits $\Gamma/C_3^R$ and $\Gamma/(C_3)^2$.

\end{ex}

\subsection{Global $KU$}
\label{sec:KUG}

Consider the $G_\infty$-ring spectrum $KU$ (\cite[6.4.9]{global}). The associated $G$-Green functor $\ul{KU}_G^0$ is the representation ring Green functor, with $\ul{KU}_G^0(G/H) \iso RU(H)$.
The restriction and induction maps correspond to restriction and induction of representations, respectively. 
The $m$th power operation associated to the $G_\infty$-ring spectrum $KU$ as in \cref{GLOBAL_OPS_SEC}  takes the form
\[ P_m \colon \ul{RU}_G \powerarr \MackUp{RU}\]
and is given at $G/H$ by  $V\mapsto V^{\otimes m}$, where the latter is consider as a $(H\times\Sigma_m)$-representation.
On the other hand, the $m$th power operation associated to the $H_\infty$-ring $G$-spectrum $KU_G$ takes the form
\[ P_m \colon \ul{RU}_G = \ul{KU}_G^0 \powerarr \ul{KU}_G^0(B\Sigma_m).\]
as in \cref{GPOWER_OPS_SEC}. 
The map of $G\times\Sigma_m$-spaces $E\Sigma_m\rtarr \ast$ induces a map of $G$-Green functors $\MackUp{RU} \rtarr \ul{KU}^0_G(B\Sigma_m)$. At an orbit $G/H$, this is the map
\[ \MackUpOrbit{RU}{G/H} \iso RU(H\times \Sigma_m) \xrtarr{(-)_{h\Sigma_m}} KU_H^0(B\Sigma_m) \iso \ul{KU}^0_G(B\Sigma_m)(G/H)\]
which takes an $H\times\Sigma_m$-representation and passes to homotopy orbits with respect to the $\Sigma_m$-action.
As we show in the following proposition, 
this map is a completion, in the sense that 
 \[ \ul{KU}^0_G(B\Sigma_m)(G/H) \iso KU^0_H(B\Sigma_m) 
 \iso RU(H\times \Sigma_m)^{\wedge}_{I_{\Sigma_m}}
 \iso \MackUpOrbit{RU}{G/H}^{\wedge}_{I_{\Sigma_m}}.
 \]

\begin{prop}
For any groups $H$ and $L$, the map 
\[ (-)_{hL}\colon RU(H\times L) \rtarr KU_H^0(BL)\]
is completion at the augmentation ideal $I_L\subset RU(L)$.
\end{prop}

\begin{proof}
We have a 
commuting square
\[
            \begin{tikzcd}
                 RU(H\times L) \arrow{d}[swap]{\iso} \ar[r]
                 &
                 KU_H^0(BL) \arrow{d}{\iso}
                  \\
                  RU(H)\otimes RU(L)  \ar[r] 
                  & 
                  RU(H) \otimes KU(BL),
                               \end{tikzcd}
\]
where the right vertical map is an isomorphism because $H$ acts trivially on $BL$.
Now the Atiyah-Segal  Completion Theorem \cite{AtiyahSegal} states that $RU(L) \rtarr KU(BL)$ is completion at $I_L$.
The isomorphism
\[\begin{split}
 RU(H) \otimes RU(L)^{\wedge}_{I_L} & \iso RU(H)\otimes RU(L)\otimes_{RU(L)} RU(L)^{\wedge}_{I_L} \\
 &\iso \big( RU(H)\otimes RU(L)\big) \otimes_{RU(L)} RU(L)^{\wedge}_{I_L}
 \end{split} \]
 finishes the proof.
\end{proof}
 
As in \cref{sec:sphere}, we focus on the power operation with target $\MackUp{RU}$. 
Denote by $\mathrm{ev}_\sigma
\colon RU(\Sigma_m)\rtarr \Z$ the homomorphism that evaluates the character of a representation at an $m$-cycle. This homomorphism is precisely the quotient map $KU_{\Sigma_m}^0(\ast) \rtarr KU_{\Sigma_m}^0(\ast)/I_{\Tr}$ by the transfer ideal. 

\begin{prop}[\cite{AtPowOps}]
The composition
\[ 
{KU}_G^0(X) \powerxarr{P_m}  {KU}^0_{G\times \Sigma_m}(X) \iso {KU}_G^0(X) \otimes RU(\Sigma_m) 
\xrtarr{\id\otimes \mathrm{ev}_\sigma} KU^0_G(X)\otimes \Z
\]
is the Adams operation $\psi^m$.
\end{prop}

In other words, 
\[ \ul{RU}_G \powerxarr{P_m/\ul{I}_{\Tr}} \MackUp{RU}/\ul{I}_{\Tr} \iso \ul{RU_G}\]
is levelwise the Adams operation $\psi^m$. As in \cref{Green_ITr_Thm}, this is a map of coefficient systems of commutative rings but not a map of Green functors.

\begin{ex}\label{ex:RUC2}
      Consider the case $G=C_2$ and $m=2$. Then $RU(C_2)= \Z\{1,s\}$, where $s$ is the sign representation, satisfying $s^2=1$.
      Then the Mackey induction sends $1\in RU(e)=\Z$ to the regular representation $1+s$.   
      The ring homomorphism $\psi^2$ squares both of the 1-dimensional representations $1$ and $s$, so the diagram
      \[ \begin{tikzcd}
                  RU(C_2) \arrow[r,"\psi^2",color=\powercolor] & RU(C_2) \\
                  RU(e)=\Z \arrow[r,"\psi^2=\id",color=\powercolor] \arrow[u,color=\inductioncolor] & RU(e)=\Z \arrow[u,color=\inductioncolor] 
            \end{tikzcd}\]
      does {\it not} commute, since
      \begin{equation}
            \label{KU22_EQ}
            \begin{tikzcd}
                  1+s \arrow[mapsto, r,color=\powercolor]  & 1^2+s^2=2 \neq 1+s \\
                  1 \arrow[mapsto, r,color=\powercolor] \arrow[mapsto,color=\inductioncolor]{u} &  1. \arrow[mapsto,color=\inductioncolor, end anchor={[xshift=4ex]}]{u}
            \end{tikzcd}
      \end{equation}
      
      Following \cref{JG_ACTION_PRIME_NONBOREL_THM},
      we observe that $\underline{J}$ is generated by $\underline{I}_{\Tr}$ as well as an additional transfer:
      we must further collapse the image of
\[
      \begin{tikzcd}[row sep = tiny]
            RU(D) \arrow[r, color=\inductioncolor] &  RU(C_2\times \Sigma_2) \iso RU(C_2) \otimes RU(\Sigma_2) \arrow[r,"1\otimes \mathrm{ev}_\sigma"] &  RU(C_2) \otimes \Z, \\
            1 \arrow[r, mapsto] &  1+s\overline{s} \arrow[r, mapsto] &  1-s
      \end{tikzcd}
      \]
      where $D\leq C_2\times \Sigma_2$ is once again the diagonal subgroup and 
      $\overline{s}\in RU(\Sigma_2)$ is the sign representation.
      Thus, collapsing the ideal $\ul{J}$ imposes the relation $s \sim 1$,
      in particular making \eqref{KU22_EQ} commute. 
      The resulting quotient Green functor is the constant Mackey functor $\mf{\Z}$, and the power operation of Green functors
      \[
            P_2/\ul{J} \, \colon \, \mf{RU}_{C_2} \powerarr \MackUpGroupPower{RU}{C_2}{2}/\ul{J} \iso \mf{\Z}
      \]
      is the augmentation, given by restricting to the trivial subgroup.
\end{ex}

\begin{ex}\label{ex:RUS3}
Consider now $G=\Sigma_3$ and $m=2$. We have
\[ RU(\Sigma_3) = \Z[s,W]/(s^2-1,s W-W, W^2-W-s-1),\]
where $W$ is the reduced standard representation,
and
\[ RU(C_3) = \Z[\lambda]/(\lambda^3 -1).\]
By
\cref{JG_ACTION_PRIME_NONBOREL_THM}, the ideal $\underline{J}$ is generated by $\underline{I}_{\Tr}$ and three additional transfers:
\[
      \begin{tikzcd}[row sep = tiny]
            RU(D) \arrow[r, color=\inductioncolor]
            &
            RU(\Sigma_3 \times \Sigma_2) \cong RU(\Sigma_3) \otimes RU(\Sigma_2) \arrow[r, "1 \otimes \mathrm{ev}_\sigma"]
            & 
            RU(\Sigma_3) \otimes \mathbb Z, &
            \\
            1 \arrow[r, mapsto]
            &
            1 + s \bar s + W + W \bar s \arrow[r, mapsto]
            &[10pt]
            1-s
            \\ 
            RU(\Gamma(\mathrm{sgn})) \arrow[r, color=\inductioncolor]
            &
            RU(\Sigma_3 \times \Sigma_2) \cong RU(\Sigma_3) \otimes RU(\Sigma_2) \arrow[r, "1 \otimes \mathrm{ev}_\sigma"]
            &
            RU(\Sigma_3) \otimes \mathbb Z, & \hspace{-2em} \text{and}
            \\
            1 \arrow[r, mapsto]
            &
            1+s\bar s \arrow[r, mapsto]
            &
            1-s
            \\ 
            RU(D) \arrow[r, color=\inductioncolor]
            &
            RU(C_2 \times \Sigma_2) \cong RU(C_2) \otimes RU(\Sigma_2) \arrow[r, "1 \otimes \mathrm{ev}_\sigma"]
            &
            RU(C_2) \otimes \mathbb Z,
            \\
            1 \arrow[r, mapsto]
            &
            1 + s \bar s \arrow[r, mapsto]
            &
            1-s
      \end{tikzcd}
\]
where
$C_2 \leq \Sigma_3$
is a choice of order two subgroup,
$D \leq C_2\times \Sigma_2 \leq \Sigma_3 \times \Sigma_2$ is the diagonal order 2 subgroup,
$\mathrm{sgn} \colon \Sigma_3 \to \Sigma_2$ is the sign homomorphism, and
$\Gamma(\mathrm{sgn}) \leq \Sigma_3 \times \Sigma_2$ the associated graph subgroup.
As in \cref{ex:RUC2}, we conclude that we must impose the relation $s\sim 1$ in the cyclic 2-subgroups and in $RU(\Sigma_3)$. 
The resulting power operation of Green functors
\[ P_2/\ul{J} \, \colon \, \ul{RU}_{\Sigma_3} \powerarr \MackUpGroupPower{RU}{\Sigma_3}{2}/\ul{J}\]
is given by
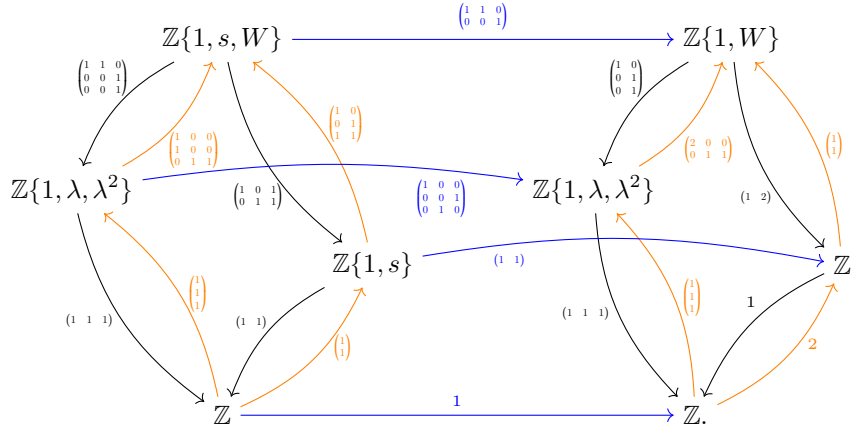
\begin{center}
\begin{tikzpicture}
\node (GL) at (0,5) {$\Z\{1,s,W\}$};
\node (C3L) at (-2,3) {$\Z\{1,\lambda,\lambda^2\}$};
\node (C2L) at (2,2) {$\Z\{1,s\}$};
\node (eL) at (0,0) {$\Z$};

\draw[bend right=20,->] (C2L) to node[pos=0.4,left] {\mylittlematrix{1 & 1}} (eL);
\draw[bend right=20,color=\inductioncolor,->] (eL) to node[pos=0.6,right] {\mylittlematrix{1\\ 1}} (C2L);
\draw[bend right=20,->] (C3L) to node[left] {\mylittlematrix{1 & 1 & 1 }} (eL);
\draw[bend right=20,color=\inductioncolor,->] (eL) to node[right]{\mylittlematrix{1 \\ 1 \\ 1 }} (C3L);
\draw[bend right=20,->] (GL) to node[left,pos=0.7]{\mylittlematrix{1 & 0 & 1 \\ 0 & 1 & 1  }} (C2L);
\draw[bend right=20,color=\inductioncolor,->] (C2L) to node[right,pos=0.6,xshift={1pt}]{\mylittlematrix{1 & 0 \\ 0 & 1 \\ 1 & 1  }} (GL);
\draw[bend right=20,->] (GL) to node[left,pos=0.2,xshift={-1ex}]{\mylittlematrix{1 & 1 & 0 \\ 0 & 0 & 1 \\ 0 & 0 & 1 }} (C3L);
\draw[bend right=20,color=\inductioncolor,->] (C3L) to node[right,pos=0.2,xshift={1ex}] {\mylittlematrix{1 & 0 & 0 \\ 1 & 0 & 0 \\ 0 & 1 & 1}}(GL);

\node[ right = 60mm of GL] (G) at (0,5) {$\Z\{1,W\}$};
\node[ right = 60mm of C3L] (C3) at (-2,3) {$\Z\{1,\lambda,\lambda^2\}$};
\node[ right = 60mm of C2L] (C2) at (2,2) {$\Z$};
\node[ right = 60mm of eL] (e) at (0,0) {$\Z$.};

\draw[bend right=20,->] (C2) to node[pos=0.3,left,font=\tiny,xshift={-1ex}] {1} (e);
\draw[bend right=20,color=\inductioncolor,->] (e) to node[pos=0.6,right,font=\tiny] {2} (C2);
\draw[bend right=20,->] (C3) to node[left] {\mylittlematrix{1 & 1 & 1 }} (e);
\draw[bend right=20,color=\inductioncolor,->] (e) to node[right]{\mylittlematrix{1 \\ 1 \\ 1 }} (C3);
\draw[bend right=20,->] (G) to node[left,pos=0.7]{\mylittlematrix{1 & 2  }} (C2);
\draw[bend right=20,color=\inductioncolor,->] (C2) to node[right]{\mylittlematrix{1 \\ 1  }} (G);
\draw[bend right=20,->] (G) to node[left,pos=0.2,xshift={-1ex}]{\mylittlematrix{1 & 0 \\ 0 & 1 \\ 0 & 1}} (C3);
\draw[bend right=20,color=\inductioncolor,->] (C3) to node[right,pos=0.2,xshift={1ex}] {\mylittlematrix{2 & 0 & 0 \\ 0 & 1 & 1}}(G);

\draw[->,color=\powercolor] (eL) to node[above,font=\tiny] {1} (e);
\draw[->,bend left=10,color=\powercolor] (C2L) to node[below,pos=0.2] {\mylittlematrix{1 & 1 }}  (C2);
\draw[->,bend left=8,color=\powercolor] (C3L) to node[below,pos=0.8] {\mylittlematrix{1 & 0 & 0 \\ 0 & 0 & 1 \\ 0 & 1 & 0 }} (C3);
\draw[->,color=\powercolor] (GL) to node[above] {\mylittlematrix{1 & 1 & 0 \\ 0 & 0 & 1 }} (G);
\end{tikzpicture}
\end{center}
Here,  the value of $\psi_2(W)$ may be deduced by using the (character) embedding of the representation ring into the ring of class functions and using the formula for the Adams operation $\psi_2$ given in \cref{prop:ClassModTransfer}. Since all other representations that appear are 1-dimensional, the operation $\psi_2$ simply squares them.
\end{ex}

\begin{ex}
Consider now $G=\Sigma_3$ and $m=3$. 
We begin with the same source Green functor as in \cref{ex:RUS3}.
As in \cref{ex:RUS3}, $\underline{J}$ is generated by $\underline{I}_{\Tr}$ and three additional transfers:
\[
      \begin{tikzcd}[row sep = tiny]
            RU(D_{C_3}) \arrow[r, color=\inductioncolor]
            &
            RU(\Sigma_3 \times \Sigma_3) \cong RU(\Sigma_3) \otimes RU(\Sigma_3) \arrow[r, "1 \otimes \mathrm{ev}_{\sigma}"]
            &
            RU(\Sigma_3) \otimes \mathbb Z,
            \\
            1 \arrow[r, mapsto]
            &
            1+s + \bar s + s \bar s + 2 W \bar W \arrow[r, mapsto]
            &
            2+2s-2W
            \\
            RU(D_{\Sigma_3}) \arrow[r, color=\inductioncolor]
            &
            RU(\Sigma_3 \times \Sigma_3) \cong RU(\Sigma_3) \otimes RU(\Sigma_3) \arrow[r, "1 \otimes \mathrm{ev}_{\sigma}"]
            &
            RU(\Sigma_3) \otimes \mathbb Z, & \hspace{-2em} \text{and}
            \\
            1 \arrow[r, mapsto]
            &
            1+s\bar s + W \bar W \arrow[r, mapsto]
            &
            1+s-W
            \\
            RU(D_{C_3}) \arrow[r, color=\inductioncolor]
            &
            RU(C_3 \times \Sigma_3) \cong RU(C_3) \otimes RU(\Sigma_3) \arrow[r, "1 \otimes \mathrm{ev}_{\sigma}"]
            &
            RU(C_3) \otimes \mathbb Z,
            \\
            1 \arrow[r, mapsto]
            &
            1+\bar s \to \lambda \bar W + \lambda^2 \bar W \arrow[r, mapsto]
            &
            2 - \lambda - \lambda^2
      \end{tikzcd}
\]
where
$D_{C_3} \leq C_3 \times \Sigma_3$ is the order 3 subgroup generated by $(\rho, \sigma_3)$ for $\rho$ a generator of $C_3$ and $\sigma_3$ a 3-cycle, and
$D_{\Sigma_3} \leq \Sigma_3 \times \Sigma_3$ is the diagonal subgroup.
Thus we collapse the ideals $(W- s-1) \subset RU(\Sigma_3)$ and $(\lambda^2 +\lambda-2)\subset RU(C_3)$. The quotients are
\[ RU(\Sigma_3)/(W-s-1) \iso \Z\{1,s\} \qquad \text{and} \qquad RU(C_3)/(\lambda^2+\lambda-2) \iso \Z\{1\} \oplus \Z/3\{\bar{\lambda}\},\]
where $\bar{\lambda} = \lambda-1$.

The resulting power operation of Green functors
\[ P_3/\ul{J} \, \colon \, \ul{RU}_{\Sigma_3} \powerarr \MackUpGroupPower{RU}{\Sigma_3}{3}/\ul{J}\]
is given by
\begin{center}
\begin{tikzpicture}
\node (GL) at (0,5) {$\Z\{1,s,W\}$};
\node (C3L) at (-2,3) {$\Z\{1,\lambda,\lambda^2\}$};
\node (C2L) at (2,2) {$\Z\{1,s\}$};
\node (eL) at (0,0) {$\Z$};

\draw[bend right=20,->] (C2L) to node[pos=0.4,left] {\mylittlematrix{1 & 1}} (eL);
\draw[bend right=20,color=\inductioncolor,->] (eL) to node[pos=0.6,right] {\mylittlematrix{1\\ 1}} (C2L);
\draw[bend right=20,->] (C3L) to node[left] {\mylittlematrix{1 & 1 & 1 }} (eL);
\draw[bend right=20,color=\inductioncolor,->] (eL) to node[right]{\mylittlematrix{1 \\ 1 \\ 1 }} (C3L);
\draw[bend right=20,->] (GL) to node[left,pos=0.7]{\mylittlematrix{1 & 0 & 1 \\ 0 & 1 & 1  }} (C2L);
\draw[bend right=20,color=\inductioncolor,->] (C2L) to node[right,pos=0.6,xshift={1pt}]{\mylittlematrix{1 & 0 \\ 0 & 1 \\ 1 & 1  }} (GL);
\draw[bend right=20,->] (GL) to node[left,pos=0.2,xshift={-1ex}]{\mylittlematrix{1 & 1 & 0 \\ 0 & 0 & 1 \\ 0 & 0 & 1 }} (C3L);
\draw[bend right=20,color=\inductioncolor,->] (C3L) to node[right,pos=0.2,xshift={1ex}] {\mylittlematrix{1 & 0 & 0 \\ 1 & 0 & 0 \\ 0 & 1 & 1}}(GL);

\node[ right = 60mm of GL] (G) at (0,5) {$\Z\{1,s\}$};
\node[ right = 60mm of C3L] (C3) at (-2,3) {$\Z\oplus \Z/3\{\bar{\lambda}\}$};
\node[ right = 60mm of C2L] (C2) at (2,2) {$\Z\{1,s\}$};
\node[ right = 60mm of eL] (e) at (0,0) {$\Z$.};

\draw[bend right=20,->] (C2) to node[pos=0.3,left,xshift={-1ex}] {\mylittlematrix{1 & 1}} (e);
\draw[bend right=20,color=\inductioncolor,->] (e) to node[pos=0.6,right] {\mylittlematrix{1 \\ 1}}  (C2);
\draw[bend right=20,->] (C3) to node[left] {\mylittlematrix{1 & 0}} (e);
\draw[bend right=20,color=\inductioncolor,->] (e) to node[right]{\mylittlematrix{3 \\ 0}} (C3);
\draw[bend right=20,->] (G) to node[left,pos=0.7]{\mylittlematrix{1 & 0 \\ 0 & 1  }} (C2);
\draw[bend right=20,color=\inductioncolor,->] (C2) to node[right]{\mylittlematrix{2 & 1 \\ 1 & 2  }} (G);
\draw[bend right=20,->] (G) to node[left,pos=0.2,xshift={-1ex}]{\mylittlematrix{1 & 1 \\ 0 & 0 }} (C3);
\draw[bend right=20,color=\inductioncolor,->] (C3) to node[right,pos=0.3,xshift={2ex}] {\mylittlematrix{1 & 0  \\ 1 & 0  }}(G);

\draw[->,color=\powercolor] (eL) to node[above,font=\tiny] {1} (e);
\draw[->,bend left=10,color=\powercolor] (C2L) to node[above,pos=0.17] {\mylittlematrix{1 & 0 \\ 0 & 1 }}  (C2);
\draw[->,bend left=8,color=\powercolor] (C3L) to node[above,pos=0.8] {\mylittlematrix{1 & 1 & 1 \\ 0 & 0 & 0}} (C3);
\draw[->,color=\powercolor] (GL) to node[above] {\mylittlematrix{1 & 0 & 1 \\ 0 & 1 & 1 }} (G);
\end{tikzpicture}
\end{center}
Again, the value of $\psi_3$ on representations may be deduced via the character embedding.
\end{ex}

\subsection{Class functions } 
\label{sec:classfunctions}

Rings of class functions appear in homotopy theory as approximations to cohomology theories. In particular, equivariant $KU$-theory is approximated by the ring of class functions $Cl(G,\C)$, which is the ring of $\C$-valued functions on the set of conjugacy classes of $G$. Further, Hopkins, Kuhn, and Ravenel \cite{hkr} have shown that the Morava $E$-theories, which are generalizations of $p$-adic $KU$-theory, all admit similar approximations by a ring of ``generalized class functions." We introduce this ring for the ``height $2$" Morava $E$-theories in \cref{sec:heighttwo}. 

The rings of class functions fit together to give Green functors with restriction and induction maps compatible with the restriction and induction maps for the equivariant cohomology theory that they approximate.

\subsubsection{$\C$-valued class functions}
\label{sec:heightone}
We begin by considering $Cl(G) = Cl(G, \C)$, the ring of class functions that arises in the representation theory of finite groups.
For any group $G$, we have a Mackey functor $\ul{Cl}_G$ defined by $\ul{Cl}_G(G/H) = Cl(H)$. For $H \leq K$, the restriction map is given by simply restricting class functions along the map $H/\mathrm{conj} \to K/\mathrm{conj}$. For $H\leq K$, the induction homomorphism 
\[ \mathrm{\Tr}\colon Cl(H) \indarr Cl(K)\]
is (cf. \cite[Theorem~12]{Serre})
\begin{equation}\label{ClassTransfer}
 \mathrm{\Tr}(f)(k) = \sum\limits_{g H \in (K/H)^k} f(g^{-1} k g), 
 \end{equation}
where $(K/H)^k$ denotes the $k$-fixed points under the left action of $K$ on $K/H$.

Conjugacy classes in $\Sigma_m$ correspond to partitions of $m$. Let $\{m_1, \ldots, m_j\}$ be a partition of $m$, so that $m_1 + \ldots + m_j = m$. The power operation $P_m \colon Cl(G) \powerarr Cl(G\times \Sigma_m)$ is given by
\[ P_m(f)(g,\{m_1, \ldots, m_j\}) = \prod_{i=1}^{j} f(g^{m_i}).\]

The following result is well-known.

\begin{prop}\label{prop:ClassModTransfer} The quotient $Cl(G\times \Sigma_m)/I_{\Tr}$ is isomorphic to $Cl(G)$, and the composition
\[ Cl(G) \powerxarr{P_m} Cl(G\times \Sigma_m) \rtarr Cl(G\times \Sigma_m)/I_{\Tr} \iso Cl(G)\]
is the Adams operation $\psi_m$, given by $\psi_m(f)(g) = f(g^m)$.
\end{prop}

\begin{proof}
For any proper partition $\{m_1, \ldots, m_j\}$ of $m$ and conjugacy class $\overline{g}$ in $G$, class functions on $G\times \Sigma$ supported on $(\overline{g},\{m_1, \ldots, m_j\})$ are in the image of the transfer along $G\times \prod_i \Sigma_{m_i}\into G\times \Sigma_m$. It follows that the quotient $Cl(G\times \Sigma_m)/I_{\Tr}$ can be identified with functions supported on $(\overline{g},\sigma)$, where $\sigma=(1\cdots m)$. This identifies the quotient with $Cl(G)$.

By the previous discussion, after passing to the quotient by the transfer ideals for subgroups $G\times \Sigma_i \times \Sigma_j$, only the value of the class function on the long cycle is retained, and 
\[P_m(f) (g,(1\cdots m)) = f(g^m) = \psi_m(f)(g). \qedhere\]
\end{proof}

Let $G_{\pdiv}\subset G$ denote the set of elements whose orders are divisible by $p$, and let 
$G_{\pnodiv}\subset G$ 
denote the set of elements whose orders are {\em not } divisible by $p$.
Similarly, we denote by $Cl_{\pdiv}$ and $Cl_{\pnodiv}$ the rings of $\C$-valued functions on $G_{\pdiv}/\mathrm{conj}$ and $G_{ \pnodiv}/\mathrm{conj}$, respectively.
Then the decomposition $G = G_{\pdiv} \amalg G_{\pnodiv}$
induces an isomorphism of commutative rings 
$Cl(G) \iso Cl_{\pdiv}(G) \times Cl_{\pnodiv}(G)$.

\begin{prop}The Green functor structure on $\ul{Cl}_G$ descends to  Green functor structures on $\ul{Cl_{\pdiv}}_G$ and $\ul{Cl_{\pnodiv}}_G$.
\end{prop}

\begin{proof}
This follows from the fact that any subgroup inclusion $H\into G$ induces inclusions $H_{\pdiv} \into G_{\pdiv}$ and $H_{\pnodiv} \into G_{\pnodiv}$.
\end{proof}

\begin{prop} 
\label{Jpdiv}
If $m=p$ is prime, the image of $\ul{J}$ under 
the isomorphism of \cref{prop:ClassModTransfer}
is the Mackey ideal $\ul{Cl_{\pdiv}}_G$.
 The  power operation of Green functors
\[ P_p/\ul{J} \colon \ul{Cl}_G \powerarr \MackUpGroupPower{Cl}{G}{p}/\ul{J} \iso \ul{Cl_{\pnodiv}}_G\]
is the composition
\[ \ul{Cl}_G \xrtarr{\text{restrict}}  \ul{Cl_{\pnodiv}}_G \xrtarr{\psi_p}  \ul{Cl_{ \pnodiv}}_G.\]
\end{prop}

\begin{proof}
By \cref{prop:ClassModTransfer}, it suffices to show that for any $g\in G_{\pdiv}$, any class function $f$ on $G\times \Sigma_p$ supported at $[g,\sigma]$ is in the image of the transfer from a graph subgroup as in \cref{PRIMESURJ_COR}. Let $S$ be the cyclic subgroup of $G$ generated by $g$, and let $a\colon S \rtarr \Sigma_p$ send $g$ to $\sigma$. Then \cref{ClassTransfer} shows that $\Tr(f)$ is also supported at $[g,\sigma]$. Furthermore, the value of $\Tr(f)$ at $[g,\sigma]$ is a positive integer multiple of $f([g,\sigma])$.
\end{proof}

If $G$ is a $p$-group, we get the following result.

\begin{cor} Suppose that $G$ is a $p$-group. Then the  $G$-Green functor $\MackUpPower{Cl}{p}/\ul{J}$
 is isomorphic to the constant Mackey functor at $\C$. The  power operation of Green functors
\[ \ul{Cl}_G \powerxarr{P_p/\ul{J}} \MackUpPower{Cl}{p}/\ul{J} \iso \ul{\C}\]
is given by evaluating a character at the identity element.
\end{cor}

The following examples all follow from \cref{Jpdiv}. We include them for comparison with the examples of \cref{sec:KUG}.

\begin{ex}
For the group $G=\Z/2$ and $m=2$,
we have
\[ \begin{tikzcd}
Cl(\Z/2) \arrow[r,"\psi_2",color=\powercolor] \restrict{d} 
& Cl(\Z/2\times \Sigma_2)/I_{\Tr} \iso Cl(\Z/2) \restrict{d} \\
Cl(0)=\C \arrow[r,"\psi_2",color=\powercolor] \induct{u} 
& Cl(\Sigma_2)/I_{\Tr}\iso \C. \induct{u}
\end{tikzcd} \]
Tracing the diagram using induction maps gives
\begin{center}
\begin{tikzpicture}
\node (BL) at (0,0) {$a$};
\node (BR) at (4,0) {$a$.};
\node[font=\scriptsize, align=center]  (UL) [draw,rectangle] at (0,2) {$ 0 \mapsto 2a$ \\ $ 1 \mapsto 0$};
\node[font=\scriptsize, align=center] (UR) [draw,rectangle] at (4,2) {$0,\sigma\mapsto 2a$ \\ $1,\sigma\mapsto 2a$};
\node[font=\scriptsize, right = 0mm of UR] (neq)  {$\neq$};
\node[font=\scriptsize, right = 0mm of neq, align=center] (URR) [draw,rectangle]  {$0,\sigma\mapsto 2a$ \\ $1,\sigma\mapsto 0$};

\draw[|->,color=\powercolor] (BL) --  (BR);
\draw[|->,color=\inductioncolor] (BL) -- (UL);
\draw[|->,color=\inductioncolor] (BR) -- (URR);
\draw[|->,color=\powercolor] (UL) -- (UR);
\end{tikzpicture}
\end{center}
If we further quotient by the image of the transfer from the diagonal subgroup $C_2 \xrtarr{\Delta}\Z/2\times \Sigma_2$, then the induction diagram commutes, so that we have a map of Green functors. The quotient Green functor is constant at $\C$, and the 
resulting map of Green functors  $P_2 \colon \mf{Cl}_{\Z/2} \rtarr \mf{\C}$ is given by restriction of class functions to the identity element.
\end{ex}

\begin{ex} Consider $G=\Sigma_3$ and $m=2$.
We use $\rho$ to denote a 3-cycle and $\tau$ to denote (any choice of) transposition. We 
write $C_3$ and $C_2$ for the subgroups generated by $\rho$ and $\tau$

Then by \cref{Jpdiv}, in order for the power operation to be a map of Mackey functors, we must quotient $Cl(\Sigma_3)$ by the values on the conjugacy class of $\tau$, and we must also quotient $Cl(C_2)$  by the same conjugacy class. 
The resulting power operation of Green functors
\[ P_3/\ul{J} \, \colon \, \ul{Cl}_{\Sigma_3} \powerarr \MackUpGroupPower{Cl}{\Sigma_3}{2}/\ul{J}\]
is given by

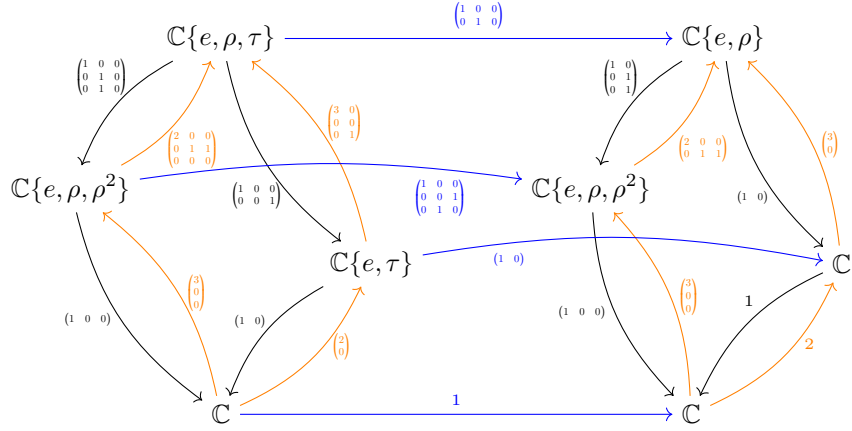
\begin{center}
\begin{tikzpicture}
\node (GL) at (0,5) {$\C\{e,\rho,\tau\}$};
\node (C3L) at (-2,3) {$\C\{e,\rho,\rho^2\}$};
\node (C2L) at (2,2) {$\C\{e,\tau\}$};
\node (eL) at (0,0) {$\C$};

\draw[bend right=20,->] (C2L) to node[pos=0.4,left] {\mylittlematrix{1 & 0}} (eL);
\draw[bend right=20,color=\inductioncolor,->] (eL) to node[pos=0.6,right] {\mylittlematrix{2 \\ 0}} (C2L);
\draw[bend right=20,->] (C3L) to node[left] {\mylittlematrix{1 & 0 & 0 }} (eL);
\draw[bend right=20,color=\inductioncolor,->] (eL) to node[right]{\mylittlematrix{3 \\ 0 \\ 0 }} (C3L);
\draw[bend right=20,->] (GL) to node[left,pos=0.7]{\mylittlematrix{1 & 0 & 0 \\ 0 & 0 & 1  }} (C2L);
\draw[bend right=20,color=\inductioncolor,->] (C2L) to node[right,pos=0.6,xshift={1pt}]{\mylittlematrix{3  & 0 \\ 0 & 0 \\ 0 & 1  }} (GL);
\draw[bend right=20,->] (GL) to node[left,pos=0.2,xshift={-1ex}]{\mylittlematrix{1 & 0 & 0 \\ 0 & 1 & 0 \\ 0 & 1 & 0 }} (C3L);
\draw[bend right=20,color=\inductioncolor,->] (C3L) to node[right,pos=0.2,xshift={1ex}] {\mylittlematrix{2 & 0 & 0 \\ 0 & 1 & 1 \\ 0 & 0 & 0}}(GL);

\node[ right = 60mm of GL] (G) at (0,5) {$\C\{e,\rho\}$};
\node[ right = 60mm of C3L] (C3) at (-2,3) {$\C\{e,\rho,\rho^2\}$};
\node[ right = 60mm of C2L] (C2) at (2,2) {$\C$};
\node[ right = 60mm of eL] (e) at (0,0) {$\C$};

\draw[bend right=20,->] (C2) to node[pos=0.3,left,font=\tiny,xshift={-1ex}] {1} (e);
\draw[bend right=20,color=\inductioncolor,->] (e) to node[pos=0.6,right,font=\tiny] {2} (C2);
\draw[bend right=20,->] (C3) to node[left] {\mylittlematrix{1 & 0 & 0 }} (e);
\draw[bend right=20,color=\inductioncolor,->] (e) to node[right]{\mylittlematrix{3 \\ 0 \\ 0 }} (C3);
\draw[bend right=20,->] (G) to node[left,pos=0.7]{\mylittlematrix{1 & 0  }} (C2);
\draw[bend right=20,color=\inductioncolor,->] (C2) to node[right]{\mylittlematrix{3 \\ 0  }} (G);
\draw[bend right=20,->] (G) to node[left,pos=0.2,xshift={-1ex}]{\mylittlematrix{1 & 0 \\ 0 & 1 \\ 0 & 1}} (C3);
\draw[bend right=20,color=\inductioncolor,->] (C3) to node[right,pos=0.2,xshift={1ex}] {\mylittlematrix{2 & 0 & 0 \\ 0 & 1 & 1}}(G);

\draw[->,color=\powercolor] (eL) to node[above,font=\tiny] {1} (e);
\draw[->,bend left=10,color=\powercolor] (C2L) to node[below,pos=0.2] {\mylittlematrix{1 & 0 }}  (C2);
\draw[->,bend left=8,color=\powercolor] (C3L) to node[below,pos=0.8] {\mylittlematrix{1 & 0 & 0 \\ 0 & 0 & 1 \\ 0 & 1 & 0 }} (C3);
\draw[->,color=\powercolor] (GL) to node[above] {\mylittlematrix{1 & 0 & 0 \\ 0 & 1 & 0 }} (G);
\end{tikzpicture}
\end{center}
The target Green functor has been made constant on the $2$-torsion subgroup.
\end{ex}

\begin{ex} Consider again $G=\Sigma_3$, but with $m=3$.
We again use $\rho$ to denote a 3-cycle and $\tau$ to denote (any choice of) transposition. We continue to abuse notation by writing $C_3$ and $C_2$ for the subgroups generated by $\rho$ and $\tau$

Then by \cref{Jpdiv}, in order for the power operation to be a map of Mackey functors, we must quotient $Cl(\Sigma_3)$ by the values on the conjugacy classes of $\rho$ and $\rho^2$, and we must also quotient $Cl(C_3)$  by the (collapsed) conjugacy class of $\rho$. 
The resulting power operation of Green functors
\[ P_3/\ul{J} \, \colon \, \ul{Cl}_{\Sigma_3} \powerarr \MackUpGroupPower{Cl}{\Sigma_3}{3}/\ul{J}\]
is given by

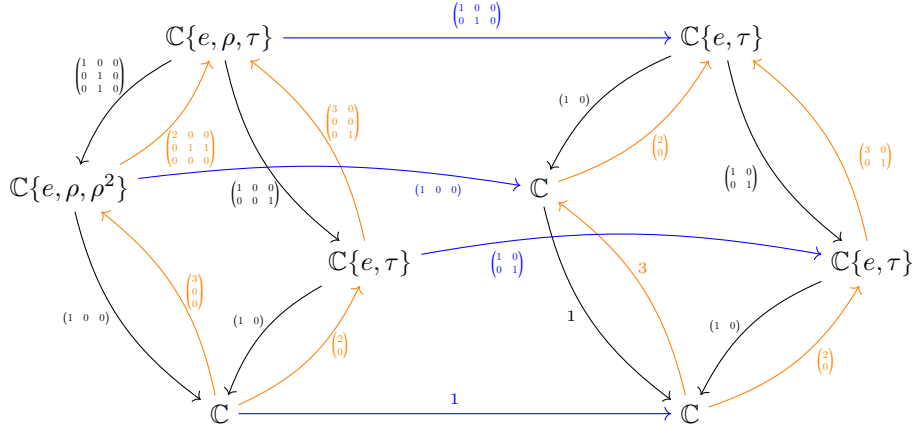
\begin{center}
\begin{tikzpicture}
\node (GL) at (0,5) {$\C\{e,\rho,\tau\}$};
\node (C3L) at (-2,3) {$\C\{e,\rho,\rho^2\}$};
\node (C2L) at (2,2) {$\C\{e,\tau\}$};
\node (eL) at (0,0) {$\C$};

\draw[bend right=20,->] (C2L) to node[pos=0.4,left] {\mylittlematrix{1 & 0}} (eL);
\draw[bend right=20,color=\inductioncolor,->] (eL) to node[pos=0.6,right] {\mylittlematrix{2 \\ 0}} (C2L);
\draw[bend right=20,->] (C3L) to node[left] {\mylittlematrix{1 & 0 & 0 }} (eL);
\draw[bend right=20,color=\inductioncolor,->] (eL) to node[right]{\mylittlematrix{3 \\ 0 \\ 0 }} (C3L);
\draw[bend right=20,->] (GL) to node[left,pos=0.7]{\mylittlematrix{1 & 0 & 0 \\ 0 & 0 & 1  }} (C2L);
\draw[bend right=20,color=\inductioncolor,->] (C2L) to node[right,pos=0.6,xshift={1pt}]{\mylittlematrix{3  & 0 \\ 0 & 0 \\ 0 & 1  }} (GL);
\draw[bend right=20,->] (GL) to node[left,pos=0.2,xshift={-1ex}]{\mylittlematrix{1 & 0 & 0 \\ 0 & 1 & 0 \\ 0 & 1 & 0 }} (C3L);
\draw[bend right=20,color=\inductioncolor,->] (C3L) to node[right,pos=0.2,xshift={1ex}] {\mylittlematrix{2 & 0 & 0 \\ 0 & 1 & 1 \\ 0 & 0 & 0}}(GL);

\node[ right = 60mm of GL] (G) at (0,5) {$\C\{e,\tau\}$};
\node[ right = 60mm of C3L] (C3) at (-2,3) {$\C$};
\node[ right = 60mm of C2L] (C2) at (2,2) {$\C\{e,\tau\}$};
\node[ right = 60mm of eL] (e) at (0,0) {$\C$};

\draw[bend right=20,->] (C3) to node[pos=0.5,left,font=\tiny,xshift={-1pt}] {1} (e);
\draw[bend right=20,color=\inductioncolor,->] (e) to node[pos=0.6,right,font=\tiny] {3} (C3);
\draw[bend right=20,->] (C2) to node[left] {\mylittlematrix{1 & 0  }} (e);
\draw[bend right=20,color=\inductioncolor,->] (e) to node[right,xshift={1ex}]{\mylittlematrix{2  \\ 0 }} (C2);
\draw[bend right=20,->] (G) to node[left,pos=0.5]{\mylittlematrix{1 & 0  }} (C3);
\draw[bend right=20,color=\inductioncolor,->] (C3) to node[right,pos=0.4,xshift={1ex}]{\mylittlematrix{2 \\ 0  }} (G);
\draw[bend right=20,->] (G) to node[left,pos=0.6,xshift={-1ex}]{\mylittlematrix{1 & 0 \\ 0 & 1}} (C2);
\draw[bend right=20,color=\inductioncolor,->] (C2) to node[right,pos=0.4,xshift={1ex}] {\mylittlematrix{3 & 0 \\ 0 & 1}}(G);

\draw[->,color=\powercolor] (eL) to node[above,font=\tiny] {1} (e);
\draw[->,bend left=10,color=\powercolor] (C2L) to node[below,pos=0.2] {\mylittlematrix{1 & 0 \\ 0 & 1 }}  (C2);
\draw[->,bend left=8,color=\powercolor] (C3L) to node[below,pos=0.8] {\mylittlematrix{1 & 0 & 0  }} (C3);
\draw[->,color=\powercolor] (GL) to node[above] {\mylittlematrix{1 & 0 & 0 \\ 0 & 1 & 0 }} (G);
\end{tikzpicture}
\end{center}
The target Mackey functor has been made constant on the $3$-torsion subgroup.
\end{ex}

\subsubsection{Height $2$}
\label{sec:heighttwo}
We now turn our attention to height $2$. Let $E$ be height $2$ Morava $E$-theory at the prime $p$, so that $E^0 \iso \mathbb{W}(k)[[u_1]]$, where $k$ is a perfect field of characteristic $p$. Hopkins, Kuhn, and Ravenel \cite{hkr} introduced a rational $E^0$-algebra $C_0$ and produced an isomorphism of $C_0$-algebras
\[ C_0 \otimes_{E^0} E^0(BG) \iso Cl_{2,p}(G,C_0),\]
where $Cl_{2,p}(G)=Cl_{2,p}(G,C_0)$ denotes the character ring of $C_0$-valued functions on conjugacy classes of commuting pairs of $p$-power order elements of $G$.
This extends to a Green functor $\ul{Cl_{2,p}}_G$, where the restriction and induction maps are similar to those described in \cref{sec:heightone}. See also \cite[Theorem~D]{hkr}.

A description of the power operation 
\[ P_m \colon Cl_{2,p}(G) \powerarr Cl_{2,p}(G\times \Sigma_m)\]
can be found in the introduction to \cite{StapleHbk}, which is a specialization of the main result of \cite{cottpo}.
One important point is that the power operation on class functions depends on a {\it choice} of ordered set of generators for sublattices of $\Z_p^{\times 2}$.

\begin{ex}
Let $G=\Z/2$, $p=2$, and $m=2$. We are to consider
\[ \begin{tikzcd}
Cl_{2,p}(\Z/2) \arrow[r,"P_2",color=\powercolor] \restrict{d} & 
Cl_{2,p}(\Z/2\times \Sigma_2) \restrict{d} \\
Cl_{2,p}(e) \arrow[r,"P_2",color=\powercolor] \induct{u} & 
Cl_{2,p}(\Sigma_2) . \induct{u}
\end{tikzcd} \]
A class function $f\in Cl_{2,p}(\Sigma_2)$ can be displayed as a table of values
\[ \begin{array}{|cccc}
(e,e) & (e,\sigma) & (\sigma,e) & (\sigma,\sigma) \\ \hline
a & b & c & d, 
\end{array}\]
where $a,b,c,d \in C_0$. For simplicity, we will assume that $a,b,c,d \in \Z \subset C_0$ so that certain ring-automorphisms of $C_0$ that appear in the general formula in \cite{StapleHbk} do not appear here. Following \cite{StapleHbk}, the power operation $Cl_{2,p}(e) \xrtarr{P_2} Cl_{2,p}(\Sigma_2)$ is given by
\[ a \mapsto 
 \begin{array}{|cccc}
(e,e) & (e,\sigma) & (\sigma,e) & (\sigma,\sigma) \\ \hline
a^2 & a & a & a. 
\end{array}\]
Collapsing the transfer from the subgroup $\Sigma_1\times \Sigma_1$ of $\Sigma_2$ will eliminate the (nonlinear) value at $(e,e)$.

We will similarly describe a class function $f\in Cl_{2,p}(\Z/2\times \Sigma_2)$ via a table of values
\[ 
\begin{array}{c|cccc}
 & (e,e) & (e,\sigma) & (\sigma,e) & (\sigma,\sigma) \\ \hline
(0,0) & a & b & c & d \\
(0,1) & e & f & g & h \\
(1,0) & i & j & k & l \\
(1,1) & m & n & o & p.
\end{array}
\]
Then a {\it choice} of power operation $P_2 \colon  Cl_{2,p}(\Z/2) \powerarr Cl_{2,p}(\Z/2\times \Sigma_2)$, compatible with the second power operation on height $2$ Morava $E$-theory at the prime $2$, is 
\[
 \begin{array}{|cccc}
(0,0) & (0,1) & (1,0) & (1,1) \\ \hline
a & b & c & d 
\end{array}
\mapsto 
\begin{array}{c|cccc}
 & (e,e) & (e,\sigma) & (\sigma,e) & (\sigma,\sigma) \\ \hline
(0,0) & a^2 & a & a & a \\
(0,1) & b^2 & a & \choicecolor{b} &  \choicecolor{b}\\
(1,0) & c^2 & \choicecolor{c} & a &  \choicecolor{b} \\
(1,1) & d^2 & \choicecolor{c} & \choicecolor{b} & a
\end{array},
\]
where all of the values that depend on a choice are displayed in color.
Collapsing the transfer from the subgroup $\Z/2\times \Sigma_1\times \Sigma_1$ of $\Z/2\times \Sigma_2$ will eliminate the first column of values, while collapsing the transfer from the diagonal subgroup of $\Z/2\times \Sigma_2$ will eliminate the $a$'s on the (slope negative one) diagonal.

Then the diagrams of restriction maps and power operations
\begin{center}
\begin{tikzpicture}[yscale=0.75]
\node (e) at (0,0) {$a$};
\node (C2) at (0,4) {$ \begin{array}{|cccc}
(0,0) & (0,1) & (1,0) & (1,1) \\ \hline
a & b & c & d 
\end{array}
$};
\node (S2) at (7,0) {$ \begin{array}{|cccc}
(e,\sigma) & (\sigma,e) & (\sigma,\sigma) \\ \hline
 a & a & a 
\end{array}$,};
\node (C2S2) at (7,4) {$
\begin{array}{c|cccc}
  & (e,\sigma) & (\sigma,e) & (\sigma,\sigma) \\ \hline
(0,0)  & a & a & a \\
(0,1) &  & \choicecolor{b} &  \choicecolor{b}\\
(1,0) & \choicecolor{c} &  &  \choicecolor{b} \\
(1,1)  & \choicecolor{c} & \choicecolor{b} & 
\end{array}
$};

\draw[|->] (C2) to  (e);
\draw[|->,color=\powercolor] (e) to  (S2);
\draw[|->,color=\powercolor] (C2) to  (C2S2);
\draw[|->] (C2S2) to  (S2);
\end{tikzpicture}
\end{center}
and the diagram of induction maps and power operations 
\begin{center}
\begin{tikzpicture}[yscale=0.75]
\node (e) at (0,0) {$a$};
\node (C2) at (0,4) {$ \begin{array}{|cccc}
(0,0) & (0,1) & (1,0) & (1,1) \\ \hline
2a & 0 & 0 & 0 
\end{array}
$};
\node (S2) at (7,0) {$ \begin{array}{|cccc}
(e,\sigma) & (\sigma,e) & (\sigma,\sigma) \\ \hline
 a & a & a 
\end{array}$,};
\node (C2S2) at (7,4) {$
\begin{array}{c|cccc}
  & (e,\sigma) & (\sigma,e) & (\sigma,\sigma) \\ \hline
(0,0)  & 2a & 2a & 2a \\
(0,1) &  & 0 & 0 \\
(1,0) & 0 &  & 0\\
(1,1)  & 0 & 0 & 
\end{array}
$};

\draw[|->,color=\inductioncolor] (e) to  (C2);
\draw[|->,color=\powercolor] (e) to  (S2);
\draw[|->,color=\powercolor] (C2) to  (C2S2);
\draw[|->,color=\inductioncolor] (S2) to  (C2S2);
\end{tikzpicture}
\end{center}
both commute.
\end{ex}

\bibliographystyle{amsalpha}
\bibliography{bibliography}

\end{document}